\documentclass[11pt]{article}
\usepackage{amsthm,amsmath,amssymb}
\usepackage{todo}
\usepackage{bm}
\usepackage{natbib}
\usepackage{bbold}
\usepackage{xcolor}
\usepackage{cancel}
\usepackage{stmaryrd}
\usepackage[english]{algorithm2e}
\usepackage{ulem}
\RequirePackage[colorlinks,citecolor=blue,urlcolor=blue]{hyperref}
\usepackage{graphicx}

\newcommand{\xddots}{%
	\raise 4pt \hbox {.}
	\mkern 6mu
	\raise 1pt \hbox {.}
	\mkern 6mu
	\raise -2pt \hbox {.}
}

\SetKwComment{Comment}{/* }{ */}

\def\bR{\mathbb R}

\def\bR{\boldsymbol{R}}

\newcommand{\E}{\operatorname{\mathbb{E}}}
\renewcommand{\P}{\operatorname{\mathbb{P}}}

\newtheorem{definition}{Definition}
\newtheorem{remark}{Remark}
\newtheorem{Theorem}{Theorem}
\newtheorem{theorem}{Theorem}
\newtheorem{lemma}[Theorem]{Lemma}

\newtheorem{prop}[Theorem]{Proposition}
\newtheorem{Assumption}[Theorem]{Assumption}

\def\argmax{\mathop{\rm arg\,max}}
\def\argmin{\mathop{\rm arg\,min}}

\def\det{\mathop{\rm det}}
\def\diag{\mathop{\rm diag}\nolimits}

\def\Hyp{{\rm Hyp}}
\def\id{\mathop{\rm id}\nolimits}

\def\rank{{\rm rank}}

\def\rank{{\rm rank}}
\def\op{{\mathrm {op}}}

\def\beq{\begin{equation}}
\def\eeq{\end{equation}}

\def\E{{\mathrm E}}

\def\P{{\mathrm P}}
\def\cpset{\mathcal{D}}

\def\Hyp{{\mathrm H}}
\def\one{{\mathbf 1}}

\def\tr{{\mathrm{tr}}}
\def\arg{{\mathrm{arg}}}
\def\diag{{\mathrm{diag}}}

\def\bR{{\mathbb R}}

\def\cN{{\mathcal N}}
\def\cF{{\mathcal F}}
\def\mX{{\mathcal X}}

\def\T{{ \mathrm{\scriptscriptstyle T} }}

\def\mX{{\mathfrak X}}

\def\mZ{{\mathfrak Z}}

\makeatletter
\g@addto@macro{\UrlBreaks}{\UrlOrds}
\makeatother

\usepackage{natbib}
\bibpunct[, ]{(}{)}{,}{a}{}{,}%
\makeatletter
\g@addto@macro{\UrlBreaks}{\UrlOrds}
\makeatother

\DeclareOldFontCommand{\bf}{\normalfont\bfseries}{\mathbf}

\begin{document}
	%
	%
	%
	%
	
	\title{Change point detection in low-rank VAR processes}
	\author{Farida Enikeeva, Olga Klopp and Mathilde Rousselot}
	\maketitle
	\begin{abstract}
		Vector autoregressive (VAR) models are widely used in multivariate time series analysis for describing the short-time dynamics of the data. The reduced-rank VAR models are of particular interest when dealing with  high-dimensional and highly correlated time series.  Many results for these models are based on the stationarity assumption that does not hold  in several applications when the data exhibits structural breaks. We consider a low-rank piecewise stationary VAR model with possible changes in the transition matrix of the observed process. We develop a new test of presence of a change-point in the transition matrix and show its minimax optimality with respect to the dimension and the sample size. Our two-step change-point detection strategy is based on the construction of estimators for the transition matrices and using them in  a penalized version of the likelihood ratio test statistic. The effectiveness of the proposed procedure is illustrated on synthetic data.
	\end{abstract}
	
\section{Introduction}
Vector autoregression (VAR) is a classical model of multivariate time series analysis that  has been successfully used to model data in epidemiology 
\citep{khan_modelling_2020}, economics and finance \citep{fan_sparse_2010}, medical research 
\citep{wild_graphical_2010}, econometrics \citep{StockWatson2001} and neuroscience 
\citep{gorrostieta_hierarchical_2013}. During the past two decades, a considerable  interest has been focused on the problems of high-dimensional inference. High-dimensional multivariate time series have a large amount of variables  but only a limited number of time steps. In such a situation the
VAR model is ill-posed: it suffers from the over-parametrization issue as the number of parameters in the coefficient matrix is comparable to or much larger than the number of time series observations. To address this issue  different structural assumptions on the transition matrix  have been proposed.  The idea  is that one expects the system to be controlled primarily by a low-dimensional subset of variables.  For example,  in finance, the data can have  a huge ambient dimension which includes financial instruments such as   stocks, bonds, and etc.  These financial instruments may be combined into a much smaller subset of macro-variables that actually govern the market.  

A reduced-rank VAR model for multivariate time series analysis was first introduced by \cite{velu_reduced_1986} (see also \citep{reinsel_multivariate_0000, luetkepohl2005} for detailed background). In  \citep{velu_reduced_1986} the authors show the consistency of the  least squares  estimator  of the VAR transition matrix modeled by a product of two  rectangular matrices. \cite{AhnReinsel1988}  generalize these results to nested low-rank autoregressive models.  \cite{negahban2011}  propose a least-squares nuclear norm penalized estimator for the transition matrix and show its consistency in the Frobenius norm. \cite{AlquierBDG20} consider the problem of prediction for low-rank VAR modes. Their method is based on rank-penalized least-squares estimator. More recently \cite{WangTsayVARopt2022}  proposed a method of the transition matrix estimation using the constrained Yule-Walker equations  and showed its  optimality under the $\beta$-mixing dependency condition. 

Another popular assumption on the matrix structure is the entry-wise sparsity. Under the entry-wise sparsity assumptions, \cite{SheHeLiWu2015} estimate jointly the transition matrix and the precision matrix of the process using joint regularization; \cite{BasuMichailidis2015} use an $\ell_1$-penalized log-likelihood for estimating the transition matrix and \cite{MelnykBanerjee2016}  use  a penalized log-likelihood estimator with $\ell_1$-norm penalization and  the group lasso penalty.  Finally, \cite{BasuLiMichailidis2019} consider the VAR model under both sparsity and low rank assumptions.  They assume that the transition matrix can be written as the sum of a sparse matrix and a low rank matrix. The proposed penalized least-squares estimator uses  the nuclear norm penalty combined with the $\ell_1$-penalization. 

Many of these results are based on the assumption of stationarity. However, in applications the data might exhibit structural breaks with several discontinuity points in the distribution and the stationarity assumption  does not hold for the entire data set. In such situations a natural approach is to assume a piece-wise stationary model and apply existing methods to the intervals of stationarity. Then, the crucial question is that of the identification of change-points.

Detection of structural changes  in multivariate time series is one of the major problems arising in applications. In neuroscience the breaks in the sequence of the electroencephalogram (EEG) signals correspond to changes in brain activity \citep{MichelMurray2011}. In financial analysis the volatility of the market indexes might change at certain time points due to some external event  \citep{AueHormann2009}. Efficient detection of such breaks heavily relies on the underlying mechanism of the data temporal evolution. 
The goal of this paper is to propose a procedure that allows to detect the presence of changes in the matrix parameter of a low-rank  VAR process.
Note that  classical methods of multidimensional change-point detection (see, for example, \citep{basseville_nikiforov_1993}) are not applicable in the high-dimensional setting since the model dimension can be much larger than the sample size.

\subsection{Low-rank VAR Model}
We say that the process $\{X_t\}_{t\in \mathbb Z}$ is a $p$-dimensional  \emph{vector autoregressive} (VAR) process $\mathrm{VAR}_p(\Theta,\Sigma_Z)$  with Gaussian innovations, if 
\begin{equation} \label{eq:var_def}
X_{t+1}=\Theta \,X_{t}+Z_{t+1}\quad \forall  t\in \mathbb Z,
\end{equation}
where $\Theta\in \mathbb R^{p\times p}$ is the transition matrix and $\{Z_t\}_{t\in\mathbb Z}$ is a $p$-dimensional centered  Gaussian independent    noise  with the covariance matrix $\Sigma_Z\in\mathbb R^{p\times p}$, $Z_t\sim \cN_p(0,\Sigma_Z)$. We assume that the operator norm of $\Theta$ satisfies $\|\Theta\|_{\mathrm{op}}=\gamma<1$. Under this condition $\{X_t\}_{t\in \mathbb Z}$ is a stationary causal process (cf. \citep{luetkepohl2005} for more details on stationary VAR processes) and the covariance matrix $\Sigma$ of $X_t$ satisfies the Lyapunov equation
\begin{equation}\label{eq:lyap}
\Sigma=\Theta\Sigma \Theta^{\T}+\Sigma_{Z}.
\end{equation}
If $\|\Theta\|_{\mathrm{op}}<1$, then~\eqref{eq:lyap} has a unique positive-definite solution $\Sigma$.

In a more general setting  the transition matrix may change over time. In our problem  we observe a trajectory $\mX=(X_0,\dots,X_T)$ of a piece-wise stationary centered $p$-dimensional Gaussian $\mathrm{VAR}_p(\Theta_1,\Theta_2,\Sigma_Z)$ process,
\begin{equation} \label{eq:var}
X_{t+1}=\Theta^t \,X_{t}+Z_{t+1}\quad  t=0,\dots,T-1,
\end{equation}
where $Z_t\sim \cN_p(0,\Sigma_Z)$ are i.i.d. and  
$$
\Theta^t=\Theta_1\one_{\{0\le t\le \tau\}}+ \Theta_2\one_{\{\tau+1 \le t\le T\}}.
$$
Here $\tau$ stands for the change-point in transition matrix. We assume that the transition matrices are of the rank at most $R$, $\mathrm{rank} (\Theta_j)\le R$ $\forall j=1,2$.
In general, $\Sigma_{Z}$ can vary over segments, but we consider it to be fixed to avoid additional technicalities. 
In the following we assume that the system matrices $\Theta_{j}$ are stable:
\begin{Assumption}\label{ass:stability}
	\label{ass: stability}For $j=1,2$ there exists $0<\gamma_j<1$ such that 
	$\| \Theta_{j}\|_{\op}\leq\gamma_j$.
\end{Assumption}
Note that under this assumption the matrices $\Theta_j$ satisfy the Lyapunov equation~\eqref{eq:lyap}, its unique solution gives the corresponding covariance matrix $\Sigma_j$ of the VAR process.
\subsection{Related work}
The existing literature mainly considers the related question of change-point localization under entry-wise sparsity or entry-wise sparsity  plus low-rank assumptions on the transition matrix. For example,  under sparsity assumptions on the transition matrices  \cite{SafikhaniShojaie2022} use a fused Lasso approach to estimate the breakpoints as well as the process parameters. They obtain the localization error bound  under the assumption  that the minimum distance between two change-points  is a sufficiently large constant independent of the number of observations. In the same setting, \cite{WangYu2019} obtain a stronger result   allowing decreasing distance between the change-points. Their estimator  is based on the combination of the  Lasso and group Lasso methods.  \cite{bai_multiple_2020} assume that the  transition matrices can be decomposed into a constant low-rank component and a sparse time evolving component. They develop a strategy for identification of change-points in the sparse component and provide probabilistic guarantees for the accuracy of their identification.  More recently, \cite{BaiSafikahniMichailidis2022} considered the  "low rank plus sparse" VAR model where both matrix components may change with time under the assumption that  the maximum absolute value of the entries of the low-rank component is bounded by $\Bigl(\dfrac{\log(pT)}{T}\Bigr)^{1/2}$, which goes to zero when the number of observations $T$ is growing. They propose a method of multiple change-point estimation based on the plug-in estimators obtained in  \citep{BasuLiMichailidis2019} and provide its theoretical guarantees.

\subsection{Our contributions}
We consider the problem of testing the presence of a change in the low-rank transition matrix. Our testing procedure is based on the plug-in test statistic with the estimated low-rank matrices before and after an eventual change.  We prove that our test allows for reliable change-point detection when the squared Frobenius norm of the change is larger then  $\dfrac{Rp}{Tq^{2}(t/T)}$  (up to a logarithmic factor, for the precise statement see Theorem \ref{th:UpperBound}) where $q(t)=\sqrt{t(1-t)}$ for $t\in[0,1]$  controls the impact of the change-point location on the rate.
An important point is that this result does not require any condition on the minimum spacing between the change-point and the boundaries of the interval of observations. We also show that our testing procedure is minimax-rate optimal both in terms of the dimension and the sample size  (Theorem~\ref{th:LowerBoundTesting}).
As a bi-product of our analysis we provide a new result on the consistency of the nuclear norm-penalized estimator of the transition matrix in operator norm (see Proposition~\ref{thm:operator_norm}).

\subsection{Notation}\label{subsec:notation}
We start with basic notation used in this paper. For any matrix $M$, we denote by $M_{ij}$ its entry in the $i$th row and $j$th column and by $M_{i\cdot}$ its $i$th row.  The notation $\diag(M)$ stands for the diagonal of a square matrix $M$ and $M^{\T}$ for the transpose of $M$. The column vector of dimension $n$ with unit entries is denoted by $\one_n=(1,\dots,1)^\T$ and the column vector of dimension $n$ with zero entries is denoted by $\boldsymbol{0}_n=(0,\dots,0)^\T$. The identity matrix of dimension $n$ is denoted by $\mathrm{id}_n$. For a set $A$, we denote by $\one_{A}$ its indicator function.

For any matrix $M \in \mathbb{R}^{d_1\times  d_2}$, $\| M \|_2$ is its Frobenius norm, $\|M\|_{\mathrm{op}}$ is its operator norm (its largest singular value).
We denote by $\sigma_j(M)$ the $j$th singular value and by $\sigma_{\max}(M)$ and $\sigma_{\min}(M)$ the largest and the smallest non-zero singular value of $M$. Assuming that matrix $M$ has rank $R$ we consider its ordered singular values $\sigma_1(M) \geq \sigma_2(M) \geq \dots \geq \sigma_R(M)>0$ and  we denote the condition number  of $M$ by $\kappa(M) = {\sigma_1(M)}/{\sigma_R(M)}$.

For  $M$, a $d_1\times d_2$ matrix, let $u_j(M)$ and $v_j(M)$ be respectively the left and right orthonormal singular vectors of $M$, $S_1(M)$ be the linear span of $\{u_j(M)\}$, $S_2(M)$ be the linear span of $\{v_j(M)\}$. 
We denote by $S^\bot$  the orthogonal complement of $S$.
For $B$, a $d_1\times d_2$ matrix,  let $\mathbf {Pr}_M(B)=B- \mathcal P_{S_1^{\bot}(M)}B\mathcal P_{S_2^{\bot}(M)}$ and $\mathbf {Pr}^{\bot}_M(B)= \mathcal P_{S_1^{\bot}(M)}B\mathcal P_{S_2^{\bot}(M)}$, where $\mathcal P_S$ is the orthogonal projector on the linear vector subspace $S$. 

We write $X\lesssim Y$ and $X\gtrsim Y$ if $X\le CY$ and, respectively, $X\ge CY$ for some absolute constant $C>0$.
%

We denote by $\mathcal{M}_p(R,\gamma)$ the set of all $p\times p$ real matrices of rank at most $R$ with the operator norm bounded by $\gamma$:
\begin{equation*}
\mathcal{M}_p(R,\gamma)=\Bigl\{M\in \mathbb R^{p\times p}:\rank (M)\leq R \quad\mbox{and}\quad \| M\|_{\mathrm{op}}\leq\gamma\Bigr\}
\end{equation*}

For any $t\in \{1,\dots,T-1\}$, we    introduce  the following   random matrices : 
\begin{align*}
\mX_{\ge t}:= \left(X_{t}, \dots, X_{T-1} \right)\in \mathbb R^{p\times(T-t)},&\quad \mX_{<t}:= \left(X_{0}, \dots, X_{t-1} \right)\in \mathbb R^{p\times t}\\
\mathfrak{Y}_{>t}:= \left(X_{t+1}, \dots, X_{T} \right)\in \mathbb R^{p\times(T-t)},&\quad \mathfrak{Y}_{\le t}:= \left(X_{1}, \dots, X_{t} \right)\in \mathbb R^{p\times t}\\
\mathfrak{Z}_{\ge t}:= \left(Z_{t}, \dots, Z_{T-1} \right)\in \mathbb R^{p\times(T-t)}&, \quad\mathfrak{Z}_{<t}:= \left(Z_{0}, \dots, Z_{t-1} \right)\in \mathbb R^{p\times t}.
\end{align*}
Here $\mX_{<t}$ ($\mX_{\ge t}$) contains our observations before  (after) a given time point $t$ and $\mathfrak{Z}_{\cdot}$ contains the innovation noise. $\mathfrak{Y}_{\cdot}$ is obtained from $\mX_{\cdot}$ by shifting our observations by one time step.

\section{Change-point detection problem}\label{single_change_point}
We will consider the problem of detection of a single change-point in the VAR model~\eqref{eq:var} with the transition matrix
$\Theta^t$ that can change at some unknown point $\tau\in \cpset_T\subseteq \{1,\dots,T-1\}$,
$$
\Theta^t=\Theta_1\one_{\{0\le t\le \tau\}}+ \Theta_2\one_{\{\tau+1 \le t\le T\}}.
$$
The difficulty of assessing the existence of a change-point can be quantified  by what  is called \textit{energy} of the change point. It is defined as the product of the Frobenius norm of the jump in transition matrix and the function $q(t)=\sqrt{t(1-t)}$ for $t\in[0,1]$. The function $q(t)$  quantifies the impact of  change-point location to the difficulty of detecting the change. Thus, we write the detection problem  as the problem of testing whether the  jump energy $$\mathcal E(t,\Theta_1,\Theta_2)=q(t/T)\|\Theta_1-\Theta_2\|_{2}$$ is zero or not. To formulate the hypothesis testing problem, we define the set of all pairs of  matrices with the operator norm bounded by $\gamma\in(0,1)$ before ($\Theta_1$) and after ($\Theta_2$)  the change at the location $t$ such that  the jump energy is at least $r>0$,
\begin{equation}\label{def_GammaSet}
\mathcal V_{p,t}(r)=\Bigl\{\bigl(\Theta_1,\Theta_2\bigr)\in \mathcal M_p^{\otimes2}(R,\gamma):\ \mathcal E(t,\Theta_1,\Theta_2)\ge r  \Bigr\}.
\end{equation}
Let $\mathcal V_{p,0}$ denote the set without a jump:
\begin{equation*}
\mathcal V_{p,0}=\Bigl\{\bigl(\Theta_1,\Theta_2\bigr)\in \mathcal M_p^{\otimes2}(R,\gamma):\Theta_1=\Theta_2 \Bigr\}
\end{equation*}
We will test the null hypothesis of no-change
\begin{equation}\label{eq:h_0}
\Hyp_0:\ (\Theta_1,\Theta_2)\in \mathcal V_{p,0}
\end{equation}
against the alternative hypothesis of a change in the transition matrix: 
\begin{equation}\label{eq:h_1}
\Hyp_1:\ (\Theta_1,\Theta_2)\in \mathcal V_{p,\tau}(\mathcal R_{p,\cpset_T})\ \mbox{for some}\ \tau \in \cpset_T,
\end{equation}
where $\mathcal R_{p,\cpset_T}>0$ is the minimal  amount of energy  that guarantees the change-point detection and $\cpset_T\subseteq\{1,\dots, T-1\}$.

We construct a change-point detection procedure  based on the penalized least-squares minimization approach. Our procedure has two steps. In the first step, for each $t\in \cpset_T$, we  compute  estimators of the transition matrix at each of the two intervals $[0,t]$ and $[t,T]$. Once equipped with such estimators we use an information criteria to build the test statistic and the change-point estimator. Our estimators of transition matrices are based on the nuclear norm minimization criteria which is a convex relaxation of the rank-constrained minimization problem. Note that the estimation of the transition matrix is easier if $t$ (respectively $T-t$) is large comparing to the matrix dimension $p$. For difficult cases of $t<p$ (respectively $T-t<p$) we use a slightly different penalization which allows us to cope with the lack of observations.

\subsection{Transition matrix estimation}

We start by giving a general construction of an estimator of the transition matrix $\Theta$ from the observations $\mX_{[t_1,t_2]}=(X_{t_1},\dots,X_{t_2})\in \bR^{p\times(t_2-t_1+1)}$ of a $\mathrm{VAR}_p(\Theta,\Sigma_Z)$ process observed at the consecutive time moments $t_1,t_1+1,\dots,t_2$. Later on, we will use the estimators of the transition matrices obtained withing the intervals $[0,t]$ and $[t,T]$ in order to construct our change-point detection procedure. 

For $t_2-t_1>p$, we set
\begin{equation}\label{eq:sdp1}
\varphi(\mX_{[t_1,t_2]},M)=\frac 1{t_2-t_1}\sum_{i=t_1}^{t_2-1} \|X_{i+1}-M X_i\|_2^2 + \lambda\|M \|_*
\end{equation}
and 
for $t_2-t_1\le p$,
\begin{equation}\label{eq:sdp2}
\varphi(\mX_{[t_1,t_2]},M)=\frac 1{t_2-t_1}\sum_{i=t_1}^{t_2-1} \|X_{i+1}-M X_i\|_2^2 + \lambda\|M \mX_{[t_1,t_2]}\|_*.
\end{equation}
Here $\|\cdot\|_*$ stands for the nuclear norm and  the regularization parameter$\lambda_{[t_1,t_2]}$ is given by 
\begin{equation}\label{eq:lambda}
\lambda=6c_1^* \sqrt{\| \Sigma_{Z}\|_{\mathrm{op}}}  \frac {\sqrt p} {t_2-t_1}\one_{\{t_2-t_1\leq p\}}+ 2c_2^* \frac{\| \Sigma\|_{\mathrm{op}}}{1-\gamma} \sqrt{\frac {p}{t_2-t_1}}\one_{\{t_2-t_1> p\}},
\end{equation}
where $c_1^*$ and $c_2^*$ are absolute constants  provided in Lemmas~\ref{lem:bound_stoch_term1} and~\ref{lem:bound_stoch_term2}.

We consider the following estimator of $\Theta$ within the interval $[t_1,t_2]$:
\begin{equation}\label{eq:Theta_hat}
\widehat \Theta\in\argmin_{M\in\mathbb{R}^{p\times p}} \varphi(\mX_{[t_1,t_2]},M).
\end{equation}

Estimator \eqref{eq:sdp1} was first introduced in \citep{negahban2011}.
In the case of large number of observations $t_2-t_1\ge p$, with no change in the matrix of parameter, ~\cite{negahban2011} prove the  estimator consistency in the Frobenius norm (see Proposition~\ref{cor:bound_risk_estimation2}). Note that estimator provided by~\eqref{eq:Theta_hat} is also consistent in the operator norm, see Proposition~\ref{thm:operator_norm} in the Appendix.

\subsection{Testing procedure}
For any $t\in \cpset_T$,  we  will estimate the transition matrix on the intervals  before and after the time $t$, $[0,t)$ and $[t,T)$ using, respectively, the observations $\mX_{\le t}$ and $\mX_{\ge t}$. These estimators are obtained as the solutions of the SDP~\eqref{eq:sdp1} or \eqref{eq:sdp2} depending on the proximity of the point $t$ to the endpoints of the observation interval $[0,T]$.
Denote 
$$
\varphi_1^t(\mX,M)= \varphi(\mX_{[0,t]},M)\quad \mbox{and}\quad \varphi_2^{t}(\mX,M)= \varphi(\mX_{[t,T]},M).
$$
Then the optimization programs for estimation of $\Theta_1$ and $\Theta_2$ can be written as 
\begin{equation}
\widehat{\Theta}_1^t\in\argmin_{M\in\mathbb{R}^{p\times p}} \varphi_1^t(\mX,M) \label{eq:Theta_hat1}
\end{equation}
and 
\begin{equation}
\widehat{\Theta}_2^{t}\in\argmin_{M\in\mathbb{R}^{p\times p}} \varphi_2^{t}(\mX,M)\label{eq:Theta_hat2}
\end{equation}
where 
$$
\varphi_1^t(\mX,M)= \frac1t \| \mathfrak{Y}_{\le t}-M\mX_{<t}\| _{2}^{2}+\lambda_1^t\Bigl(\| M\mX_{<t}\| _{*}\one_{\{t\le  p\}}+ \| M\|_* \one_{\{t> p\}}\Bigr)
$$
and 
$$
\varphi_2^t(\mX,M)= \frac1{T-t}\| \mathfrak{Y}_{> t}-M\mX_{\ge t}\| _{2}^{2}+\lambda_2^t\Bigl(\| M\mX_{\ge t}\| _{*}\one_{\{T-t\le p\}}+ \| M\|_* \one_{\{T-t> p\}}\Bigr)
$$
with the penalties 
\begin{align}
\lambda_1^t&:= 6c_1^* \sqrt{\| \Sigma_{Z}\|_{\mathrm{op}}}  \frac {\sqrt p} {t}\one_{\{t\leq p\}}+ 2c_2^* \frac{\| \Sigma_1\|_{\mathrm{op}}}{1-\gamma_1} \sqrt{\frac {p}{t}}\one_{\{t> p\}}\label{choice_lambda1}\\
\lambda_2^t&:= 6c_1^* \sqrt{\| \Sigma_{Z}\|_{\mathrm{op}}}  \frac {\sqrt p} {T-t}\one_{\{T-t\leq p\}}+ 2c_2^* \frac{\| \Sigma_2\|_{\mathrm{op}}}{1-\gamma_2} \sqrt{\frac {p}{T-t}}\one_{\{T-t> p\}}.\label{choice_lambda2}
\end{align}
Finally we define the following test statistic:
\begin{equation}\label{eq:TestStat}
\mathcal{F}(t)= \left (\varphi^t_1(\mX,\widehat \Theta_2^{t})
- \varphi^t_1(\mX,\widehat \Theta_1^t)\right )\one_{\{t< T/2\}} 
+  \left (\varphi_2^{t}(\mX,\widehat \Theta_1^t)
- \varphi_2^{t}(\mX,\widehat \Theta_2^{t})\right )\one_{\{t\geq T/2\}}.
\end{equation}
Let $\mathcal T$ be a subset of $\cpset_T$ that approximates the set of possible change-point locations.  In case of testing against a simple alternative of change at a given point $\tau$, we take $\mathcal T=\cpset_T=\{\tau\}$. In case of a composite alternative, we can choose $ \mathcal T=\cpset_T$  or use an appropriate grid on $\cpset_T$.  For  a given significance level $\alpha\in (0,1)$ we introduce the test 
\begin{equation}\label{eq:test}
\psi_{\alpha}(\mX)=\one\Biggl\{ \max_{t\in \mathcal T} \frac{\mathcal F(t)}{H_{\alpha,t}}>1 \Biggr\}
\end{equation}
where $ H_{\alpha,t}$ is the threshold defined as follows:
\begin{equation}\label{def:test_threshold}
H_{\alpha,t}=\begin{cases}
C^*  \frac{Rp}{Tq^{2}(t/T)}  \Bigl(1+\frac2{1-\gamma}\frac{p+\log(8|\mathcal T|/\alpha)}{T}\Bigr),& t\wedge (T-t)\le p\\
C^*  \frac{Rp}{Tq^{2}(t/T)}  \Biggl(1+\frac2{1-\gamma}\left(1+\frac{\log(8|\mathcal T|/\alpha)} {T}\right)\Biggr),& t\wedge (T-t)> p\\
\end{cases}
\end{equation}
with $C^*=C\max \left( \| \Sigma_Z\|_{\mathrm{op}},\| \Sigma\|_{\mathrm{op}} \dfrac{ \kappa^{2}(\Sigma)}{(1-\gamma)^2} \right)$ for an absolute constant $C$. Here $\Sigma$ is the covariance matrix of the VAR process under the null.
We can estimate its operator norm from the Lyapunov equation~\eqref{eq:lyap} as $\|\Sigma\|_{\mathrm{op}}\le \|\Sigma_Z\|_{\mathrm{op}}/(1-\gamma^2)$ and its condition number as $\kappa(\Sigma)\le \kappa(\Sigma_Z)\frac{1+\gamma^2}{1-\gamma^2}$ using Lemma~\ref{lem:stab_lyapunov}.

\begin{theorem}\label{th:UpperBound}
	Let $\alpha,\beta  \in (0,1)$ be given significance levels such that $p>C \log\left (\dfrac{|\mathcal T|}{\alpha\beta}\right )$ for some absolute constant $C>0$. Let $p\le T/2$ and $(\Theta_1,\Theta_2)\in \mathcal M_p^{\otimes2}(R,\gamma) $ be a  couple of transition matrices with the change-point energy satisfying
	\[
	\mathcal E^{2}(\tau,\Theta_1,\Theta_2)\geq \Xi  \frac{ Rp}T \left(1+\frac{1}T \log\left (\dfrac{64|\mathcal T|}{\alpha\beta}\right )\right),
	\] 
	where the constant $\Xi=\Xi (\Sigma_1,\Sigma_2,\Sigma,\gamma)$ is defined in~Lemma~\ref{lem:typeIIerror}.	
	Then, the $\alpha$-level test $\psi_{\alpha}$ defined in \eqref{eq:test} has the type II error smaller than $\beta$. 
\end{theorem}

\begin{proof}
	The proof of Theorem \ref{th:UpperBound} is based on Lemmas~\ref{lem:typeIerror} and~\ref{lem:typeIIerror}. Lemma~\ref{lem:typeIerror} implies that $\alpha(\psi_{\alpha})\le\alpha$. By Lemma~\ref{lem:typeIIerror}, the type II error smaller than $\beta$ is guaranteed if $R_{p, \cpset_T}$ satisfies~(\ref{eq:separation_rate})
	which completes the proof of the theorem.
\end{proof}
\begin{remark}
	Note that the detection rate is of the order $\sqrt{Rp/T}$ in the case of testing at a given point $\tau$ and when $T>\log\left (\dfrac{T}{\alpha\beta}\right )$ in the case of undefined change-point location with $\mathcal T=\cpset_T=\{1,\dots,T-1\}$. For smaller $T$,
	we can replace $\log T$ by $\log\log T $  by taking the dyadic grid   (see, for example, \citep{LiuGaoSamworth2021}) defined in the following way:  $\mathcal T=\mathcal T^L \cup \mathcal T^R$, where 
	\begin{equation}\label{eq:dyadic_grid}
	\mathcal T^L=\Bigl\{2^{k},\  k=0,\dots, \lfloor \log_2 (T/2)\rfloor  \},\quad \mathcal T^R= \Bigl\{T-2^{k},\  k=0,\dots, \lfloor \log_2 (T/2)\rfloor  \Bigr\}.
	\end{equation}
\end{remark}

\begin{remark}
	The test defined in~\eqref{eq:test} detects the change-points  $\tau$ located in any time point. If we know that the change-point belongs to the interval $(p,T-p)$, where the consistent estimation of the transition matrices is possible, then, we can define a  penalized likelihood ratio statistic defined as 
	\begin{equation}\label{eq:Gstat}
	\mathcal G(t)=\frac t{T} \left (\varphi^t_1(\mX,\widehat \Theta_2^{t})
	- \varphi^t_1(\mX,\widehat \Theta_1^t)\right )+ \frac{T-t}T   \left (\varphi_2^{t}(\mX,\widehat \Theta_1^t)
	- \varphi_2^{t}(\mX,\widehat \Theta_2^{t})\right ).
	\end{equation}
	We can show, using exactly the same technique as in Lemmas~\ref{lem:typeIerror} and~\ref{lem:typeIIerror} that the procedure will detect the change-point with the same detection rate  $Rp/T$ up to the logarithmic term and with a slightly different  constant. 
\end{remark}

We have the following optimality result on the minimal detectable change-point energy. 
\begin{theorem}\label{th:LowerBoundTesting}
	Let $\alpha\in (0,1)$ be given significance level. Assume that the change-point location satisfies the following condition: $$\min(\tau/T,1-\tau/T)\ge h_*>0.$$
	Let the change-point energy $\mathcal E(\tau,\Delta\Theta)=q(\tau/T) \|\Delta\Theta\|_2$ satisfy
	\begin{enumerate}
		\item[(a)]
		$\lim\limits_{p,T\to\infty}  \mathcal E(\tau,\Theta_1,\Theta_2) \Bigl(\frac{T}{p}\Bigr)^{1/2} =0$  if $p\lesssim \sqrt T$
		\item[(b)] $
		\lim\limits_{p,T\to\infty}  \mathcal E(\tau,\Theta_1,\Theta_2) T^{1/4}=0
		$ if  $p\gtrsim \sqrt T$.
	\end{enumerate}
	Then, the type II error  of any $\alpha$-level test $\psi_\alpha$ satisfies $\liminf\limits_{p,T\to\infty} \beta(\psi) \ge 1-\alpha$ and the lower bound on the testing rate is
	$\mathcal R_{p,\tau}=o\Bigl(\sqrt{\frac pT}\wedge \frac{1}{T^{1/4}}\Bigr)$.
\end{theorem}
For the matrices of dimension $p\lesssim \sqrt T$, Theorems~\ref{th:UpperBound} and~\ref{th:LowerBoundTesting} imply that the minimal detectable energy satisfies the condition
$$
\sqrt{\frac pT}\lesssim \mathcal E(\tau,\Theta_1,\Theta_2)\lesssim \sqrt{\frac{Rp}{T}}
$$
and our testing procedure is minimax rate optimal up to a possible loss of order $\sqrt{R}$ and a $\log$ term.

\begin{remark}
	We can construct a change-point estimator based on he statistic $\mathcal G(t)$ defined in~\eqref{eq:Gstat}: $$\widehat\tau=\argmax\limits_{t\in\cpset_T} \mathcal G(t).$$ The consistency of the change-point localization can be proven under the condition that $\mathcal E(\tau,\Theta_1,\Theta_2)\gtrsim \sqrt{Rp(\log T)/T}$.  Note that using the same technique as in Theorem~\ref{th:LowerBoundTesting}, it can be shown that the consistent estimation is impossible if  $\mathcal E(\tau,\Theta_1,\Theta_2) \lesssim \sqrt{p/T}$ for $p\lesssim \sqrt T$. 
\end{remark}

\section{Simulations}

We suppose that an eventual change-point is located within the interval $[Th,T-Th]$, where $p/T<h<1/2$ is given. The process is stationary within the intervals $[1,Th]$ and $[T-Th,T]$  that will be used for calibration of the quantiles and estimation of the transition matrices.  We have implemented the testing procedure based on the test statistic $\mathcal G(t)$ defined in~\eqref{eq:Gstat}
$$
\mathcal G(t)=\frac{t}T \Bigl(\varphi_1^t( {\mathfrak X},\widehat\Theta_2^t)-\varphi_1^t( {\mathfrak X},\widehat\Theta_1^t)\Bigr) + \frac{T-t}T\Bigl(\varphi_2^t({\mathfrak X},\widehat\Theta_1^t)-\varphi_2^t({\mathfrak X},\widehat\Theta_2^t)\Bigr)
$$
with the corresponding test
$$
\psi_{\alpha}(\mathfrak X)=\one\left\{\max_{t\in \mathcal T} \frac{\mathcal G(t)}{q_{\alpha,t}}>1\right\},
$$
where $\mathcal T$ is a grid approximating the set of possible change-points  $\cpset_T\subseteq[Th, T-Th]$.
\begin{figure}[htbp!]
	\begin{minipage}{0.48\linewidth}	
		\centering
		\includegraphics[width=\linewidth]{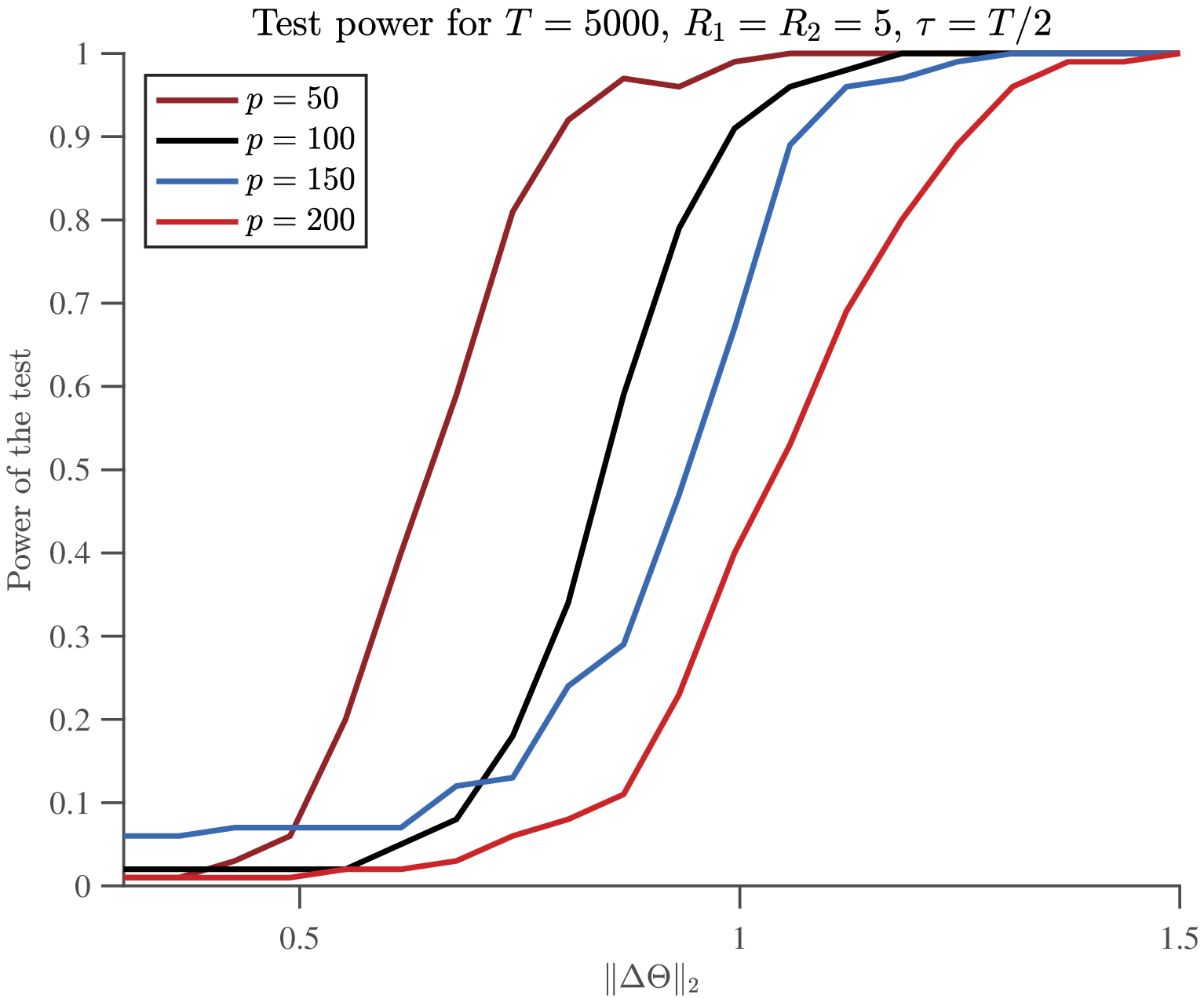}
		\caption{Test power depending on the dimension of transition matrices}
		\label{fig:power-p}
	\end{minipage}
	\hfill
	\begin{minipage}{0.48 \linewidth}
		\centering
		\includegraphics[width=\linewidth]{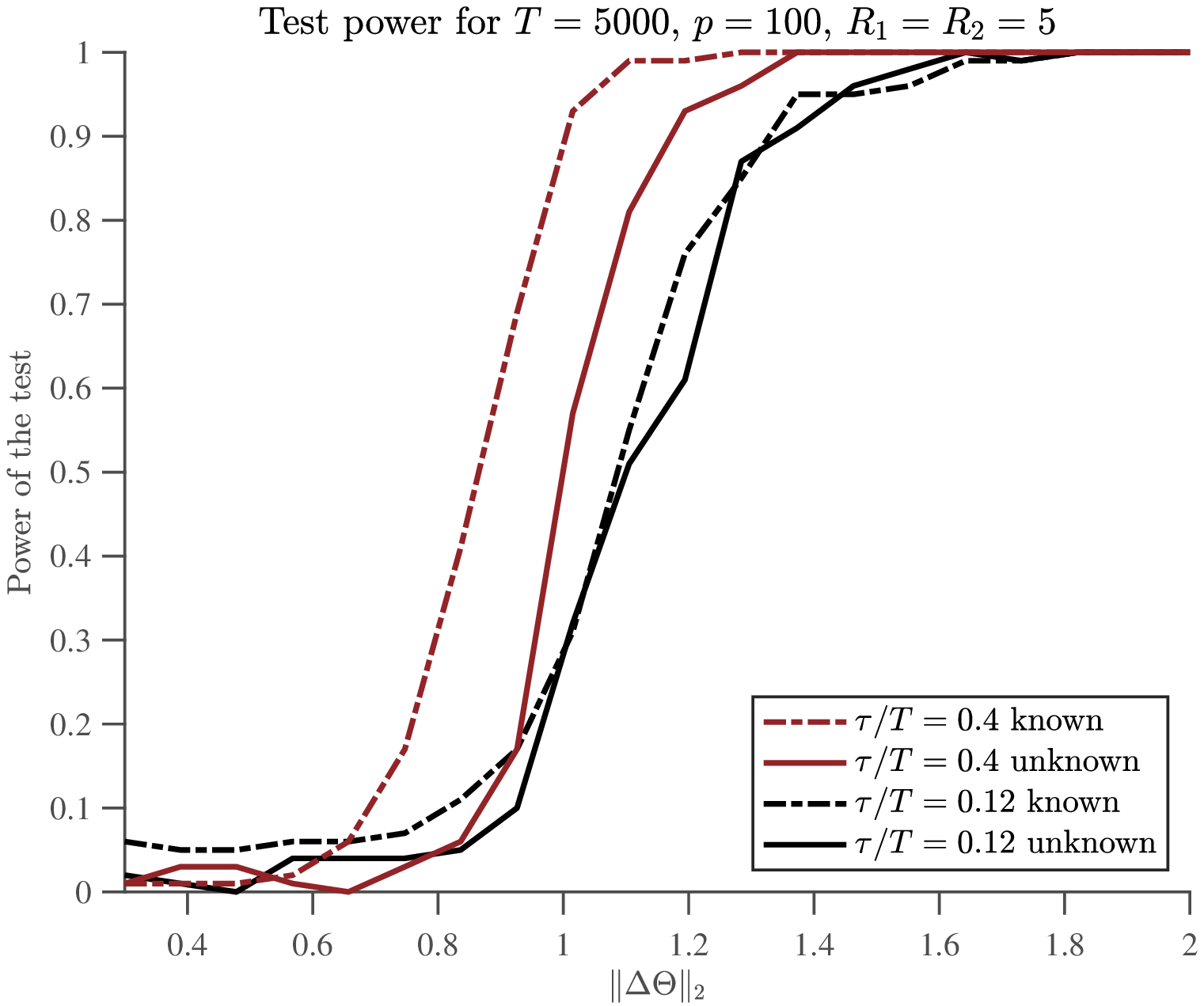}
		\caption{Test power for known or unknown location of the change-point}
		\label{fig:power-tau-dyadic}
	\end{minipage}
\end{figure}

To calculate the test statistic  $\mathcal G(t)$ we need to find the solutions $\widehat \Theta_1^t$ and $\widehat\Theta_2^t$ of the SDPs~\eqref{eq:Theta_hat1} and \eqref{eq:Theta_hat2}. The estimation quality of the transition matrices  depends on the choice of constants in the regularization parameters $\lambda_1^t:=\tilde c_1\sqrt{p/t}$ and $\lambda_2^t:=\tilde c_2 \sqrt{p/(T-t)}$. 
We tune the constants $\tilde c_i$ by cross-validation in the following way. We estimate the transition matrices $\Theta_1$ and $\Theta_2$ using, respectively, the first and the last  $\lfloor \delta Th\rfloor$ observations $X_1,\dots,X_{\lfloor \delta Th\rfloor}$ and  $X_{T-\lfloor Th\rfloor},\dots,X_{T-\lfloor (1-\delta)Th\rfloor}$, where $\delta\in(0,1)$ is chosen a priori.  Let $\widehat\Theta_{1,\tilde c_1}$ and $\widehat\Theta_{2,\tilde c_2}$ be the corresponding solutions of  SDPs~\eqref{eq:Theta_hat1}  and  \eqref{eq:Theta_hat2} with the regularization constants $\tilde c_1$ and $\tilde c_2$ varying within some appropriately chosen grid. The optimal constants $\tilde c_1^*$ and $\tilde c_2^*$ will minimize the least-squares criteria for $\widehat\Theta_{1,_{\tilde c_1}}$ and $\widehat\Theta_{2,\tilde c_{2}}$  calculated over the interval of size $(1-\delta)Th$ :
$$
\tilde c_1^*= \argmin_{\tilde c_1\in \mathrm{grid}}  \sum_{t=\lfloor \delta Th\rfloor}^{\lfloor Th\rfloor-1}\|X_{t+1}-\widehat \Theta_{1,\tilde c_1}X_t\|_2^2,\quad  \tilde c_2^* = \argmin_{\tilde c_2\in \mathrm{grid}}  \sum_{t=T-\lfloor (1-\delta) Th\rfloor}^{T-1}\|X_{t+1}-\widehat \Theta_{2,\tilde c_2}X_t\|_2^2.
$$
To obtain the plug-in estimates $\widehat \Theta_1^t$ and $\widehat \Theta_2^t$ used in $\mathcal G(t)$, we take the regularization parameters $\lambda_1^t$ and $\lambda_2^t$ with the constants $\tilde c_1^*$ and $\tilde c_2^*$. The numerical solution to the SDPs is calculated using the accelerated gradient decent algorithm of~\cite{Ji2009AnAG}.

The theoretical quantile of the change-point detection procedure depends on an unknown universal constant. In this study we use Monte-Carlo simulated quantiles $q_{\alpha,t}$. The details about the quantile simulation are provided in the Appendix. 

We have performed 100 simulations of $T=5000$ observations of  $\mathrm{VAR}_p(\Theta_1,\Theta_2,\Sigma_Z)$ process  with independent noise, $\Sigma_Z=\id_p$ and transition matrices $\Theta_1$ and $\Theta_2$ of the same rank $R_1=R_2=5$. The significance level $\alpha$ is fixed to 0.05. We have chosen the intervals of size $Th=5p$ for the estimation of  matrices $\Theta_1$ and $\Theta_2$ with $\delta=4/5$ for the constant  calibration interval. We consider the problems of testing for a change-point at a given point $\tau\in \cpset_T$ and at an unknown point within the set $\cpset_T$.  In the second case the dyadic grid defined  in~\eqref{eq:dyadic_grid} was chosen as $\mathcal T$.  

Fig.~\ref{fig:power-p} shows the dependence of the test power on the Frobenius norm of the change in transition matrices, $\|\Delta\Theta\|_2$, for  different values of $p$ varying  from $50$ to $200$. Here we test at a given point located in the middle. We see that the power decreases if the dimension increases and the detection problem  becomes harder for the matrices of larger dimension. This  confirms our theoretical detection rate  $C\sqrt{Rp/T}$. In Fig.~\ref{fig:power-tau-dyadic} we compare  our test's performance for a given change-point versus an unknown change-point. In this simulation the dimension is fixed, $p=100$. The bold line corresponds to the test power within the dyadic grid for the case of unknown change-point location. We see that the adaptive test performs quite well with respect to  testing at a known change-point location. We can also see the impact of the change-point location to the detection rate: the detection is harder near the boundary of $\cpset_T$ and easier in the middle of the interval. Additional simulation results can be found in the Appendix.

\bibliography{references}

\begin{appendix}
	\section{Definitions from minimax testing theory}
	
	Let $\mX=(X_0,\dots,X_T)$ be observed data satisfying $\mathrm{VAR}_p(\Theta_1,\Theta_2,\Sigma_Z)$ model~(\ref{eq:var}) with the change at point $\tau$ and the transition matrices $\Theta_1$ and $\Theta_2$ before and after the change. Denote by $\P^\tau_{\Theta_1,\Theta_2}$ the distribution of $\mX$. Let  $\psi:\mX\to \{0,1\}$ be a test for the presence of a change. Define the testing errors of $\psi(\mX)$. The type I error is given by 
	$$
	\alpha(\psi)=\sup_{(\Theta_1,\Theta_2)\in\mathcal V_{p,0}}  \P_{\Theta_1,\Theta_2}\{\psi(\mX)=1\},
	$$
	the type II error of testing at the given point $\tau$ is defined as 
	$$
	\beta(\psi,\mathcal R_{p,T}^\tau)=\sup_{(\Theta_1,\Theta_2)\in\mathcal V_{p,\tau}(\mathcal R_{p,T}^\tau)} \P^\tau_{\Theta_1,\Theta_2}\{\psi(\mX)=0\},
	$$
	and the type II error of testing at an unknown change-point $\tau\in\cpset_T$ is defined as 
	$$
	\beta(\psi,\mathcal R_{p,T})=\sup_{\tau\in\cpset_T}\sup_{(\Theta_1,\Theta_2)\in\mathcal V_{p,\tau}(\mathcal R_{p,T})} \P^\tau_{\Theta_1,\Theta_2}\{\psi(\mX)=0\}.
	$$
	Let $\alpha\in (0,1)$ be a given significance level. Denote by $\Psi_\alpha$  the set of all tests of level at most $\alpha$:
	$$
	\Psi_\alpha=\left\{\psi:\ \alpha(\psi)\le \alpha \right\}.
	$$

	It is important to know what are the conditions on the jump matrix $\Delta\Theta$ and the radius $\mathcal R_{p,T}$ that allow to detect the change-point with a given significance level and reasonable type II error. These conditions are formulated in terms of the minimax separation rate. 
	
	\begin{definition}\label{def:nonasymp_sep_rate}
		Let $\alpha,\beta\in(0,1)$ be given.  We say that the radius $\mathcal R^{*}_{p,T}$ is  $(\alpha,\beta)$-minimax detection boundary in problem of testing the hypothesis of no change against the alternative $\mathcal V_p,\tau(\mathcal R_{p,\cpset_T})$ if 
		$$
		\mathcal R^{*}_{p,\cpset_T} =\inf_{\psi\in\Psi_\alpha} \mathcal R_{p,\cpset_T}(\alpha,\psi),
		$$
		where 
		$$
		\mathcal R_{p,\cpset_T}(\alpha,\psi)=\inf \Bigl\{\mathcal R>0:\ \beta(\psi,\mathcal R)\le \beta \Bigr\}.
		$$ 	
	\end{definition}
	The minimax detection boundary is often written as the product $\mathcal R^{*}_{p,\cpset_T}=C\varphi_{p, \cpset_T}$, where $\varphi_{p,\cpset_T}$ is called {\it minimax detection rate} and  $C$ is a constant  independent of $p$ and $T$.   We say that the radius $\mathcal R_{p,\cpset_T}$ satisfies the upper bound condition  if there exists a constant $C^{*}>0$ and a test $\psi^{*}\in \Psi_\alpha$ such that $\forall C>C^{*}$ $\beta(\psi^{*}, \mathcal R_{p,\cpset_T})\le \beta$.  We say that $ \mathcal R_{p,\cpset_T}$ satisfies the lower  bound condition if for any $0<C\le C_*$  there is no test of level $\alpha$ with type II error smaller than $\beta$. Our goal is to find the minimax detection rate  $\varphi_{p,\cpset_T}$ and two constants $C_*$ and $C^{*}$ such that 
	$$
	C_* \varphi_{p,\cpset_T}\le \mathcal R^{*}_{p,\cpset_T}\le C^{*} \varphi_{p,\cpset_T}.
	$$
	
	\section{Proof of the upper bound}
	Denote 
	$$
	\cF_1(t)=\varphi^t_1(\mX,\widehat \Theta_2^{t})
	- \varphi^t_1(\mX,\widehat \Theta_1^t)\quad\mbox{and}\quad \cF_2(t)=\varphi^t_2(\mX,\widehat \Theta_1^{t})
	- \varphi^t_2(\mX,\widehat \Theta_2^t)
	$$
	\begin{lemma}[Type  I error]\label{lem:typeIerror}
		Let $\alpha\in(0,1)$ be a given significance level. Assume that  $T/2\geq p >\log\left (8|\mathcal T|/\alpha\right )$. If we solve \eqref{eq:Theta_hat1}--\eqref{eq:Theta_hat2} with  $\lambda^t_1$ and  $\lambda^t_2$ defined in \eqref{choice_lambda1}--\eqref{choice_lambda2}, then the  type I error of test~\eqref{eq:test}  is bounded by $\alpha$. 
	\end{lemma}
	\begin{proof} 
		By the union bound, the type I error of the test $\psi_\alpha$ can be bounded as follows,
		$$
		\alpha(\psi_{\alpha})=\sup_{(\Theta_1,\Theta_2)\in\mathcal V_p(\tau,0)} \P_{\Theta_1,\Theta_2} \left(  \max_{t\in \mathcal T} \frac{\mathcal F(t)}{H_{\alpha,t}}>1  \right)\le \sum_{t\in\mathcal T}  \sup_{(\Theta_1,\Theta_2)\in\mathcal V_p(\tau,0)}\P_{\Theta_1,\Theta_2} \Bigl\{   \mathcal{F}(t)\geq H_{\alpha,t}\Bigr\}.
		$$
		Denote by $\Theta:=\Theta_1=\Theta_2$ the transition matrix of the process under the null hypothesis. To provide an $\alpha$-level test, we need to show that the type I error of testing  at each point $t\in \mathcal T$ is bounded by $\alpha/|\mathcal T|$:
		$$
		\sup_{\Theta \in\mathcal M_p(R,\gamma)}\P_{\Theta} \Bigl\{   \mathcal{F}(t)\geq H_{\alpha,t}\Bigr\}\le \alpha/|\mathcal T|,
		$$

		For any two matrices $A$ and $B$ we have
		\begin{align}\label{eq:1}
		\sum_{i=t}^{T-1}&\Bigl (\| X_{i+1}-AX_{i}\| _{2}^{2} -\| X_{i+1}-BX_{i}\| _{2}^{2}\Bigr)\nonumber  \\
		&=\sum_{i=t}^{T-1}\Bigl(\| \left (\Theta-A\right )X_{i}+Z_{i+1}\| _{2}^{2}-
		\| \left (\Theta-B\right )X_{i}+Z_{i+1}\| _{2}^{2}\Bigr)\nonumber \\
		&=\sum_{i=t}^{T-1}\Bigl(\|  (\Theta-A)X_{i}\| _{2}^{2}-
		\| \left (\Theta-B\right )X_{i}\| _{2}^{2}\Bigr)
		+2\sum_{i=t}^{T-1}\Bigl(\left\langle (\Theta-A)X_{i},Z_{i+1}\right\rangle -\left\langle (\Theta-B)X_{i},Z_{i+1}\right\rangle \Bigr)\nonumber \\
		&=\| \left (\Theta-A\right )\mX_{\ge t}\| _{2}^{2}-
		\| \left (\Theta-B\right )\mX_{\ge t}\| _{2}^{2}+2\left\langle \left (B-A\right )\mX_{\ge t},\mathfrak{Z}_{\ge t}\right\rangle. 
		\end{align}
		A similar calculation provides
		\begin{align}\label{eq:2}
		&\sum_{i=0}^{t-1}\left (\| X_{i+1}-AX_{i}\| _{2}^{2}- \| X_{i+1}-BX_{i}\| _{2}^{2}\right ) \nonumber \\ & \hskip 2 cm=\| \left (\Theta-A\right )\mX_{<t}\| _{2}^{2}-
		\| \left (\Theta-B\right )X_{i}\| _{2}^{2}+2\left\langle \left (B-A\right )\mX_{\ge t},\mathfrak{Z}_{\ge t}\right\rangle.
		\end{align}

		Since the test statistic $\mathcal F(t)$ depends on the location of $t$ with respect to the boundary, we will consider four following cases: 
		(a)~$t\ge T/2$ and $T-t\le p$, $t\ge p$,
		(b)~$t\ge T/2$ and $T-t\ge p$,
		(c)~$t\le T/2$ and $t\le p$, $T-t\ge T-p$,
		and (d)~$t\le T/2$ and $t\ge p$.
		
		Note that the assumption $p\le T/2$ is important. It implies, for example, that $t\ge p$ if $T-t\le p$ and $t\ge T/2$.
		
		By symmetry, it is sufficient to consider the cases (a) and (b). Consider the case (a) when $t\ge T/2\ge p$ and $T-t\le p$. Using~\eqref{eq:1} and \eqref{eq:2} we have 
		\begin{align*}
		&\sum_{i=t}^{T-1}\left (\| X_{i+1}-\widehat{\Theta}_1^tX_{i}\| _{2}^{2}-\| X_{i+1}-\widehat{\Theta}_2^tX_{i}\| _{2}^{2}\right ) \nonumber\\
		& \hskip 1 cm =\|(\Theta-\widehat{\Theta}_1^t)\mX_{\ge t} \| _{2}^{2}-
		\| (\Theta-\widehat{\Theta}_2^t )\mX_{\ge t}\| _{2}^{2}+2\left\langle  (\widehat{\Theta}_2^t -\widehat{\Theta}_1^t )\mX_{\ge t},\mathfrak{Z}_{\ge t}\right\rangle. 
		\end{align*}
		Thus $\mathcal F(t)$ can be written as the sum of three terms, 
		\begin{align}
		\mathcal{F}(t)&=\frac{1}{T-t}\left(\|(\Theta-\widehat{\Theta}_1^t)\mX_{\ge t} \| _{2}^{2}-
		\| (\Theta-\widehat{\Theta}_2^t )\mX_{\ge t}\| _{2}^{2}\right)+\frac2{T-t}\left\langle  (\widehat{\Theta}_2^t -\widehat{\Theta}_1^t )\mX_{\ge t},\mathfrak{Z}_{\ge t}\right\rangle.\nonumber\\
		&+\lambda_{2}^{t}\left (\| \widehat \Theta_1^t\mathfrak{X}_{\ge t}\|_* -\| \Theta\mathfrak{X}_{\ge t}\|_*+\| \Theta\mathfrak{X}_{\ge t}\|_*-\| \widehat \Theta_2^t\mathfrak{X}_{\ge t}\|_*\right ):=A+B+C.\label{eq:ABC}
		\end{align}
		
		Introduce the following random events,
		\begin{gather*}
		\mathcal{A}=\left \{ \| \mathfrak{Z}_{\ge t}\|_{\mathrm{op}} \leq c_1^* \sqrt{\| \Sigma_{Z}\|_{\mathrm{op}} }\left ( \sqrt{T-t}+\sqrt{p}+\sqrt{\log  \frac{8|\mathcal T|}{\alpha}}\right)\right  \},\\
		\mathcal{B}_1=\left \{ \| \widehat{\Theta}_1^t-\Theta\|_{2}^{2}\leq c^*\frac{\kappa^2(\Sigma)}{(1-\gamma)^2}\frac{Rp}{t} \right \}\quad \mbox{and}\quad \mathcal{B}_2=\left \{ \| \widehat{\Theta}_2^t-\Theta\|_{2}^{2}\leq c^* \frac{\kappa^2(\Sigma)}{(1-\gamma)^2}\frac{Rp}{T-t} \right \}
		\end{gather*}
		Lemma~\ref{lem:bound_stoch_term1} implies that $\P(\mathcal{A}) \geq 1-\dfrac{\alpha}{4|\mathcal T|}$ and Proposition~\ref{cor:bound_risk_estimation2} gives $\P(\mathcal{B}_i) \geq 1-c_{2}\exp(-c_{3}p)\ge  1-\dfrac{\alpha}{4|\mathcal T|}$ for sufficiently large $p$.  Thus, we have 
		$$
		\P_\Theta\{\mathcal F(t)>H_{\alpha,t}\} \le \P_\Theta\Bigl\{\{\mathcal F(t)>H_{\alpha,t}\}\cap\{ \mathcal A\cap \mathcal B_1\cap \mathcal B_2\}\Bigr\} + \P(\mathcal A^c) + \P_\Theta(\mathcal B_1^c) + \P_\Theta(\mathcal B_2^c).
		$$
		To show that the type I error is bounded by $\alpha$,  we have to show that the first probability is bounded by $\alpha/(4|\mathcal T|)$. 
		
		Consider first the scalar product term $B$ in~\eqref{eq:ABC}. We have
		\begin{align*}
		\left (\widehat\Theta_2^t-\widehat\Theta_1^t\right )\mathfrak{X}_{\ge t}&=\left( \Theta-\widehat\Theta_1^t\right )\mathfrak{X}_{\ge t} +\left (\widehat\Theta_2^t-\Theta \right )\mathfrak{X}_{\ge t}\\
		&= \left( \Theta-\widehat\Theta_1^t\right )\mathfrak{X}_{\ge t} - \mathbf{Pr}_{\Theta\mathfrak{X}_{\ge t}}\Bigl[\Theta\mathfrak{X}_{\ge t}\Bigr]+\mathbf{Pr}_{\Theta\mathfrak{X}_{\ge t}}\Bigl[\widehat\Theta_2^t \mathfrak{X}_{\ge t}\Bigr]+\mathbf {Pr}^{\bot}_{\Theta\mathfrak{X}_{t}}\Bigl[\widehat\Theta_2^t \mathfrak{X}_{\ge t}\Bigr].
		\end{align*}
		Using this identity and the H\"{o}lder inequality, we obtain
		\begin{align*}
		\left\langle \left (\widehat\Theta_2^t-\widehat\Theta_1^t\right )\mathfrak{X}_{\ge t}, \mathfrak{Z}_{\ge t}\right\rangle&\le 
		\| \left (\widehat\Theta_1^t-\Theta\right )\mathfrak{X}_{\ge t}\|_{*}\| \mathfrak{Z}_{\ge t}\|_{\mathrm{op}}\\
		&+\| \mathbf {Pr}_{\Theta\mathfrak{X}_{\ge t}}\left [\left (\Theta-\widehat\Theta_2^t\right )\mathfrak{X}_{\ge t}\right ]\|_{*} \| \mathfrak{Z}_{\ge t}\|_{\mathrm{op}}+\|\mathbf {Pr}^{\bot}_{\Theta\mathfrak{X}_{\ge t}}\left [\widehat\Theta_2^t\mathfrak{X}_{\ge t}\right ]\|_{*}\| \mathfrak{Z}_{\ge t}\|_{\mathrm{op}}.
		\end{align*}
		By Lemma~\ref{lem:bound_stoch_term1}, taking into account the fact that $p\ge T-t$ and $p> \log(8 |\mathcal T|/\alpha)$, we obtain that with probability at least $1-\alpha/(4 |\mathcal T|)$, 
		$$
		\frac2{T-t} \|\mathfrak{Z}_{\ge t}\|_{\mathrm{op}} \le 6c_1^*\sqrt{\|\Sigma_Z\|_{\mathrm{op}}} \frac{\sqrt p}{T-t}= \lambda_2^t.
		$$
		Thus, given  $\mathcal A$, we have 
		$$
		B \le \lambda^{t}_{2}\| (\widehat\Theta_1^t-\Theta)\mathfrak{X}_{\ge t}\|_{*} + \lambda^{t}_{2}\| \mathbf {Pr}_{\Theta\mathfrak{X}_{\ge t}}\left [(\Theta-\widehat\Theta_2^t )\mathfrak{X}_{\ge t}\right ]\|_{*} + \|\mathbf {Pr}^{\bot}_{\Theta\mathfrak{X}_{\ge t}}\left [\widehat\Theta_2^t\mathfrak{X}_{\ge t}\right ]\|_{*}
		$$
		Using the triangle inequality,  $\| \widehat \Theta_1^t\mathfrak{X}_{\ge t}\|_* -\| \Theta\mathfrak{X}_{\ge t}\|_* \leq \| \left (\widehat\Theta_1^t-\Theta\right )\mathfrak{X}_{\ge t}\|_{*}$ and the fact that by~\eqref{triangle_inequality_trace_norm},
		$\| \Theta\mathfrak{X}_{\ge t}\|_*-\| \widehat \Theta_2^t\mathfrak{X}_{\ge t}\|_* \leq \| \mathbf {Pr}_{\Theta\mathfrak{X}_{\ge t}}\left [\left (\Theta-\widehat\Theta_2^t\right )\mathfrak{X}_{\ge t}\right ]\|_{*} - \|\mathbf {Pr}^{\bot}_{\Theta\mathfrak{X}_{\ge t}}\left [\left( \Theta - \widehat\Theta_2^t\right)\mathfrak{X}_{\ge t}\right ]\|_{*}$, we obtain for the term $C$, 
		$$
		C \leq \lambda^{t}_{2}\| \left (\widehat\Theta_1^t-\Theta\right )\mathfrak{X}_{\ge t}\|_{*} +\lambda^{t}_{2}\| \mathbf {Pr}_{\Theta\mathfrak{X}_{\ge t}}\left [\left (\Theta-\widehat\Theta_2^t\right )\mathfrak{X}_{\ge t}\right ]\|_{*}-\lambda^{t}_{2}\|  \mathbf {Pr}^{\bot}_{\Theta\mathfrak{X}_{\ge t}}\left [ \widehat\Theta_2^t\mathfrak{X}_{\ge t}\right ] \|_{*}.
		$$
		Gathering these bounds, we obtain that given $\mathcal A$,
		\begin{align}
		\mathcal{F}(t)& \leq \frac{1}{T-t}\| \left (\Theta-\widehat\Theta_1^t\right )\mathfrak{X}_{\ge t}\| _{2}^{2}-\frac{1}{T-t} \| \left (\Theta-\widehat\Theta_2^t\right )\mathfrak{X}_{\ge t}\| _{2}^{2}
		\nonumber
		\\&
		+2\lambda^{t}_{2}\| \left (\widehat\Theta_1^t-\Theta\right )\mathfrak{X}_{\ge t}\|_{*}
		+2\lambda^{t}_{2}\| \mathbf {Pr}_{\Theta\mathfrak{X}_{\ge t}}\left [\left (\Theta-\widehat\Theta_2^t\right )\mathfrak{X}_{\ge t}\right ]\|_{*}.
		\label{eq:FtboundABC}		
		\end{align}
		Using Lemma~\ref{lem:triangle_inequality_trace_norm}, we get the bound
		$$
		\| \widehat\Theta_1^t-\Theta\|_{*} \le  \| \mathbf {Pr}^{\bot}_{\Theta}\left [ \widehat\Theta_1^t-\Theta\right ]\|_{*} + \| \mathbf {Pr}_{\Theta}\left [ \widehat\Theta_1^t-\Theta\right ]\|_{*}\le 6\| \mathbf {Pr}_{\Theta}\left [ \widehat\Theta_1^t-\Theta\right ]\|_{*}.
		$$
		Taking into account that 
		$\rank \left(\mathbf {Pr}_{\Theta}\Bigl[\widehat\Theta_1^t-\Theta \Bigr ] \right) \leq 2 \rank( \Theta ) =2R$,  using $\|A\|_*^2 \le \rank(A)\|A\|_2^2$ and the trivial inequality $2ab\leq a^{2}+b^{2}$, we get
		\begin{align*}
		2\lambda^{t}_{2}\| \left (\widehat\Theta_1^t-\Theta\right )\mathfrak{X}_{\ge t}\|_{*} &\leq 2\lambda^{t}_{2}\| \widehat\Theta_1^t-\Theta \|_{* } \| \mathfrak{X}_{\ge t}\|_{\mathrm{op}}\\
		& \leq 12 \lambda^{t}_{2}\| \mathbf {Pr}_{\Theta}\left [ \widehat\Theta_1^t-\Theta\right ]\|_{*} \| \mathfrak{X}_{\ge t}\|_{\mathrm{op}}\\
		& \leq 12 \lambda^{t}_{2} \sqrt{2R}\| \widehat\Theta_1^t-\Theta \|_{2 } \| \mathfrak{X}_{\ge t}\|_{\mathrm{op}}\\
		&\leq 2 \left (\lambda^{t}_{2}\right )^{2}R(T-t)+\dfrac{36}{T-t}\| \widehat\Theta_1^t-\Theta \|_{2 }^{2} \| \mathfrak{X}_{\ge t}\|_{\mathrm{op}}^{2}.
		\end{align*}
		Thus by Proposition~\ref{cor:bound_risk_estimation2}, given  $\mathcal B_1$ with probability at least $1-c_2\exp (-c_3p)\ge 1-\alpha/(4|\mathcal T|)$, we have
		\begin{center}
			$2\lambda^{t}_{2}\| \left (\widehat\Theta_1^t-\Theta\right )\mathfrak{X}_{\ge t}\|_{*} \leq 2 \left (\lambda^{t}_{2}\right )^{2}R(T-t)+\dfrac{36c^* Rp}{t(T-t)}  \dfrac{\kappa^2(\Sigma)}{(1-\gamma)^2}\| \mathfrak{X}_{\ge t}\|_{\mathrm{op}}^{2}$
		\end{center}
		Since $\rank\left (\Theta\mathfrak{X}_{t}\right )\leq\rank (\Theta)$,  we have, using the same argument as above, we have 
		\begin{align*}
		2\lambda^{t}_{2}\| \mathbf {Pr}_{\Theta\mathfrak{X}_{\ge t}}\left [(\Theta-\widehat\Theta_2^t )\mathfrak{X}_{\ge t}\right ]\|_{*} &\leq 2 \lambda^{t}_{2} \sqrt{2R} \|  (\Theta-\widehat\Theta_2^t )\mathfrak{X}_{\ge t}\| _{2}\\
		&\leq 2 \left (\lambda^{t}_{2}\right )^{2}R(T-t) + \frac{1}{T-t}\| (\Theta-\widehat\Theta_2^t )\mathfrak{X}_{\ge t} \|_{2}^{2}.
		\end{align*}
		Similarly, by Proposition~\ref{cor:bound_risk_estimation2}, given  $\mathcal B_2$ we have
		\begin{align*}
		\frac{1}{T-t}\| (\Theta-\widehat\Theta_1^t )\mathfrak{X}_{\ge t}\| _{2}^{2} \leq \dfrac{c^*Rp}{t(T-t)}  \dfrac{\kappa^2(\Sigma)}{(1-\gamma)^2} \| \mathfrak{X}_{\ge t}\|_{\mathrm{op}}^{2}
		\end{align*}
		Finally, for sufficiently large $p$, with the probability at least $1-\alpha/(4|\mathcal T|)$, \eqref{eq:FtboundABC} and the above inequalities and the definition of $\lambda_2^t$ imply, given the event $\mathcal A\cap \mathcal B_1\cap\mathcal B_2$, 
		$$
		\mathcal{F}(t) \leq \dfrac{37c^*Rp}{t(T-t)}  \dfrac{\kappa^2(\Sigma)}{(1-\gamma)^2} \| \mathfrak{X}_{\ge t}\|_{\mathrm{op}}^{2}+ \dfrac{144 (c_1^*)^{2} \| \Sigma_{Z}\|_{\mathrm{op}}Rp}{T-t}.
		$$
		Thus, for any $\Theta\in\mathcal M_p(R,\gamma)$,
		\begin{align*}
		\P_\Theta\Bigl\{\{\mathcal F(t)>H_{\alpha,t}\}
		&\cap\{ \mathcal A\cap \mathcal B_1\cap\mathcal B_2\}\Bigr\}\\
		&\le \P\left\{\dfrac{37c^*Rp}{t(T-t)}\dfrac{\kappa^2(\Sigma)}{(1-\gamma)^2} \| \mathfrak{X}_{\ge t}\|_{\mathrm{op}}^{2}+ \dfrac{144 (c^*)^{2}Rp \| \Sigma_{Z}\|_{\mathrm{op}}}{T-t}>H_{\alpha,t}\right\}\\
		&\le \P\left\{ \frac1{T-t}\|\mathfrak{X}_{\ge t}\|_{\mathrm{op}}^{2} \ge 2\|\Sigma\|_{\mathrm{op}} \Bigl(1+\frac {c_0}{1-\gamma}\frac{p+\log(8|\mathcal T|/\alpha)}{T-t} 
		\Bigr)\right\}\le \frac{\alpha}{4|\mathcal T|},
		\end{align*}
		where we used  Lemma~\ref{lem:upper_bound_sampling_operator} with $\delta= \dfrac{\alpha}{8 |\mathcal T|}$, $T-t\le p$, $t\ge p$ and the fact that  for a sufficiently large universal constant $C$ depending on $c^*$, $c_0$ and $c_1^*$, we have the bound
		\begin{align*}
		H_{\alpha,t}&= C\max\Bigl(\| \Sigma_Z\|_{\mathrm{op}},\dfrac{\kappa^2(\Sigma)}{(1-\gamma)^2} \|\Sigma\|_{\mathrm{op}}\Bigr) \dfrac{Rp}{Tq^{2}(t/T)} \left(1+\frac1{1-\gamma}\dfrac{p+\log(8|\mathcal T|/\alpha)}{T}\right)\\
		&\ge \dfrac{144 (c_1^*)^{2}Rp \| \Sigma_{Z}\|_{\mathrm{op}}}{T-t} + \frac{74c^*Rp}{t} \dfrac{\kappa^2(\Sigma)}{(1-\gamma)^2} \|\Sigma\|_{\mathrm{op}} \Bigl(1+\frac {c_0}{1-\gamma}\frac{p+\log(8|\mathcal T|/\alpha)}{T-t} 
		\Bigr).
		\end{align*}
		
		Let us turn to the case (b) of the candidate change-point located far from the endpoints of the interval, $T-t\ge p$ and $t\ge T/2$. Calculations similar to the case (a) imply
		\begin{align*}
		\mathcal{F}(t)& =\frac{1}{T-t}\sum_{i=t}^{T-1}\left (\| \left (\Theta-\widehat\Theta_1^t\right )X_i\| _{2}^{2}-\| \left (\Theta-\widehat\Theta_2^t\right )X_i\| _{2}^{2}\right )
		+ \frac{2}{T-t}\left\langle \left (\widehat\Theta_2^t-\widehat\Theta_1^t\right ), \mathfrak {Z}_{\ge t}\mathfrak{X}^{\T}_{\ge t}\right\rangle\\
		&+\lambda_{2}^{t}\left (\| \widehat \Theta_1^t\|_* -\| \Theta\|_*+\| \Theta\|_*-\|\left ( \widehat \Theta_2^t\right )\|_*\right )=A+B+C.
		\end{align*} 
		Introduce the  random events
		$$
		\mathcal{A}'=\left \{\frac1{T-t} \| \mathfrak {Z}_{\ge t}\mathfrak{ X}_{\ge t}^{\T}\|_{\mathrm{op}} \leq \frac{c_2^*  \|\Sigma\|_{\mathrm{op}}}{1-\gamma}\sqrt{ \frac{p}{T-t}}\right \}\quad \mbox{and}\quad \mathcal{D}=\left\{\frac1{T-t}\sigma_{\min}(\mX_{\ge t}^\T\mX_{\ge t})\ge \frac{\sigma_{\min}(\Sigma_2)}{4}\right\}.
		$$
		Lemma~\ref{lem:bound_stoch_term2} and Lemma~\ref{lem:upper_bound_sampling_operator} imply that $\P(\mathcal{A}') \ge 1-\alpha/(8|\mathcal T|)$ and  $\P(\mathcal{D}')\ge1-\alpha/(8|\mathcal T|)$ for sufficiently large such that$p\ge a_1\log(8|\mathcal T|/a_2)$ for some universal constants $a_1,a_2>0$.
		Similarly to the case (a), we have 
		$$
		\P_\Theta\{\mathcal F(t)>H_{\alpha,t}\} \le \P_\Theta\Bigl\{\{\mathcal F(t)>H_{\alpha,t}\}\cap \mathcal A'\cap \mathcal B_1\cap \mathcal B_2\cap \mathcal D\Bigr\} + \P_\Theta((\mathcal A')^c) + \P_\Theta(\mathcal B_1^c) + \P_\Theta(\mathcal B_2^c) +\P_\Theta(\mathcal D^c).
		$$
		We have to show that the first probability is bounded by $\alpha/(4|\mathcal T|)$. 
		As in the previous case, we get using the H\"{o}lder inequality, 
		\begin{align*}
		\left\langle \left (\widehat\Theta_2^t-\widehat\Theta_1^t\right ), \mathfrak {Z}_{\ge t}\mathfrak{X}^{\T}_{\ge t}\right\rangle& \leq \| \widehat\Theta_1^t-\Theta\|_{*}\| \mathfrak {Z}_{\ge t}\mathfrak{X}^{\T}_{\ge t}\|_{\mathrm{op}}
		\\
		& +\| \mathbf {Pr}_{\Theta}\left (\Theta-\widehat\Theta_2^t\right )\|_{*} \| \mathfrak {Z}_{\ge t}\mathfrak{X}^{\T}_{\ge t}\|_{\mathrm{op}}+\|\mathbf {Pr}^{\bot}_{\Theta}\left (\widehat\Theta_2^t\right )\|_{*}\| \mathfrak {Z}_{\ge t}\mathfrak{X}^{\T}_{\ge t}\|_{\mathrm{op}}.
		\end{align*}
		Remind that $\lambda^t_2=\dfrac{2c_2^*\| \Sigma \|_{\mathrm{op}} }{1-\gamma}\sqrt{\displaystyle{\frac p{T-t}}}$. Lemma~\ref{lem:bound_stoch_term2} implies that for sufficiently large $p$, with the probability at least  $1-k_1\exp(-k_2p)$,  
		$$
		\dfrac{2}{T-t}\| \mathfrak {Z}_{\ge t}\mathfrak{X}^{\T}_{\ge t}\|_{\mathrm{op}}\leq \dfrac{2c_2^*\|\Sigma\|_{\mathrm{op}}}{1-\gamma}\sqrt{\dfrac{p}{T-t}}=\lambda_2^t.
		$$
		Thus, given the event $\mathcal A'$, we get 
		$B\leq \lambda^{t}_{2}\| \widehat\Theta_1^t-\Theta\|_{*} + \lambda^{t}_{2}\| \mathbf {Pr}_{\Theta}\left [\Theta-\widehat\Theta_2^t\right ]\|_{*} + \lambda^{t}_{2}\|\mathbf {Pr}^{\bot}_{\Theta}\left [\widehat\Theta_2^t\right ]\|_{*}$.
		Similarly, for the term $C$, using the triangle inequality and  the inequality \eqref{triangle_inequality_trace_norm} implying that $\| \Theta\|_*-\| \widehat \Theta_2^t\|_* \leq \| \mathbf {Pr}_{\Theta}\left [\Theta-\widehat\Theta_2^t\right ]\|_{*} - \|\mathbf {Pr}^{\bot}_{\Theta}\left [ \Theta - \widehat\Theta_2^t\right ]\|_{*}$, we obtain 
		$$
		C \leq \lambda^{t}_{2}\| \widehat\Theta_1^t-\Theta\|_{*} +\lambda^{t}_{2}\| \mathbf {Pr}_{\Theta}\left [\Theta-\widehat\Theta_2^t\right ]\|_{*}- \lambda^{t}_{2}\|  \mathbf {Pr}^{\bot}_{\Theta}\left [ \widehat\Theta_2^t\right ] \|_{*}.
		$$
		Thus, given $\mathcal A'$, we have 
		$$
		\mathcal{F}(t) \leq \dfrac{1}{T-t}\| \left (\Theta-\widehat\Theta_1^t\right )\mathfrak{X}_{\ge t}\| _{2}^{2}-\dfrac{1}{T-t} \| \left (\Theta-\widehat\Theta_2^t\right )\mathfrak{X}_{\ge t}\| _{2}^{2} 
		+2\lambda^{t}_{2}\| \widehat\Theta_1^t-\Theta\|_{*}
		+2\lambda^{t}_{2}\| \mathbf {Pr}_{\Theta}\left (\Theta-\widehat\Theta_2^t\right )\|_{*}.	
		$$
		As in the case (a),  using Lemma~\ref{lem:triangle_inequality_trace_norm} and $2ab\leq a^{2}+b^{2}$, we can bound the  last two terms of this inequality
		$$
		2\lambda^{t}_{2}\| \widehat\Theta_1^t-\Theta\|_{*} \leq 12\lambda^{t}_{2}\sqrt{2R} \| \widehat\Theta_1^t-\Theta\|_{2}
		\leq \dfrac{9R \left (\lambda^{t}_{2}\right )^{2}}{\sigma_{\min}(\Sigma)} + 8\sigma_{\max}(\Sigma)\|\Theta-\widehat\Theta_1^t \|_{2}^{2}
		$$
		and 
		$$
		2\lambda^{t}_{2}\| \mathbf {Pr}_{\Theta}\left (\Theta-\widehat\Theta_2^t\right )\|_{*} \leq 2\lambda^{t}_{2}\sqrt{2R} \|\widehat\Theta_2^t- \Theta\|_{2}\\
		\leq \dfrac{8R\left (\lambda^{t}_{2}\right )^{2}}{\sigma_{\min}(\Sigma)}+\dfrac{\sigma_{\min}(\Sigma)}{4}\| \widehat\Theta_2^t-\Theta \|_{2}^{2}.
		$$
		By Proposition~\ref{cor:bound_risk_estimation2}, with probability at least $1-\alpha/(4|\mathcal T|)$, 
		\begin{align*}
		\frac{1}{T-t}\| \left (\Theta-\widehat\Theta_1^t\right )\mathfrak{X}_{\ge t}\| _{2}^{2} \leq \dfrac{c^*Rp}{t(T-t)}\dfrac{\kappa^2(\Sigma)}{(1-\gamma)^2}\| \mathfrak{X}_{\ge t}\|_{\mathrm{op}}^{2}.
		\end{align*}
		This implies that given $\mathcal A'\cap \mathcal B_1$ 
		\begin{align*}
		\mathcal{F}(t) & \leq \dfrac{c^*Rp}{t(T-t)}\dfrac{\kappa^2(\Sigma)}{(1-\gamma)^2} \| \mathfrak{X}_{\ge t}\|_{\mathrm{op}}^{2} - \dfrac{1}{T-t} \sigma_{\min}(\mathfrak{X}_{\ge t}^{\T}\mathfrak{X}_{\ge t}) \| \Theta-\widehat\Theta_2^t \|_{2}^{2} + \dfrac{17R \left (\lambda^{t}_{2}\right )^{2}}{\sigma_{\min}(\Sigma)}
		\\
		&+ 8\sigma_{\max}(\Sigma)\|\Theta-\widehat\Theta_1^t \|_{2}^{2} + \dfrac{\sigma_{\min}(\Sigma)}{4}\| \widehat\Theta_2^t-\Theta \|_{2}^{2}.
		\end{align*}
		Using the definition of $\lambda_2^t$, Lemma~\ref{lem:upper_bound_sampling_operator} and Proposition~\ref{cor:bound_risk_estimation2}, we finally obtain that given $\mathcal A'\cap \mathcal B_1\cap\mathcal B_2\cap \mathcal D$ 
		\begin{align*}
		\mathcal{F}(t) & \leq \dfrac{c^*Rp}{t(T-t)}\dfrac{\kappa^2(\Sigma)}{(1-\gamma)^2}  \| \mathfrak{X}_{\ge t}\|_{\mathrm{op}}^{2} + \dfrac{68 (c_2^*)^2\|\Sigma\|_{\mathrm{op}}^2}{(1-\gamma)^2\sigma_{\min}(\Sigma)}\frac {Rp}{T-t}
		+  8c^*\sigma_{\max}(\Sigma) \frac{\kappa^2(\Sigma)}{(1-\gamma)^2}\frac{Rp}{t} \\
		&\le \dfrac{c^*Rp}{t(T-t)}\dfrac{\kappa^2(\Sigma)}{(1-\gamma)^2} \| \mathfrak{X}_{\ge t}\|_{\mathrm{op}}^{2}  + \max(c^*,68(c_2^*)^2)\frac{\kappa^2(\Sigma)\|\Sigma\|_{\op}}{(1-\gamma)^2}\frac{Rp}{Tq^2(t/T)} \\
		&\le C\frac{Rp}{Tq^2(t/T)} \dfrac{\kappa^2(\Sigma)}{(1-\gamma)^2} \left(\frac 1T \| \mathfrak{X}_{\ge t}\|_{\mathrm{op}}^{2} +\|\Sigma\|_{\op}\right),
		\end{align*}
		where $C$ is an absolute constant depending on $c^*$ and $c_2^*$. 
		Note that since $T-t>p$ and $T-t\le T/2$,  we have the following bound for a sufficiently large universal constant $C$,
		\begin{align*} 
		H_{\alpha,t}
		&\ge C  \dfrac{ \kappa^{2}(\Sigma)\| \Sigma\|_{\mathrm{op}}}{(1-\gamma)^2}\dfrac{Rp}{Tq^{2}(t/T)} \left(1+\frac{T-t}{T}\left(1+\frac{c_0}{1-\gamma}\left(1+\dfrac{\log(8|\mathcal T|/\alpha)}{T-t}\right)\right)\right).
		\end{align*}
		Using this fact and Lemma~\ref{lem:upper_bound_sampling_operator} with $\delta= \dfrac{\alpha}{8 |\mathcal T|}$, we obtain that for any $\Theta\in\mathcal M_p(R,\gamma)$,
		\begin{align*}
		\P_\Theta\Bigl\{\{\mathcal F(t)>H_{\alpha,t}\}&\cap\{ \mathcal A'\cap \mathcal B_1\cap\mathcal B_2\}\Bigr\} 
		\le  \P\left\{C\frac{Rp}{Tq^2(t/T)} \dfrac{\kappa^2(\Sigma)}{(1-\gamma)^2} \left(\frac 1{T} \| \mathfrak{X}_{\ge t}\|_{\mathrm{op}}^{2} +\|\Sigma\|_{\op}\right)>H_{\alpha,t}\right\}\\
		&\le \P\left\{ \frac1{T-t}\|\mathfrak{X}_{\ge t}\|_{\mathrm{op}}^{2} \ge 2\|\Sigma\|_{\mathrm{op}} \left(1+\frac{c_0}{1-\gamma}\left(1+
		\frac{\log(8|\mathcal T|/\alpha)}{T-t} \right)\right)\right\}
		\le \frac{\alpha}{4|\mathcal T|}
		\end{align*}
		and the lemma follows. \end{proof}
	

	\begin{lemma}[Type  II error]\label{lem:typeIIerror}
		Let $\beta\in(0,1)$ be given significance level  such that $p>\max\Bigl(k_2^{-1}\log(4k_1/\beta)\Bigr),\\ c_3^{-1}\log(4c_2/\beta), \log(8/\beta)\Bigr)$, where $k_1,k_2,c_2,c_3$ are the constants from Proposition~\ref{cor:bound_risk_estimation2} and Lemma~\ref{lem:bound_stoch_term2}. Assume that 
		\begin{equation}\label{eq:separation_rate}
		q^2\Bigl(\frac \tau T\Bigr)\|\Theta_1-\Theta_2\|_2^2\geq 
		\Xi (\Sigma_1,\Sigma_2,\Sigma,\gamma) \frac{Rp}{T} \left(1+\frac{\log(8|\mathcal T|/\alpha)+\log(8/\beta)}{T}\right),
		\end{equation}
		where, for some absolute constant $C^{*}$, 
		$$
		\Xi(\Sigma_1,\Sigma_2,\Sigma,\gamma)= \frac{C^{*}}{(1-\gamma)^3}\frac{\mathfrak M_1\vee \mathfrak M_2}
		{\mathfrak m}
		$$
		with 
		\begin{align}
		\mathfrak M_1&=\max\Bigl(\frac{\|\Sigma_2\|_{\op}^2}{\sigma_{\min}(\Sigma_Z)\wedge\sigma_{\min}(\Sigma_1)},\frac{\|\Sigma_1\|_{\op}^2}{\sigma_{\min}(\Sigma_Z)\wedge\sigma_{\min}(\Sigma_2)}\Bigr)\label{constM1}\\
		\mathfrak M_2 &= \max\Bigl(\kappa^2(\Sigma_1)\|\Sigma_2\|_{\op},\kappa^2(\Sigma_2)\|\Sigma_1\|_{\op} , \kappa^2(\Sigma)\|\Sigma\|_{\mathrm{op}},\|\Sigma_Z\|_{\mathrm{op}}\Bigr)\label{constM2}
		\end{align}
		and 
		\begin{equation}\label{constm}
		\mathfrak m= \sigma_{\min}(\Sigma_Z)\wedge\sigma_{\min}(\Sigma_1)\wedge\sigma_{\min}(\Sigma_2).
		\end{equation}
		Then,  if we solve \eqref{eq:Theta_hat1}--\eqref{eq:Theta_hat2} with $\lambda^t_1$ and  $\lambda^t_2$ given by \eqref{choice_lambda1} and \eqref{choice_lambda2}, the type II error of the test $\psi_\alpha$ defined by~\eqref{eq:test}  is bounded by $\beta$.
	\end{lemma}
	
	\begin{proof} We consider the case of $\mathcal T=\cpset_T$. Let $\tau\in\cpset_T$ be the true change-point. 
		The type II error of the test $\psi_{\alpha}$ is defined as 
		\begin{align*}
		\beta(\psi_{\alpha},\mathcal{R}_{\rho , \cpset_{T}}) &= \sup_{\tau \in\cpset_T} \sup_{(\Theta_{1}, \Theta_{2}) \in V_{p}(\tau,\mathcal R_{p,\cpset_T})} \P_{\Theta_{1},\Theta_{2}}^\tau\Bigl\{\psi_{\alpha}=0\Bigr\}\\
		&=\sup_{\tau \in\cpset_T} \sup_{(\Theta_{1}, \Theta_{2}) \in V_{p}(\tau,\mathcal R_{p,\cpset_T})} \P_{\Theta_{1},\Theta_{2}}^\tau\Bigr\{\max_{t\in \mathcal T} \frac{\mathcal F(t)}{H_{\alpha,t}} <1\Bigr\}\\
		&\le \inf_{t\in \mathcal T}  \sup_{\tau \in\cpset_T} \sup_{(\Theta_{1}, \Theta_{2}) \in V_{p}(\tau,\mathcal R_{p,\cpset_T})} \P_{\Theta_{1},\Theta_{2}}^\tau\Bigr\{ \mathcal F(t) < H_{\alpha,t}\Bigr\}\\
		&\le  \sup_{\tau \in\cpset_T} \sup_{(\Theta_{1}, \Theta_{2}) \in V_{p}(\tau,\mathcal R_{p,\cpset_T})} \P_{\Theta_{1},\Theta_{2}}^\tau\Bigr\{ \mathcal F(\tau) < H_{\alpha,\tau}\Bigr\}.
		\end{align*}
		We have to show that if the minimal jump energy satisfies condition~\eqref{eq:separation_rate}, then the type II error of the test is at most $\beta$, $\P_{\Theta_{1},\Theta_{2}}^\tau\Bigr\{ \mathcal F(\tau) < H_{\alpha,\tau}\Bigr\}\le \beta$. As in the case of the type I error, we have to consider four cases of location of the change-point $\tau$ with respect to $T/2$ and the boundaries. 
		
		Consider the case (a) of $\tau\ge T/2$ and $T-\tau\le p$. Let us introduce the events
		\begin{gather*}
		\mathcal{A}=\left \{ \| \mathfrak{Z}_{\ge t}\|_{\mathrm{op}} \leq c_1^* \sqrt{\| \Sigma_{Z}\|_{\mathrm{op}} }\left ( \sqrt{T-\tau}+\sqrt{p}+\sqrt{\log  \frac{8}{\beta}}\right)\right  \},\quad 
		\mathcal{B}_1=\left \{ \| \widehat{\Theta}_1^t-\Theta_1\|_{2}^{2}\leq c^*\frac{\kappa^2(\Sigma_1)}{(1-\gamma)^2}\frac{Rp}{t} \right \},\\  \mathcal{C}=\left \{ \frac 1{T-\tau}\|(\Theta_1-\Theta_2)\mX_{\ge \tau}\|_2^2\ge \frac{\sigma_{\min}(\Sigma_Z)\vee \sigma_{\min}(\Sigma_2)}{2(\gamma^2+1)}\|\Theta_1-\Theta_2\|_2^2\right \}.
		\end{gather*}
		Lemma~\ref{lem:bound_stoch_term1} implies that $\P(\mathcal{A}) \geq 1-\beta/4$.
		Proposition~\ref{cor:bound_risk_estimation2} implies $\P(\mathcal{B}_1) \geq 1-c_{2}\exp(-c_{3}p)\ge  1-\beta/4$ for sufficiently large $p$. Moreover, Lemma~\ref{lem:low_bound_sampling_operator} implies that $\P(\mathcal{C}) \geq 1-\beta/4$ under the condition 
		\begin{equation}\label{eq:deltacond}
		\|\Theta_{2} - \Theta_{1}\|_{2}> \frac{4\gamma(1+\gamma^2)\sigma_{\max}(\Sigma_2)}{\sigma_{\min}(\Sigma_{Z})\wedge \sigma_{\min}(\Sigma_2)} \sqrt{\frac 1c\frac{\log (8/\beta)}{T-\tau}}.
		\end{equation}
		Note that \eqref{eq:deltacond} always holds under condition~\eqref{eq:separation_rate}. We can write 
		$$
		\P_{\Theta_{1},\Theta_{2}}^\tau\Bigr\{ \mathcal F(\tau) < H_{\alpha,\tau}\Bigr\}
		\le \P \Bigl\{ \{\mathcal F(\tau) < H_{\alpha,\tau}\}\cap \mathcal A\cap \mathcal B_1\cap\mathcal C \}+\P(\mathcal A^c)+\P(\mathcal B_1^c)+\P(\mathcal C^c).
		$$
		To bound the type II error by $\beta$, we need to show that the first probability is bounded by $\beta/4$.
		
		By definition of  $\widehat \Theta_2^\tau$, the test statistic can be bounded by 
		$$
		\mathcal{F}(\tau)=\varphi^\tau_2(\mathfrak X,\widehat \Theta_1^\tau)- \varphi^\tau_2(\mathfrak X,\widehat \Theta_2^\tau) \geq \varphi^\tau_2(\mathfrak X,\widehat \Theta_1^\tau)- \varphi^\tau_2(\mathfrak X, \Theta_2) 
		$$
		Taking into account that $\varphi^\tau_2(\mathfrak X, \Theta_2) = \dfrac{1}{T-\tau}\sum\limits_{i=\tau}^{T-1}\| Z_{i}\| _{2}^{2}+\lambda^\tau_2\| \Theta_{2}\mX_{\ge \tau}\| _{*}$ and 
		\begin{align*}
		\varphi^\tau_2(\mathfrak X,\widehat \Theta_1^\tau)&= \dfrac{1}{T-\tau} \sum\limits_{i=\tau}^{T-1}\| \left (\Theta_{2}-\widehat\Theta_1^\tau\right )X_{i}\| _{2}^{2} + \dfrac{1}{T-\tau}\sum\limits_{i=\tau}^{T-1}\| Z_{i}\| _{2}^{2} \\
		&+ \dfrac{2}{T-t}\left\langle \left (\Theta_2-\widehat\Theta_1^\tau\right )\mX_{\ge \tau}, \mathfrak{Z}_{-\tau}\right\rangle + \lambda_{2}^{\tau} \| \widehat\Theta_1^\tau\mX_{\ge \tau}\| _{*}
		\end{align*}
		we obtain
		\begin{align*}
		\mathcal{F}(\tau) &\geq \dfrac{1}{T-\tau} \sum\limits_{i=\tau}^{T-1}\| \left (\Theta_{2}-\widehat\Theta_1^\tau\right )X_{i}\| _{2}^{2}
		+ \dfrac{2}{T-t}\left\langle \left (\Theta_2-\widehat\Theta_1^\tau\right )\mX_{\ge \tau}, \mathfrak{Z}_{-\tau}\right\rangle  \\
		&+ \lambda_{2}^{\tau} \left( \| \widehat\Theta_1^\tau \mX_{\ge \tau}\| _{*}-\| \Theta_{2}\mX_{\ge \tau}\| _{*}\right):=A+B+C.
		\end{align*}
		Note that for any $i$ we have 
		$\dfrac{1}{2} \| \left(\Theta_{2}-\Theta_1\right)X_{i}\| _{2}^{2} \leq \| \left (\Theta_{2}-\widehat\Theta_1^\tau\right )X_{i}\| _{2}^{2} + \| \left (\Theta_{1}-\widehat\Theta_1^\tau\right )X_{i}\| _{2}^{2}$.
		Thus
		$$
		A\geq \dfrac{1}{2(T-\tau)}\sum\limits_{i=\tau}^{T-1} \| \left (\Theta_{2}-\Theta_1\right )X_{i}\| _{2}^{2} - \dfrac{1}{T-\tau} \sum\limits_{i=\tau}^{T-1} \| \left (\Theta_{1}-\widehat\Theta_1^\tau\right )X_{i}\| _{2}^{2}.
		$$
		Next, by the H\"{o}lder inequality,
		$$
		B\geq -\dfrac{2}{T-\tau}\| \left (\widehat\Theta_1^\tau-\Theta_{1}\right )\mX_{\ge \tau}\|_{*}\| \mathfrak{Z}_{-\tau}\|_{\mathrm{op}}-	\dfrac{2}{T-\tau}\| \left ({\Theta}_2-\Theta_{1}\right )\mX_{\ge \tau}\|_{*}\| \mathfrak{Z}_{-\tau}\|_{\mathrm{op}}.
		$$
		Since  $\lambda^\tau_2=\dfrac{6c_1^* \sqrt{\| \Sigma_{Z}\|_{\mathrm{op}}}  \sqrt{p}}{T-\tau}$,  using Lemma~\ref{lem:bound_stoch_term1}  we can show using $p\ge \log (8/\beta)$ and $T-\tau\le p$  that $\lambda^\tau_2 \geq \dfrac{2\| \mathfrak{Z}_{-\tau}\|_{\mathrm{op}}}{T-\tau}$ holds under the event $\mathcal A$. Consequently, given $\mathcal A$,
		\begin{align*}
		B &\geq -\lambda^{\tau}_{2}\| \left (\widehat\Theta_1^\tau-\Theta_{1}\right )\mX_{\ge \tau}\|_{*} - \lambda^{\tau}_{2}  \| \left (\Theta_{2}-\Theta_{1}\right )\mX_{\ge \tau}\|_{*}  \\
		&=-\lambda^{\tau}_{2} \left(  \| \mathbf {Pr}^{\bot}_{\Theta_1\mX_{\ge \tau}}\left [\left (\Theta_1-\widehat{\Theta}_1\right )\mX_{\ge \tau}\right ]\|_{*} + \| \mathbf {Pr}_{\Theta_1\mX_{\ge \tau}}\left [\left (\Theta_1-\widehat{\Theta}_1\right )\mX_{\ge \tau}\right ]\|_{*} +  \| \left (\Theta_{2}-\Theta_{1}\right )\mX_{\ge \tau}\|_{*}  \right)
		\end{align*}
		Similarly to Lemma~\ref{lem:typeIerror} we can bound the term $C$ using the triangle inequality and \eqref{triangle_inequality_trace_norm}:
		$$
		C\geq \lambda^{\tau}_{2}\left (\| \mathbf {Pr}^{\bot}_{\Theta_1\mX_{\ge \tau}}\left [\left (\Theta_1-\widehat\Theta_1^\tau\right )\mX_{\ge \tau}\right ]\|_{*} -\| \left (\Theta_{2}-\Theta_{1}\right )\mX_{\ge \tau}\|_{*}-\| \mathbf {Pr}_{\Theta_1\mX_{\ge \tau}}\left [\left (\Theta_1-\widehat{\Theta}_1\right )\mX_{\ge \tau}\right ]\|_{*}\right ).
		$$
		Gathering all these bounds, we obtain that given $\mathcal A\cap \mathcal B_1\cap\mathcal C$,
		\begin{align*}
		\mathcal{F}(t) &\geq 
		\frac{1}{2(T-\tau)}\sum_{i=\tau}^{T-1}\| \left (\Theta_{2}-\Theta_1\right )X_{i}\| _{2}^{2}-\frac{1}{T-\tau}\sum_{i=\tau}^{T-1}\| \left (\Theta_{1}-\widehat\Theta_1^\tau\right )X_{i}\| _{2}^{2}
		\\&  -2\lambda^{\tau}_{2} \| \left (\Theta_{2}-\Theta_{1}\right )\mX_{\ge \tau}\|_{*}-2\lambda^{\tau}_{2}\| \mathbf {Pr}_{\Theta_1\mX_{\ge \tau}}\left (\Theta_1-\widehat\Theta_1^\tau\right )\mX_{\ge \tau}\|_{*}.
		\end{align*}
		Using the trivial inequality $a^2+b^2\ge 2ab$ and the fact that $\rank(\Theta_2-\Theta_1)\le 2R$, we can show that 
		\begin{align*}
		2\lambda^{\tau}_{2} \| \left (\Theta_{2}-\Theta_{1}\right )\mX_{\ge \tau}\|_{*} &\leq 2 \lambda^{\tau}_{2}  \sqrt{2 \rank ((\Theta_{2}-\Theta_{1}) \mX_{\ge \tau})}  \| \left (\Theta_{2}-\Theta_{1}\right )\mX_{\ge \tau}\|_{2}\\
		& \leq 4 \lambda^{\tau}_{2} \sqrt{R } \| \left (\Theta_{2}-\Theta_{1}\right )\mX_{\ge \tau}\|_{2}
		\\
		&\leq 8 \left( \lambda^{\tau}_{2} \right)^{2}R(T-\tau) + \dfrac{\| \left (\Theta_{2}-\Theta_{1}\right )\mX_{\ge \tau}\|_{2}^{2}}{4(T-\tau)}
		\end{align*}
		and
		\begin{align*}
		2\lambda^{\tau}_{2}\| \mathbf {Pr}_{\Theta_1\mX_{\ge \tau}}\left (\Theta_1-\widehat\Theta_1^\tau\right )\mX_{\ge \tau}\|_{*} 
		&\leq 2 \lambda^{\tau}_{2}  \sqrt{2\rank(\Theta_{1} \mX_{\ge \tau})}  \| \left (\Theta_{1}-\widehat\Theta_1^\tau\right )\mX_{\ge \tau}\|_{2}
		\\
		&\leq 2 \lambda^{\tau}_{2}  \sqrt{2 R}  \| \left (\Theta_{1}-\widehat\Theta_1^\tau\right )\mX_{\ge \tau}\|_{2}\\
		&\leq  2\left( \lambda^{\tau}_{2}\right)^{2}R(T-\tau) + \frac1{T-\tau} \| \left (\Theta_{1}-\widehat\Theta_1^\tau\right )\mX_{\ge \tau}\|_{2}^{2}
		\end{align*}
		Consequently, gathering the bounds, we get that given $\mathcal A\cap \mathcal B_1\cap\mathcal C$,
		\begin{align*}
		\mathcal{F}(\tau)	
		\geq -\frac{2}{T-\tau}\| \left (\Theta_{1}-\widehat\Theta_1^\tau\right )\mX_{\ge \tau}\|^{2}_{2}	+\frac{1}{4(T-\tau)}\| \left (\Theta_1-\Theta_{2}\right )\mX_{\ge \tau}\|^{2}_{2}
		-10\left (\lambda^{\tau}_{2}\right )^{2}R(T-\tau)\\	
		\ge -\frac2{T-\tau}  \| \mX_{\ge \tau}\| _{\op}^{2}	 \| \Theta-\widehat\Theta_1^\tau\| _{2}^{2} +\frac{1}{4(T-\tau)}\| \left (\Theta_1-\Theta_{2}\right )\mX_{\ge \tau}\|^{2}_{2}
		-10\left (\lambda^{\tau}_{2}\right )^{2}R(T-\tau)
		\end{align*}
		By Proposition~\ref{cor:bound_risk_estimation2} and Lemma~\ref{lem:low_bound_sampling_operator}, we have, given $\mathcal B_1$ and $\mathcal C$, 
		$\| \widehat\Theta_1^\tau-\Theta_1\|_{2}^{2}\leq c^*\dfrac{\kappa^2(\Sigma_{1})}{(1-\gamma)^2}\dfrac{Rp}{\tau}$
		and 
		\begin{align*}
		\dfrac{1}{4(T-\tau)}\| \left (\Theta_1-\Theta_{2}\right )\mX_{\ge \tau}\|^{2}_{2} &\geq \dfrac{\sigma_{\min}(\Sigma_{Z})\wedge \sigma_{\min}(\Sigma_2)}{8(\gamma^{2}+1)} \| \Theta_1-\Theta_{2}\|^{2}_{2}
		\end{align*}
		which implies, given the event  $\mathcal A\cap \mathcal B_1\cap\mathcal C$,
		$$
		\mathcal F(\tau)\ge -\frac{2c^* Rp}{\tau(T-\tau)} \dfrac{\kappa^2(\Sigma_{1})}{(1-\gamma)^2}\| \mX_{\ge \tau}\| _{\mathrm{op}}^{2} +  \dfrac{\sigma_{\min}(\Sigma_{Z})\wedge \sigma_{\min}(\Sigma_2)}{8(\gamma^{2}+1)} \| \Theta_1-\Theta_{2}\|^{2}_{2} - 360 \frac{Rp \|\Sigma_Z\|_{\mathrm{op}} }{T-\tau}.
		$$
		Thus, to show that $\P\Bigl\{\mathcal F(\tau)\le H_{\alpha,\tau}\Bigr\}\le\beta$ we need to prove that under assumption~\eqref{eq:separation_rate} 
		\begin{equation}\label{eq:beta4}
		\P\Bigl\{-\frac{2c^*Rp}{\tau(T-\tau)}\dfrac{\kappa^2(\Sigma_{1})}{(1-\gamma)^2} \| \mX_{\ge \tau}\| _{\mathrm{op}}^{2} +  \dfrac{\sigma_{\min}(\Sigma_{Z})\wedge \sigma_{\min}(\Sigma_{2})}{2(\gamma^{2}+1)} \| \Theta_1-\Theta_{2}\|^{2}_{2} - 360 \frac{Rp \|\Sigma_Z\|_{\mathrm{op}} }{T-\tau}\le H_{\alpha,\tau}\Bigr\}\le \beta/4.
		\end{equation}
		Using Lemma~\ref{lem:upper_bound_sampling_operator} we can show that \eqref{eq:beta4} holds under the condition 
		\begin{align}
		\dfrac{\sigma_{\min}(\Sigma_{Z})\wedge \sigma_{\min}(\Sigma_{2})}{2(\gamma^{2}+1)} \| \Theta_1-\Theta_{2}\|^{2}_{2}&\ge H_{\alpha,\tau}+ 360 \frac{Rp \|\Sigma_Z\|_{\mathrm{op}} }{T-\tau} \nonumber\\
		&+ \frac{4c^*Rp}{\tau}\dfrac{\kappa^2(\Sigma_{1})}{(1-\gamma)^2}  \|\Sigma_2\|_{\mathrm{op}} \left(1+\frac{c_0}{1-\gamma}\frac{p+\log(8/\beta)}{T-\tau}\right).\label{eq:beta5}
		\end{align}
		Using the definition of $H_{\alpha,\tau}$ and the inequality
		$$
		360 \dfrac{Rp \|\Sigma_Z\|_{\mathrm{op}} }{T-\tau} +  \frac{2c^*Rp}{\tau}\dfrac{\kappa^2(\Sigma_{1}) \|\Sigma_2\|_{\mathrm{op}}}{(1-\gamma)^2} \le \max(360,4c^*)\max\Bigl(\|\Sigma_Z\|_{\mathrm{op}},\dfrac{\kappa^2(\Sigma_{1}) \|\Sigma_2\|_{\mathrm{op}}}{(1-\gamma)^2}\Bigr) \frac{Rp}{Tq^2(\tau/T)}
		$$
		we can bound the right-hand side of  inequality~\eqref{eq:beta5} from above by
		$$
		C \dfrac{Rp  }{T q^2(\tau/T)} \max\Bigl(\|\Sigma_Z\|_{\mathrm{op}}, \frac{\kappa^2(\Sigma)\|\Sigma\|_{\mathrm{op}}}{(1-\gamma)^2},\frac{\kappa^2(\Sigma_1)\|\Sigma_2\|_{\mathrm{op}}}{(1-\gamma)^2}\Bigr) \left(1+\frac{c_0}{1-\gamma} \frac{2p+\log(8|\mathcal T|/\alpha)+\log(8/\beta)}T\right),
		$$
		where $C$ is a constant depending only on $c^*$, $c_1^*$ and $c_2^*$.
		Applying the condition $2p<T$ we obtain that~\eqref{eq:beta5}  is guaranteed if  assumption~\eqref{eq:separation_rate} on the norm of the matrix jump is satisfied. This implies~\eqref{eq:beta4}.
		
		Now we have to bound the type II error in case of $\dfrac{T}{2} \leq \tau < T-p $. Set 
		$$
		\mathcal{A}'=\left \{\frac1{T-t} \| \mathfrak {Z}_{\ge t}\mathfrak{ X}_{\ge \tau}^{\T}\|_{\mathrm{op}} \leq \frac{c_2^*  \|\Sigma\|_{\mathrm{op}}}{1-\gamma}\sqrt{ \frac{p}{T-\tau}}\right \}.
		$$
		Lemma~\ref{lem:bound_stoch_term2} implies that $\P(\mathcal{A}') \geq 1-k_1\exp(-k_2p)\ge 1-\beta/4$ if $p\ge k_2^{-1}\log (4k_1/\beta)$. We have to show that $\P\{\{\mathcal F(\tau)<H_{\alpha,\tau}\}\cap \mathcal A'\cap \mathcal B_1\cap \mathcal C\}\le \beta/4$. 
		
		Using the same reasoning as above we obtain 
		\begin{align*}
		\mathcal{F}(\tau)&= \varphi^\tau_2(\mathfrak X,\widehat \Theta_1^\tau)- \varphi^\tau_2(\mathfrak X,\widehat \Theta_2^\tau) \geq \varphi^\tau_2(\mathfrak X,\widehat \Theta_1^\tau)- \varphi^\tau_2(\mathfrak X, \Theta_2^\tau) 
		\\
		&=\frac{1}{T-\tau}\sum_{i=\tau}^{T-1}\| \left (\Theta_{2}-\widehat\Theta_1^\tau\right )X_{i}\| _{2}^{2}-\frac{2}{T-\tau}\left \langle	\widehat\Theta_1^\tau-\Theta_{2},\mathfrak{Z}_{-\tau}\mX_{\ge \tau}^{\T}\right \rangle +\lambda^{\tau}_{2}\left (\| \widehat\Theta_{1}(\tau)\|_{*}-\| \Theta_{2}\|_{*}\right )
		\\&
		\geq \frac{1}{2(T-\tau)}\sum_{i=\tau}^{T-1}\| \left (\Theta_{2}-\Theta_1\right )X_{i}\| _{2}^{2}-\frac{1}{T-\tau}\sum_{i=\tau}^{T-1}\| \left (\Theta_{1}-\widehat\Theta_1^\tau\right )X_{i}\| _{2}^{2}+
		\\&
		\hskip 0.5 cm
		-\frac{2}{T-\tau}\| \widehat\Theta_1^\tau-\Theta_{1}\|_{*}\| \mathfrak{Z}_{-\tau}\mX_{\ge \tau}^{\T}\|_{\mathrm{op}}-	\frac{2}{T-\tau}\| {\Theta}_2-\Theta_{1}\|_{*}\| \mathfrak{Z}_{-\tau}\mX_{\ge \tau}^{\T}\|_{\mathrm{op}}
		\\&
		\hskip 0.5 cm
		-\lambda^{\tau}_{2}\left (\| \Theta_{2}-\Theta_{1}\|_{*}+\| \mathbf {Pr}_{\Theta_1}\left [\Theta_1-\widehat{\Theta}_1\right ]\|_{*}-\| \mathbf {Pr}^{\bot}_{\Theta_1}\left [\Theta_1-\widehat{\Theta}_1\right ]\|_{*}\right ).
		\end{align*}
		Remind that  $\lambda^\tau_2=2c_2^*\dfrac{\|\Sigma_2 \|_{\mathrm{op}}}{1-\gamma}\sqrt{  \dfrac p{T-\tau}}$. Using Lemma~\ref{lem:bound_stoch_term2}, we obtain that given $\mathcal A'$, 
		$$
		\dfrac{2}{T-\tau}\| \mathfrak{Z}_{-\tau}\mX^{\T}_{\ge t}\|_{\mathrm{op}}\leq \dfrac{2c_2^*\|\Sigma_2\|_{\mathrm{op}}}{1-\gamma}\sqrt{\dfrac{p}{T-\tau}}=\lambda_{2}^\tau.
		$$
		By Lemma~\ref{lem:low_bound_sampling_operator} provided the condition~\eqref{eq:deltacond}, we have that  given the event $\mathcal C$ 
		$$
		\dfrac{1}{2(T-\tau)}\sum\limits_{i=\tau}^{T-1}\| \left (\Theta_{2}-\Theta_1\right )X_{i}\| _{2}^{2} \geq  \dfrac{\sigma_{\min}(\Sigma_{Z})\wedge \sigma_{\min}(\Sigma_2)}{4(1+\gamma^{2})}  \| \Theta_{2}-\Theta_1\| _{2}^{2}.
		$$
		Consequently,  given the event $\mathcal A'\cap \mathcal B_1\cap \mathcal C$, 
		\begin{align*}
		\mathcal{F}(\tau) &\geq 
		\frac{1}{2(T-\tau)}\sum_{i=\tau}^{T-1}\| \left (\Theta_{2}-\Theta_1\right )X_{i}\| _{2}^{2}-\frac{1}{T-\tau}\sigma_{\max}(\mX^{\T}_{\ge t}\mX^{\T}_{\ge t})\|  \Theta_{1}-\widehat\Theta_1^\tau\|_2^2
		\\& 
		-2\lambda^{\tau}_{2}\left (\| \Theta_{2}-\Theta_{1}\|_{*}+\| \mathbf {Pr}_{\Theta_1}\left (\Theta_1-\widehat{\Theta}_1\right )\|_{*}\right )\\
		&\ge \dfrac{\sigma_{\min}(\Sigma_{Z})\wedge \sigma_{\min}(\Sigma_2)}{4(1+\gamma^{2})}  \| \Theta_{2}-\Theta_1\| _{2}^{2} -\dfrac{1}{T-\tau} \| \mX_{\ge \tau} \|_{\op}^{2} c^*\frac{Rp}\tau \frac{\kappa^2(\Sigma_1)}{(1-\gamma)^2}
		\\& 
		-2\lambda^{\tau}_{2}\left (\| \Theta_{2}-\Theta_{1}\|_{*}+\| \mathbf {Pr}_{\Theta_1}\left (\Theta_1-\widehat \Theta_1^\tau\right )\|_{*}\right ).
		\end{align*}
		Next, as in the proof of the previous case we can obtain the following bounds given $\mathcal B_1$, 
		\begin{align*}
		2\lambda^{\tau}_{2} \| \Theta_{2}-\Theta_{1}\|_{*} 
		& \leq 2 \lambda^{\tau}_{2}  \sqrt{2R}  \| \Theta_{2}-\Theta_{1}\|_{2}\\
		& \leq \dfrac{16 R(1+\gamma^{2})(\lambda^{\tau}_{2})^{2}}{  \sigma_{\min}(\Sigma_{Z})\wedge \sigma_{\min}(\Sigma_2)}  + \dfrac{ \sigma_{\min}(\Sigma_{Z})\wedge \sigma_{\min}(\Sigma_2)}{8(1+\gamma^{2})} \| \Theta_{2}-\Theta_{1}\|_{2}^{2}\\
		&= \dfrac{64 (c_{2}^*)^2(1+\gamma^{2})\| \Sigma_2 \|^{2}_{\mathrm{op}}}{ (1-\gamma)^{2} (\sigma_{\min}(\Sigma_{Z})\wedge \sigma_{\min}(\Sigma_2))}\frac{Rp}{T-\tau}+ \dfrac{  \sigma_{\min}(\Sigma_{Z})\wedge \sigma_{\min}(\Sigma_2)}{8(1+\gamma^{2})} \| \Theta_{2}-\Theta_{1}\|_{2}^{2}
		\end{align*}
		and
		\begin{align*}
		2\lambda^{\tau}_{2}\| \mathbf {Pr}_{\Theta_1}\left (\Theta_1-\widehat \Theta_1^\tau\right )\|_{*}
		&\leq 2 \lambda^{\tau}_{2} \sqrt{2 R}  \| \Theta_{1}-\widehat\Theta_1^\tau\|_{2}\\
		&\le 2c_2^* \sqrt{2c^*}\frac{\|\Sigma_2\|_{\op}\kappa(\Sigma_1)}{(1-\gamma)^2} \frac{Rp}{Tq(\tau/T)}\le \frac{\sqrt 2}2c_2^* \sqrt{c^*}\frac{\|\Sigma_2\|_{\op}\kappa(\Sigma_1)}{(1-\gamma)^2} \frac{Rp}{Tq^2(\tau/T)}.
		\end{align*}
		Consequently, gathering all the above inequalities, we obtain 
		\begin{align*}
		\Bigl\{\{\cF(\tau)<H_{\alpha,\tau}\}\cap \mathcal A'\cap \mathcal B_1\cap \mathcal C\Bigr\}&\subseteq	\Bigl\{H_{\alpha,\tau}>\cF(\tau)\ge \dfrac{\sigma_{\min}(\Sigma_{Z})\wedge \sigma_{\min}(\Sigma_2)}{8(1+\gamma^{2})}  \| \Theta_{2}-\Theta_1\| _{2}^{2}\Bigr.\\
		& -\frac{2c^* Rp}{\tau(T-\tau)}\dfrac{\kappa^2(\Sigma_{1})}{(1-\gamma)^2} \| \mX_{\ge \tau}\| _{\mathrm{op}}^{2} -  \frac{c_2^* \sqrt {2 c^*}}2\frac{\|\Sigma_2\|_{\op}\kappa(\Sigma_1)}{(1-\gamma)^2} \frac{Rp}{Tq^2(\tau/T)}\\
		&\Bigl.- \dfrac{64 (c_{2}^*)^2(1+\gamma^{2})\| \Sigma_2 \|^{2}_{\mathrm{op}}}{ (1-\gamma)^{2}( \sigma_{\min}(\Sigma_{Z})\wedge \sigma_{\min}(\Sigma_2))}\frac{Rp}{T-\tau}\Bigr\}.
		\end{align*}
		Thus, to show that $\P_{\Theta_1,\Theta_2}^\tau\Bigl\{\mathcal F(\tau)\le H_{\alpha,\tau}\Bigr\}\le\beta$ we need to prove that under assumption~\eqref{eq:separation_rate} 
		\begin{align}
		\P\Bigl\{&\frac{2c^* Rp}{\tau(T-\tau)}\dfrac{\kappa^2(\Sigma_{1})}{(1-\gamma)^2} \| \mX_{\ge \tau}\| _{\mathrm{op}}^{2} \ge  \dfrac{\sigma_{\min}(\Sigma_{Z})\wedge \sigma_{\min}(\Sigma_2)}{2(\gamma^{2}+1)} \| \Theta_1-\Theta_{2}\|^{2}_{2} \Bigr.\nonumber\\
		&\Bigl.-   \frac{c_2^* \sqrt {2 c^*}}2\frac{\|\Sigma_2\|_{\op}\kappa(\Sigma_1)}{(1-\gamma)^2} \frac{Rp}{Tq^2(\tau/T)}- \dfrac{64 c_{1}^2(1+\gamma^{2})\| \Sigma_2 \|^{2}_{\mathrm{op}}}{ (1-\gamma)^{2} (\sigma_{\min}(\Sigma_{Z})\wedge \sigma_{\min}(\Sigma_2))}\frac{Rp}{T-\tau}- H_{\alpha,\tau}\Bigr\}\le \beta/4.\label{eq:beta42}
		\end{align}
		By Lemma~\ref{lem:upper_bound_sampling_operator} and by  the definition of $H_{\alpha,\tau}$, this bound holds if 
		\begin{align*}
		\dfrac{\sigma_{\min}(\Sigma_{Z})\wedge \sigma_{\min}(\Sigma_2)}{2(\gamma^{2}+1)} \| \Theta_1-\Theta_{2}\|^{2}_{2} &\ge \frac{c_2^* \sqrt {2 c^*}}2\frac{\|\Sigma_2\|_{\op}\kappa(\Sigma_1)}{(1-\gamma)^2} \frac{Rp}{Tq^2(\tau/T)}+ \dfrac{64 c_{1}^2(1+\gamma^{2})\| \Sigma_2 \|^{2}_{\mathrm{op}}}{ (1-\gamma)^{2} (\sigma_{\min}(\Sigma_{Z})\wedge \sigma_{\min}(\Sigma_2))}\frac{Rp}{T-\tau} \\
		&+ C\|\Sigma\|_{\mathrm{op}} \frac{\kappa^2(\Sigma)}{(1-\gamma)^2} \frac{Rp}{Tq^2(\tau/T)} \left(1+\frac 2{1-\gamma} \left(1+\frac{\log(8|\mathcal T|/\alpha)}T\right)\right)\\
		&+ \frac{2c^* \kappa^2(\Sigma_1)}{(1-\gamma)^2}\frac{Rp}{\tau} \|\Sigma_2\|_{\mathrm{op}} \left(1+\frac2{1-\gamma}\left(1+\frac{\log(8/\beta)}{T-\tau}\right)\right).
		\end{align*}
		In what follows $C$ denotes an absolute constant. Using the identity $\tau^{-1}(T-\tau)^{-1}=1/(T^2q^2(\tau/T))$ and the bound
		\begin{align*}
		\dfrac{64 c_{1}^2(1+\gamma^{2})\| \Sigma_2 \|^{2}_{\mathrm{op}}}{ (1-\gamma)^{2} (\sigma_{\min}(\Sigma_{Z})\wedge \sigma_{\min}(\Sigma_2))}\frac{Rp}{T-\tau}+\frac{1+\gamma}{1-\gamma}\frac{2c^* \kappa^2(\Sigma_1)\|\Sigma_2\|_{\mathrm{op}}}{(1-\gamma)^2}\frac{Rp}{\tau}  &\le C\frac{\mathfrak M_1\vee \mathfrak M_2}{(1-\gamma)^3}\frac{Rp}{Tq^2(\tau/T)}.
		\end{align*}
		We can show that~\eqref{eq:beta42} holds if 
		$$
		\dfrac{\sigma_{\min}(\Sigma_{Z})\wedge \sigma_{\min}(\Sigma_2)}{2(\gamma^{2}+1)} \| \Theta_1-\Theta_{2}\|^{2}_{2} \ge C\frac{\mathfrak M_1\vee \mathfrak M_2}{(1-\gamma)^3}\frac{Rp}{Tq^2(\tau/T)} \left(1+\frac{\log(8|\mathcal T|/\alpha)+\log(8/\beta)}{T}\right). 
		$$
		which holds if assumption~\eqref{eq:separation_rate} on the norm of the matrix jump is satisfied.
	\end{proof}

	\section{Proof of the lower bound}
	Our goal is to show that there exists a sequence $\rho_{p,T}$ such that if $\mathcal R_{p,T}=o( \rho_{p,T})$  as $p,T\to\infty$, then 
	$$
	\liminf_{p,T\to\infty} \inf_{\psi\in\Psi_\alpha}\beta(\psi,\mathcal R_{p,T}) \ge 1-\alpha.
	$$
	
	\begin{proof}[Proof of Theorem~\ref{th:LowerBoundTesting}]\ 
		\begin{description}	
			\item[\it Step 1. Reduction to Bayesian testing.] \ \\
			Let $\mX=(X_0,\dots,X_T)$, where the observations $X_t$ are defined on a probability space $(\Omega, \mathcal A,\P)$ and follow the $\mathrm{VAR}_p(\Theta_1,\Theta_2,\Sigma_Z)$ model~\eqref{eq:var}. Let $\pi_{p,0}^T$ and $\pi_{p,1}^T$ be two families of priors on the parameter $(\Theta^t)_{0\le t<T}$  under $\Hyp_0$ and $\Hyp_1$, respectively such that $\pi_{p,i}^T(\mathcal M_p(R,\gamma))\to 1$, $i=0,1$, $p,T\to\infty$ and $\pi_{p,1}^T(\Theta_1,\Theta_2)\in\mathcal V_{p,\tau}(\rho))\to 1$, $p,T\to\infty$, where $\rho=\mathcal R_{p,T}>0$ is the minimal detectable energy we are interested in. 
			
			Define the mixtures $\P_{\pi_{p,i}^T}=\E_{\pi_{p,i}^T} \P(A)$, $A\in\mathcal A$.
			The standard technique  to prove the lower bound  (see, for example, \citep{Ingster&Suslina:2003}) consists in showing that the testing risk is bounded from below as follows,
			\begin{align}
			\inf_{\psi_{p,T}\in \Psi_\alpha}\beta (\psi_{p,T},\mathcal R_{p,T}) 
			&\ge \beta(\psi_{p,T}^{*}, \P_{\pi_{p,0}^T},\P_{\pi_{p,1}^T})+o(1)\nonumber\\
			&\ge 1-\frac12 \| \P_{\pi_{p,1}^T}-\P_{\pi_{p,0}^T}\|_{\mathrm{TV}}-\alpha+o(1)\nonumber\\
			&\ge 1-\frac12\left(\E_{\pi_{p,0}^T}\left[ \frac{d\P_{\pi_{p,1}^T}}{d\P_{\pi_{p,0}^T}}(\mX)\right]^2-1\right)^{1/2}-\alpha+o(1).\label{eq:lb_general}
			\end{align}
			Here  $\psi_{p,T}^{*}$  is the optimal Bayesian test for  the hypotheses $\Hyp_0:\ \P= \P_{\pi_{p,0}^T}$ against $\Hyp_1:\ \P= \P_{\pi_{p,1}^T}$ and $\beta(\psi_{p,T}^{*}, \P_{\pi_{p,0}^T},\P_{\pi_{p,1}^T})$ its type II error. 
			If, for the appropriately chosen priors $\pi_{p,1}^T$ and $\pi_{p,0}^T$, we can show that 
			$$
			\E_{\pi_{p,0}^T}\left[ \frac{d\P_{\pi_{p,1}^T}}{d\P_{\pi_{p,0}^T}}(\mX)\right]^2\le 1+o(1),\quad p,T\to\infty
			$$
			then \eqref{eq:lb_general} will imply $\inf\limits_{\psi_{p,T}\in \Psi_\alpha}\beta (\psi_{p,T},\mathcal R_{p,T}) \ge1-\alpha+o(1)$ and the lower bound will follow. 
			\item[\it Step 2. Computation of the second moment of likelihood ratio.]\ \\ 
			Suppose that under $\Hyp_0$ the transition matrix parameter is zero, $\Theta^t=0$ $\forall 0\le t<T$  and under the alternative 
			$\Theta^t=\Theta_1=-(1-\tau/T)\Delta\Theta$ for $0\le t\le \tau$ and $\Theta^t=\Theta_2=(\tau/T)\Delta\Theta$ for $\tau+1\le t<T$. Assume that the change matrix $\Delta\Theta=\Theta_2-\Theta_1$  is distributed according to some prior distribution $\mu=\mu_{p,T}$ defined on $\mathcal M_p(R,\gamma)$ that will be chosen later.  Define the corresponding likelihood ratio of mixtures under $\Hyp_1$ and $\Hyp_0$,
			$$
			L=\frac{d\P_{\pi_{p,1}^T}}{d\P_{\pi_{p,0}^T}}(\mX) := \E_{\Delta\Theta\sim \mu} \Bigl(\frac{d\P_{\Delta\Theta}}{d\P_0}(\mX)\Bigr),
			$$
			where $\P_0$ stands for the measure of $T$ i.i.d. $\mathcal N_p(0,\Sigma_Z)$ 
			Gaussian vectors and $\P_\Delta$ stands for the measure of $T$ consecutive $\mathrm{VAR}_p$ observations with the change-point at $\tau$ and the transition matrices $\Theta_1$ and $\Theta_2$ defined above. 
			We have
			\begin{align*}
			\frac{d\P_{\Delta\Theta}}{d\P_0}(\mX)&= \exp\left\{-\frac12 \left( \Bigl(1-\frac\tau T\Bigr)^2\sum_{t=0}^{\tau-1} X_t^{\T}\Delta\Theta^{\T} \Sigma_Z^{-1}\Delta\Theta X_t +  \Bigl(\frac\tau T\Bigr)^2\sum_{t=\tau}^{T-1} X_t^{\T}\Delta\Theta^{\T} \Sigma_Z^{-1}\Delta\Theta X_t\right.\right.\\
			&+ \Bigl(1-\frac\tau T\Bigr)\sum_{t=0}^{\tau-1} (X_t^{\T}\Delta\Theta^{\T} \Sigma_Z^{-1} X_{t+1} +X_{t+1}^{\T} \Sigma_Z^{-1}\Delta \Theta X_t)\\
			&\left.\left.- \Bigl(\frac\tau T\Bigr)\sum_{t=\tau}^{T-1} (X_t^{\T}\Delta\Theta^{\T} \Sigma_Z^{-1} X_{t+1} +X_{t+1}^{\T} \Sigma_Z^{-1}\Delta \Theta X_t) \right)
			-\frac12 X_0^{\T}(\Sigma_1^{-1}-\Sigma_Z^{-1})X_0\right\}.
			\end{align*}
			Here $\Sigma_1$ is the covariance matrix of the $\mathrm{VAR}_p(\Theta_1)$ process with $\Theta_1=-(1-\frac \tau T)\Delta\Theta$. Recall that $\Theta_1$ and $\Sigma_1$ satisfy the Lyapunov equation $\Sigma_1=\Theta_1\Sigma_1\Theta_1^{\T}+\Sigma_Z$ and, consequently, $\Sigma_1=(1-\frac \tau T)^2\Delta\Theta\Sigma_1\Delta\Theta^{\T}+\Sigma_Z$.  
			
			Let $\Delta\tilde\Theta$ be an independent copy of $\Delta\Theta$ following the law $\mu$. Then the second moment of the likelihood ratio can be written as
			\begin{align*}
			\E_0 [L^2]&= \E_0 \Bigl[ \E_{\Delta\Theta\sim \mu} \Bigl(\frac{d\P_{\Delta\Theta}}{d\P_0}(\mX)\Bigr)^2 \Bigr] \\
			&=  \E_0 \Bigl[ \E_{\mu^2} \Bigl(\frac{d\P_{\Delta\Theta}}{d\P_0}(\mX)\frac{d\P_{\Delta\tilde \Theta}}{d\P_0}(\mX)\Bigr) \Bigr] = \E_{\mu^2}\E_0 \Bigl(\frac{d\P_{\Delta\Theta}}{d\P_0}(\mX)\frac{d\P_{\Delta\tilde \Theta}}{d\P_0}(\mX)\Bigr)
			\end{align*}
			Thus, to bound from above $\E_0[ L^2]$, we will have to bound $\E_0 \Bigl(\frac{d\P_{\Delta\Theta}}{d\P_0}\frac{d\P_{\Delta\tilde \Theta}}{d\P_0}\Bigr)$. We have
			\begin{align*}
			\frac{dP_{\Delta\Theta}}{dP_0}\frac{dP_{\Delta\tilde\Theta}}{dP_0}(\mX)&=
			\exp\left\{-\frac12 \left( \Bigl(1-\frac\tau T\Bigr)^2\sum_{t=0}^{\tau-1} X_t^{\T}(\Delta\Theta^{\T} \Sigma_Z^{-1}\Delta\Theta +\Delta\tilde\Theta^{\T} \Sigma_Z^{-1}\Delta\tilde\Theta)  X_t \right.\right.\\
			&+  \Bigl(\frac\tau T\Bigr)^2\sum_{t=\tau}^{T-1} X_t^{\T}(\Delta\Theta^{\T} \Sigma_Z^{-1}\Delta\Theta +\Delta\tilde\Theta^{\T} \Sigma_Z^{-1}\Delta\tilde\Theta) X_t\\
			&+ \Bigl(1-\frac\tau T\Bigr)\sum_{t=0}^{\tau-1} (X_t^{\T}(\Delta\Theta+\Delta\tilde\Theta)^{\T} \Sigma_Z^{-1} X_{t+1} +X_{t+1}^{\T} \Sigma_Z^{-1}(\Delta\Theta+\Delta\tilde\Theta) X_t)\\
			&\left.\left.- \Bigl(\frac\tau T\Bigr)\sum_{t=\tau}^{T-1} (X_t^{\T}(\Delta\Theta+\Delta\tilde\Theta)^{\T} \Sigma_Z^{-1} X_{t+1} +X_{t+1}^{\T} \Sigma_Z^{-1}(\Delta\Theta+\Delta\tilde\Theta) X_t) \right)\right.\\
			& \left.-\frac12 X_0^{\T}(\Sigma_1^{-1}+\tilde\Sigma_1^{-1}-2\Sigma_Z^{-1})X_0\right\}.
			\end{align*}
			Since under $\Hyp_0$ the observations are i.i.d. $X_t\sim \cN_p(0,\Sigma_Z)$, we can set $X_t=\Sigma_Z^{1/2}U_t$, where $U_t\sim\cN_p(0,I_p)$, $t=0,\dots,T$ are i.i.d. Denote for simplicity  $\Omega=\Sigma_Z^{-1/2}\Delta\Theta\Sigma_Z^{1/2}$, $\tilde \Omega=\Sigma_Z^{-1/2}\Delta\tilde\Theta\Sigma_Z^{1/2}$. 
			Then 
			\begin{align*}
			\E_{0}\Bigl(\frac{dP_{\Delta\Theta}}{dP_0}\frac{dP_{\Delta\tilde\Theta}}{dP_0}(\mX)\Bigr)= \E \exp&\left\{-\frac12 \left( \Bigl(1-\frac\tau T\Bigr)^2\sum_{t=0}^{\tau-1} U_t^{\T}(\Omega^{\T}\Omega+\tilde\Omega^{\T}\tilde\Omega)U_t\right.\right.\\
			&+  \Bigl(\frac\tau T\Bigr)^2\sum_{t=\tau}^{T-1} U_t^{\T}(\Omega^{\T}\Omega+\tilde\Omega^{\T}\tilde\Omega)U_t\\
			&+ \Bigl(1-\frac\tau T\Bigr)\sum_{t=0}^{\tau-1} (U_t^{\T}(\Omega+\tilde\Omega)^{\T} U_{t+1} +U_{t+1}^{\T}(\Omega+\tilde\Omega)U_t)\\
			&\left.\left.- \Bigl(\frac\tau T\Bigr)\sum_{t=\tau}^{T-1} (U_t^{\T}(\Omega+\tilde\Omega)^{\T} U_{t+1} +U_{t+1}^{\T}(\Omega+\tilde\Omega)U_t) \right)\right.\\
			& \left.-\frac12 U_0^{\T}(\Sigma_Z^{1/2}(\Sigma_1^{-1}+\tilde\Sigma_1^{-1})\Sigma_Z^{1/2}-2I_p)U_0\right\}.
			\end{align*}
			Denote by $\mathbf{u}=\mathrm{vec}(U_0,\dots,U_T)\in \bR^{p(T+1)}$ the vector obtained by concatenation of $T+1$ vectors $U_t$ of dimension $p$. Then
			$$
			\E_{0}\Bigl(\frac{dP_{\Delta\Theta}}{dP_0}\frac{dP_{\Delta\tilde\Theta}}{dP_0}(\mX)\Bigr)=  \E \exp \Bigl\{-\frac 12 \mathbf u^{\T} M\mathbf u\Bigr\} =\Bigl(\det(I_{T+1}\otimes I_p+M)\Bigr)^{-1/2},
			$$
			where $\mathbf{u}\sim \cN_{p(T+1)}(0,I_{T+1}\otimes I_p)$ and $M=M(\Omega,\tilde\Omega)$ is the following $p(T+1)\times p(T+1)$ block-tridiagonal symmetric matrix  of blocks of size $p$:
			$$
			M=\begin{pmatrix}
			D_0 & L^{\T} &  O_{p\times p} &	\dots &\dots  &\dots & O_{p\times p}  &	 O_{p\times p}\\
			L & D_1 & L^{\T}&  O_{p\times p} &	\ddots &\ddots &\ddots  &	 O_{p\times p}  \\
			O_{p\times p} &\ddots & \ddots & \ddots& \ddots &\ddots	&\ddots& \vdots \\
			O_{p\times p}  & \ddots  & L & D_{\tau-1}& L^{\T}& O_{p\times p} &\ddots&	 \vdots \\
			\vdots & \ddots &\ddots  & K & D_{\tau}& K^{\T}&\ddots & O_{p\times p} & \\
			\vdots &\ddots & \ddots & \ddots& \ddots &\ddots&	\ddots & O_{p\times p}\\
			O_{p\times p} & \ddots &\ddots  & \ddots  & O_{p\times p}& K& D_{T-1} & K^{\T}\\
			O_{p\times p} & O_{p\times p} &\dots & \dots & \dots&O_{p\times p}& K & D_T
			\end{pmatrix} 
			$$
			with  $L=\Bigl(1-\dfrac\tau T\Bigr)( \Omega+\tilde\Omega)$, $K=-\Bigl(\dfrac\tau T\Bigr)( \Omega+\tilde\Omega)$, $D_0=\Bigl(1-\dfrac\tau T\Bigr)^2 W +\Sigma_Z^{1/2} (\Sigma_1^{-1} +\tilde \Sigma_1^{-1})\Sigma_Z^{1/2} -2I_p$, $D_T=O_{p\times p}$ and 
			$$
			D_t= \Bigl(1-\dfrac\tau T\Bigr)^2 W {\one}_{\{1\le t<\tau\}}+
			\Bigl(\dfrac\tau T\Bigr)^2W {\one}_{\{\tau\le t<T\}},
			$$
			where $W=\Omega^{\T} \Omega+\tilde\Omega^{\T} \tilde\Omega$.
			\item[\it Step 3. Bounds on the determinants.]\ \\  To calculate $\det(I_{T+1}\otimes I_p+M)$, we will use the determinant formula for block matrices: suppose that
			$
			F=\begin{pmatrix}
			A & B\\
			C & D
			\end{pmatrix}
			$
			and $D$ is invertible,  then $\det (F)=\det(D)\det(A-B D^{-1}C)$. Applying $T$ times the block matrix determinant formula starting from the lower-right block matrix $I_p+D_T$, we obtain 
			$$
			\det(I_{T+1}\otimes I_p+M)=\prod_{k=0}^{T} \det(I_p+A_k),
			$$
			where $A_T=O_{p\times p}$ and $A_k$ are defined recursively as follows,
			\begin{equation}\label{eqs:A}
			\begin{aligned}
			&A_{T-k}=\Bigl(\dfrac\tau T\Bigr)^2 W -K^{\T} (I_p+A_{T-k+1})^{-1}K ,\quad k=1,\dots,T-\tau-1\\
			&A_{\tau}=\Bigl(\dfrac\tau T\Bigr)^2 W-L^{\T}(I_p+ A_{\tau+1})^{-1} K\\
			&A_{\tau-k}=\Bigl(1-\dfrac\tau T\Bigr)^2 W -L^{\T}(I_p +A_{\tau-k+1})^{-1}: L\quad k=1,\dots,\tau-1\\
			&A_0=\Sigma_Z^{1/2} (\Sigma_1^{-1} +\tilde \Sigma_1^{-1})\Sigma_Z^{1/2} -2I_p+\Bigl(1-\dfrac\tau T\Bigr)^2 W -L^{\T}(I_p +A_{1})^{-1}L.
			\end{aligned}
			\end{equation}

			
			To bound the likelihood ratio from above, we have to calculate the lower bounds on the determinants $\det(I_p+A_{k})$ that are provided by Lemma~\ref{lem:dets}. Let $\Omega$ and $\tilde \Omega$ be chosen in a way that $\|\Omega\|_2=\|\tilde\Omega\|_2=\delta$ for some $\delta>0$ (see Step 4). 
			Using the identity
			$$
			(T-\tau)\Bigl(\frac\tau T\Bigr)^2+\tau\Bigl(1-\frac\tau T\Bigr)^2= Tq^2\Bigl(\frac\tau T\Bigr).
			$$
			we obtain from Lemma~\ref{lem:dets} that  for sufficiently small $\delta\in(0,1)$ satisfying the conditions of  the lemma, 
			\begin{align*}
			\prod_{k=0}^T\det(I+A_k)^{-1/2}&\le  \exp\left \{Tq^2\Bigl(\frac\tau T\Bigr)\tr(\Omega\tilde \Omega^T) + 3\kappa^2(\Sigma_Z)\Bigl(1-\frac\tau T\Bigr)^2\delta^2\right. \\
			&\left.+16 Tq^2\Bigl(\frac\tau T\Bigr) \delta^4+12  \Bigl(1-\frac\tau T\Bigr)^2(1+\kappa^2(\Sigma_Z))^2\delta^4\right\}.
			\end{align*}
			Consequently,
			\begin{align}\label{bound_product_det}
			\E_0[L^2]\le \exp&\Bigl(C_1\Bigl(1-\frac\tau T\Bigr)^2\delta^2+C_2\Bigl(1-\frac\tau T\Bigr)^4\delta^4+16Tq^2\Bigl(\frac\tau T\Bigr)\delta^4\Bigr) \\
			&\times \E_{(\Omega,\tilde \Omega) }\exp\left\{Tq^2\Bigl(\frac\tau T\Bigr)\tr(\Omega\tilde \Omega^{\T}) \right\},\nonumber
			\end{align}
			where the constants $C_1$ and $C_2$ depend only $\Sigma_Z$. 
			\item[\it Step 4. Prior on the jump matrix.]\ \\  
			We define a low-rank prior $\mu$ on the matrix $\Omega$ as in \citep{Carpentier2015}. Assume that $p$ is a multiple of the rank $R$. Let $v_k$, $k=1,\dots,R$ be  random vectors of size $p$ consisting of $p$ independent Rademacher random entries taking values $\pm 1$ with probability 1/2. Let the matrix $H$ be composed of $R$ blocks $H_k$, $k=1,\dots,R$ of size $p\times (p/R)$ defined as $H_k=v_k\otimes \xi_k^{\T}$, where $\xi_k=(\xi_{1,k},\dots,\xi_{p/R,k})^{\T}$ is a random vector of size $p/R$ with the Rademacher independent entries $\xi_{1,k}$:
			$$
			H=\begin{pmatrix} \underbrace{v_1\otimes \xi_1^{\T}}_{p\times p/R} &   \underbrace{v_2\otimes \xi_2^{\T}}_{p\times p/R} &\dots &  \underbrace{v_R\otimes \xi_R^{\T}}_{p\times p/R}\end{pmatrix}.
			$$
			Set $\Omega=\dfrac{\delta_{p,T}}{p} H$ so that $\|\Omega\|_{\mathrm{op}}\le \|\Omega\|_2=\delta_{p,T}=\delta$. 
			Note that 
			$$
			\|\Omega\|_2=\|\Sigma_Z^{-1/2}\Delta\Theta\Sigma_Z^{1/2}\|_2\le \kappa^{1/2}(\Sigma_Z)\|\Delta\Theta\|_2
			$$
			and 
			$$
			\|\Delta\Theta\|_2=\|\Sigma_Z^{1/2}\Omega\Sigma_Z^{-1/2}\|_2\le \kappa^{1/2}(\Sigma_Z)\|\Omega\|_2.
			$$
			The corresponding matrix $\Delta\Theta$ belongs to $\mathcal M_p(R,\gamma)$ if $\delta<\gamma\kappa^{-1/2}(\Sigma_Z)$ since in this case $\|\Delta\Theta\|_{\mathrm{op}}\le\|\Delta\Theta\|_2\le \kappa^{1/2}(\Sigma_Z)\delta< \gamma$.
			\item[\it Step 5. Bound on the second moment of likelihood ratio.]\ \\
			Using exactly the same reasoning as in the proof of Theorem~1 in \citep{Carpentier2015} (see p. 2686) we can show that 
			$$
			\E_{(\Omega,\tilde \Omega) }\exp\left \{Tq^2(\tau/T)\tr(\Omega\tilde \Omega^{\T}) \right\} \le 1+\frac{2}{\dfrac{p^2}{2T^2q^4(\tau/T)\delta^4}-1}.
			$$
			Thus \eqref{bound_product_det} implies that 
			$$
			\E_0[L^2]\le e^{\displaystyle{C_1\delta^2+C_2 \delta^4+16Tq^2\Bigl(\frac\tau T\Bigr)\delta^4}}
			\left[1+2 \Bigl(\dfrac{p^2}{2T^2q^4(\tau/T)\delta^4}-1\Bigr)^{-1}\right].
			$$
			In order to analyze the decay of each term in this upper  bound   we consider two asymptotic regimes, $p\lesssim\sqrt T$ and  $p\gtrsim\sqrt T$. 
			\begin{enumerate}
				\item In the case of  $p\lesssim\sqrt T$ and  $q^2(\tau/T)\delta^2=o(p/T)$ we have  $\dfrac{p^2}{T^2q^4(\tau/T)\delta^4}\to \infty$, as $p,T\to\infty$. Moreover,  as $\min(\tau/T,1-\tau/T)\ge h_*>0$, we have that 
				$$
				Tq^2\Bigl(\frac\tau T\Bigr)\delta^4=\frac{T}{q^2\Bigl(\dfrac\tau T\Bigr)} \left(q\Bigl(\frac\tau T\Bigr)\delta\right)^4 =\frac{o(p^2/T)}{q^2\Bigl(\dfrac\tau T\Bigr)}=o(1) \quad \mbox{as $p,T\to\infty$}
				$$
				Since $\min(\tau/T,1-\tau/T)\ge h_*>0$, and $p\lesssim T$,  $\delta^2=o(p/T)=o_T(T^{-1/2})$ as $p,T\to\infty$.  
				\item In the case of $p\gtrsim\sqrt T$ and $q^2(\tau/T)\delta^2=o(T^{-1/2})$ we have  $\dfrac{p^2}{T^2q^4(\tau/T)\delta^4}\to \infty$, and 
				$$
				Tq^2\Bigl(\frac\tau T\Bigr)\delta^4=\frac{T}{q^2\Bigl(\frac\tau T\Bigr)} \left(q\Bigl(\frac\tau T\Bigr)\delta\right)^4 =\frac{o(T^{-1})}{q^2\Bigl(\frac\tau T\Bigr)}=o_T(1) \quad \mbox{as $p,T\to\infty$}
				$$
				Similary, using $\min(\tau/T,1-\tau/T)\ge h_*>0$, we obtain $\delta^2=o_T(T^{-1/2})$.
			\end{enumerate}
			In both cases we have  $\E_0[L^2]\le  1+o(1)$ as $p,T\to\infty$ and 
			$$
			\kappa(\Sigma_Z)\mathcal E(\tau,\Theta_1,\Theta_2) =\kappa(\Sigma_Z)q^2\Bigl(\frac \tau T\Bigr)\|\Delta\Theta\|^2_2	\ge q^2\Bigl(\frac \tau T\Bigr) \delta^2=o\Bigl(\frac pT\wedge T^{-1/2}\Bigr).
			$$
		\end{description}	
		Thus the rate satisfies 
		$$
		\mathcal R_{p,T} = \kappa^{-1/2}(\Sigma_Z) o\Bigl(\sqrt{\frac pT }\wedge T^{-1/4}\Bigr)
		$$ 
		and the theorem follows.
	\end{proof}
	The following lemma provides the upper bound on the determinant of the matrix $I+M$.
	
	\begin{lemma}\label{lem:dets}
		Let $\Omega$, $\tilde\Omega\in \bR^{p\times p}$ be the matrices of rank $R$ with the same Frobenius norm, $\|\Omega\|_2=\|\tilde\Omega\|_2=\delta$, $0<\delta<1$ and with the operator norm bounded by $\gamma$. Let  $L=\Bigl(1-\dfrac\tau T\Bigr)( \Omega+\tilde\Omega)$, $K=-\Bigl(\dfrac\tau T\Bigr)( \Omega+\tilde\Omega)$, $W=\Omega^{\T}\Omega+\tilde\Omega^{\T}\tilde\Omega$ and the matrices $A_k$, $k=0,\dots,T-1$ be defined in~\eqref{eqs:A}.
		Then 
		\begin{enumerate}
			\item For any $0<\delta<1/4$ and any $1\le k\le T-1$,  we have  $\|A_k\|_{\mathrm{op}}\le 4\delta^2$,   and 
			\begin{equation}\label{bound:detAk}
			(\det(I_p+A_k))^{-1/2}\le  \begin{cases} 
			\exp\Bigl\{\Bigl(1-\dfrac\tau T\Bigr)^2 \tr(\Omega^{\T}\tilde \Omega) + 16
			\Bigl(1-\dfrac\tau T\Bigr)^2\delta^4\Bigr\},& k=1,\dots,\tau-1\\
			\\
			\exp\Bigl\{\Bigl(\dfrac\tau T\Bigr)^2 \tr(\Omega^{\T}\tilde \Omega) + 16\Bigl(\dfrac\tau T\Bigr)^2\delta^4\Bigr\},& k=\tau,\dots,T-1.	
			\end{cases}
			\end{equation}
			\item For any $0<\delta<1/4$ and 
			$\delta^2<\frac 12 \left(4+3\kappa^2(\Sigma_Z)\right)^{-1}$ we have 
			\begin{align}
			(\det(I_p+A_0))^{-1/2}&\le  \exp\left\{\Bigl(1-\dfrac\tau T\Bigr)^2 \tr(\Omega^{\T}\tilde \Omega) \right.\nonumber \\ &\left.+  3 \Bigl(1-\dfrac\tau T\Bigr)^2\kappa^2(\Sigma_Z)\delta^2 +12\Bigl(1-\dfrac\tau T\Bigr)^2(1+\kappa^2(\Sigma_Z))^2\delta^4  \right\}.\label{bound:detAT}
			\end{align}
		\end{enumerate}
	\end{lemma}
	
	\begin{proof}
		To bound from below the determinants $\det(I_p+A_{k})$, we  use the result of Lemma~\ref{lem:detbound} from~\citep{Rump2018}:
		$$
		\det(I_p+A_k)\ge  \exp\Bigl\{\tr(A_k)-\frac{\frac12\|A_k\|_2^2}{1-\rho(A_k)}\Bigr\} \ge \exp\Bigl\{\tr(A_k)-\frac{\frac12\|A_k\|_2^2}{1-\|A_k\|_{\mathrm{op}}}\Bigr\},
		$$ 
		under the condition that $\|A_k\|_{\mathrm{op}}<1$. 
		Let us start with $A_{T-1}$. We have 
		$$
		A_{T-1}=\Bigl(\dfrac\tau T\Bigr)^2 (\Omega^{\T}\Omega+\tilde\Omega^{\T}\tilde\Omega) -\Bigl(\dfrac\tau T\Bigr)^2 (\Omega+\tilde\Omega)^{\T} (\Omega+\tilde\Omega)=-\Bigl(\dfrac\tau T\Bigr)^2(\Omega^{\T}\tilde \Omega+\tilde\Omega^{\T}\Omega)
		$$
		We have $\tr(\Omega^{\T}\tilde \Omega+\tilde\Omega^{\T}\Omega)=2 \tr(\Omega^{\T}\tilde \Omega)$ and 
		\begin{equation}\label{eq:normOmegaOmegatilda}
		\|\Omega^{\T}\tilde \Omega+\tilde\Omega^{\T}\Omega\|_{\mathrm{op}}\le \|\Omega^{\T}\tilde \Omega+\tilde\Omega^{\T}\Omega\|_2\le  2\|\Omega\|_2\|\tilde \Omega\|_2=2\delta^2.
		\end{equation}
		Consequently, for any $0<\delta<1/\sqrt 2$, 
		\begin{align*}
		\det(I_p+A_{T-1})&=\det\Bigl(I_p-\Bigl(\dfrac\tau T\Bigr)^2(\Omega^{\T}\tilde \Omega+\tilde\Omega^{\T}\Omega) \Bigr)\\
		&\ge  \exp\Bigl\{-\Bigl(\dfrac\tau T\Bigr)^2\tr(\Omega^{\T}\tilde \Omega+\tilde\Omega^{\T} \Omega)-\frac{\frac12\Bigl(\dfrac\tau T\Bigr)^4\|(\Omega^{\T}\tilde \Omega+\tilde\Omega^{\T} \Omega\|_2^2}{1-\Bigl(\dfrac\tau T\Bigr)^2\|\Omega^{\T}\tilde \Omega+\tilde\Omega^{\T} \Omega\|_{\mathrm{op}}}\Bigr\}\\
		&\ge \exp\Bigl\{ -2 \Bigl(\dfrac\tau T\Bigr)^2\tr(\Omega^{\T}\tilde \Omega)-\frac{2\Bigl(\dfrac\tau T\Bigr)^4\delta^4}{1-2\delta^2}\Bigr\}, 
		\end{align*}
		where we used the fact that $\tau/T\le 1$. 
		
		Let us turn to $\det(I_p+A_{T-k})$, where $k=1,\dots,T-\tau-1$.  We will show by induction that $\|A_{T-k}\|_{\mathrm{op}}\le 4\delta^2$ if $0<\delta<1/4$ . Note that $\|A_{T-1}\|_{\mathrm{op}}\le 2\delta^2$. We have 
		\begin{align*}
		A_{T-k} &= \Bigl(\dfrac\tau T\Bigr)^2 (\Omega^{\T}\Omega+\tilde\Omega^{\T}\tilde\Omega) -\Bigl(\dfrac\tau T\Bigr)^2 (\Omega+\tilde\Omega)^{\T}(I_p+A_{T-k+1})^{-1} (\Omega+\tilde\Omega)\\
		&=-\Bigl(\dfrac\tau T\Bigr)^2(\Omega^{\T}\tilde \Omega+\tilde\Omega^{\T}\Omega) - \Bigl(\dfrac\tau T\Bigr)^2 (\Omega+\tilde\Omega)^{\T}\left(\sum_{j=1}^\infty (-1)^j A_{T-k+1}^j \right)(\Omega+\tilde\Omega).
		\end{align*}
		Suppose that  $\|A_{T-k+1}\|_{\mathrm{op}}<4\delta^2<1/4$, where $0<\delta<1/4$. Then using~\eqref{eq:normOmegaOmegatilda}, we get
		\begin{align*}
		\|A_{T-k}\|_{2}&\le 2\Bigl(\dfrac\tau T\Bigr)^2\delta^2
		+\Bigl(\dfrac\tau T\Bigr)^2 \|\Omega+\tilde\Omega\|_{2}^2 \sum_{j=1}^\infty \|A_{T-k+1}\|_{\mathrm{op}}^j\\
		&\le \Bigl(\dfrac\tau T\Bigr)^2\left(2\delta^2+ 4\delta^2\frac{\|A_{T-k+1}\|_{\mathrm{op}}}{1-\|A_{T-k+1}\|_{\mathrm{op}}}\right)\le 2\delta^2\Bigl(\dfrac\tau T\Bigr)^2\frac{1+\|A_{T-k+1}\|_{\mathrm{op}}}{1-\|A_{T-k+1}\|_{\mathrm{op}}}\le 4\Bigl(\dfrac\tau T\Bigr)^2\delta^2.
		\end{align*}
		Consequently, $\|A_{T-k}\|_{\mathrm{op}}\le 	\|A_{T-k}\|_{2} \le 4\Bigl(\dfrac\tau T\Bigr)^2\delta^2$.
		
		To bound the trace of $A_{T-k}$ we use the fact that for any two square matrices $A$ and $B$, $\tr(AB)\le \|A\|_2\|B\|_2$. Then, for $0<\delta<1/4$, since $\|A_{T-k+1}\|_{\mathrm{op}}\le 4\delta^2$,
		\begin{align*}
		\tr(A_{T-k})&=- 2\Bigl(\dfrac\tau T\Bigr)^2 \tr(\Omega^{\T}\tilde\Omega)-\tr\left(\Bigl(\dfrac\tau T\Bigr)^2(\Omega+\tilde\Omega)^{\T} \Bigl(\sum_{j=1}^\infty (-1)^jA_{T-k+1}^j\Bigr) (\Omega+\tilde\Omega)\right)\\
		&\ge- 2\Bigl(\dfrac\tau T\Bigr)^2 \tr(\Omega^{\T}\tilde\Omega)-\Bigl(\dfrac\tau T\Bigr)^2\|\Omega+\tilde\Omega\|_2^2 \Big\| \sum_{j=1}^\infty (-1)^jA_{T-k+1}^j\Big\|_{\mathrm{op}} \\
		&\ge - 2\Bigl(\dfrac\tau T\Bigr)^2 \tr(\Omega^{\T}\tilde\Omega) - 
		\Bigl(\dfrac\tau T\Bigr)^2\|\Omega+\tilde\Omega\|_2^2 \frac{\|A_{T-k+1}\|_{\mathrm{op}}}{1-\|A_{T-k+1}\|_{\mathrm{op}}} \\
		&\ge - 2\Bigl(\dfrac\tau T\Bigr)^2 \tr(\Omega^{\T}\tilde\Omega)- 4\Bigl(\dfrac\tau T\Bigr)^2\delta^2 \dfrac {4\delta^2}{1-4\delta^2}\\
		&\ge  -2\Bigl(\dfrac\tau T\Bigr)^2 \tr(\Omega^{\T}\tilde\Omega) -\frac{64}3\Bigl(\dfrac\tau T\Bigr)^2\delta^4.
		\end{align*}
		Thus, using Lemma~\ref{lem:detbound} we obtain for $0<\delta<1/4$ that 
		\begin{align*}
		\det(I+A_{T-k})&\ge  \exp\Biggl\{\tr(A_{T-k})-\frac{\frac12\|A_{T-k}\|_2^2}{1-\|A_{T-k}\|_{\mathrm{op}}}\Biggr\}\\
		&\ge \exp\left\{ -2\Bigl(\dfrac\tau T\Bigr)^2 \tr(\Omega^{\T}\tilde\Omega)-\frac{64}3\Bigl(\dfrac\tau T\Bigr)^2\delta^4 -\frac{8\Bigl(\dfrac\tau T\Bigr)^2\delta^4}{1-4\delta^2} \right\}\\
		&\ge \exp\left\{  -2\Bigl(\dfrac\tau T\Bigr)^2\tr(\Omega^{\T}\tilde\Omega)- 32\Bigl(\dfrac\tau T\Bigr)^2\delta^4 \right\}.
		\end{align*}
		Following the same approach, we obtain the same bound for $\det(I_p+A_{\tau})$. For $\det(I_p+A_{\tau-k})$, $k=1,\dots,\tau-1$ we have
		$$
		\det(I+A_{T-k})\ge  \exp\left\{  -2\Bigl(1-\dfrac\tau T\Bigr)^2\tr(\Omega^{\T}\tilde\Omega)- 32\Bigl(1-\dfrac\tau T\Bigr)^2\delta^4\right\}.
		$$
		Let us bound $\det(I+A_0)$. We have 
		\begin{align*}
		\|A_0\|_2&\le \|\Bigl(1-\frac{\tau}T\Bigr)^2(\Omega^{\T}\Omega+\tilde\Omega^{\T}\tilde\Omega)-L^{\T}(I_p+A_{1})^{-1} L\|_2\\
		&+\|\Sigma_Z^{1/2} (\Sigma_1^{-1}-\Sigma_Z^{-1}) \Sigma_Z^{1/2}\|_2+\|\Sigma_Z^{1/2} (\tilde\Sigma_1^{-1}-\Sigma_Z^{-1}) \Sigma_Z^{1/2}\|_2\\
		&\le 4\Bigl(1-\frac{\tau}T\Bigr)^2\delta^2+\|\Sigma_Z^{1/2} (\Sigma_1^{-1}-\Sigma_Z^{-1}) \Sigma_Z^{1/2}\|_2+\|\Sigma_Z^{1/2} (\tilde\Sigma_1^{-1}-\Sigma_Z^{-1}) \Sigma_Z^{1/2}\|_2
		\end{align*}
		From the Lyapunov equation $\Sigma_1=\Bigl(1-\frac{\tau}T\Bigr)^2\Delta\Theta\Sigma_1\Delta\Theta^{\T}+\Sigma_Z$ it follows that 
		$$
		\Sigma_Z^{1/2} (\Sigma_1^{-1}-\Sigma_Z^{-1}) \Sigma_Z^{1/2}=\Sigma_Z^{-1/2}(\Sigma_Z-\Sigma_1)\Sigma_1^{-1}\Sigma_Z^{1/2}	=-\Bigl(1-\frac{\tau}T\Bigr)^2\Sigma_Z^{1/2}\Delta\Theta\Sigma_1\Delta\Theta^{\T}\Sigma_1^{-1}\Sigma_Z^{1/2}	
		$$
		and
		\begin{align*}
		\|\Sigma_Z^{1/2} (\Sigma_1^{-1}-\Sigma_Z^{-1}) \Sigma_Z^{1/2}\|_2 &=\Bigl(1-\frac{\tau}T\Bigr)^2\|\Sigma_Z^{-1/2}\Delta\Theta\Sigma_1\Delta\Theta^{\T}\Sigma_1^{-1}\Sigma_Z^{1/2}	\|_2\\
		&=\Bigl(1-\frac{\tau}T\Bigr)^2\|\Omega \Sigma_Z^{-1/2}\Sigma_1 \Sigma_Z^{-1/2}\Omega^{\T}\Sigma_Z^{1/2}\Sigma_1^{-1}\Sigma_Z^{1/2}\|_2\\
		&\le\Bigl(1-\frac{\tau}T\Bigr)^2 \|\Omega\|_2^2 \kappa(\Sigma_1) \kappa(\Sigma_Z).
		\end{align*}
		Lemma~\ref{lem:stab_lyapunov} and $\|\Delta\Theta\|_{\mathrm{op}}\le \kappa^{1/2}(\Sigma_Z)\|\Omega\|_{\mathrm{op}}$ imply that for $\delta<\frac12\kappa^{-1/2}(\Sigma_Z)$, 
		$$
		\kappa(\Sigma_1)\le \kappa(\Sigma_Z) \frac{1+\Bigl(1-\frac{\tau}T\Bigr)^2\|\Delta\Theta\|_{\mathrm{op}}^2}{1-\Bigl(1-\frac{\tau}T\Bigr)^2\|\Delta\Theta\|_{\mathrm{op}}^2}\le   \kappa(\Sigma_Z)\frac{1+\Bigl(1-\frac{\tau}T\Bigr)^2\kappa(\Sigma_Z)\delta^2}{1-\Bigl(1-\frac{\tau}T\Bigr)^2\kappa(\Sigma_Z)\delta^2}\le 3\kappa(\Sigma_Z)
		$$
		Therefore, for any $\delta<\frac 12 \kappa^{-1/2}(\Sigma_Z)$,
		$$
		\|\Sigma_Z^{1/2} (\Sigma^{-1}-\Sigma_Z^{-1}) \Sigma_Z^{1/2}\|_2
		\le 3\Bigl(1-\frac{\tau}T\Bigr)^2 \kappa^2(\Sigma_Z)\delta^2.
		$$
		Thus 
		\begin{equation}\label{eq:normA0}
		\|A_0\|_{\mathrm{op}}\le\|A_0\|_2\le \Bigl(1-\frac{\tau}T\Bigr)^2(4+ 3\kappa^2(\Sigma_Z))\delta^2.
		\end{equation}
		Next,
		\begin{align*}
		\tr(A_0)&= \tr\Bigl(\Bigl(1-\frac{\tau}T\Bigr)^2(\Omega^{\T}\Omega+\tilde\Omega^{\T}\tilde\Omega)-L^{\T}(I_p+A_{1})^{-1} L\Bigr)\\
		&\quad\quad+\tr(\Sigma_Z^{1/2} (\Sigma_1^{-1}-\Sigma_Z^{-1}) \Sigma_Z^{1/2})+\tr(\Sigma_Z^{1/2} (\tilde\Sigma_1^{-1}-\Sigma_Z^{-1}) \Sigma_Z^{1/2})\\
		&= \tr\Bigl(\Bigl(1-\frac{\tau}T\Bigr)^2(\Omega^{\T}\Omega+\tilde\Omega^{\T}\tilde\Omega)-L^{\T}(I_p+A_{1})^{-1} L\Bigr)\\
		&\quad\quad-\Bigl(1-\frac{\tau}T\Bigr)^2\tr(\Sigma_Z^{-1/2}\Delta\Theta\Sigma_1\Delta\Theta^{\T}\Sigma_1^{-1} \Sigma_Z^{1/2})-\Bigl(1-\frac{\tau}T\Bigr)^2\tr(\Sigma_Z^{-1/2}\Delta\tilde\Theta\tilde\Sigma_1\Delta\tilde\Theta^{\T}\tilde\Sigma_1^{-1} \Sigma_Z^{1/2}).
		\end{align*}
		We have for $\delta<\frac 12 \kappa^{-1/2}(\Sigma_Z)$,
		\begin{align*}
		\tr(\Sigma_Z^{-1/2}\Delta\Theta\Sigma_1\Delta\Theta^{\T}\Sigma_1^{-1} \Sigma_Z^{1/2}) &=\tr(\Delta\Theta\Sigma_1\Delta\Theta^{\T}\Sigma_1^{-1} ) \le \|\Delta\Theta\Sigma_1\|_2\|\Delta\Theta^{\T}\Sigma_1^{-1}\|_2\\
		&\le \|\Delta\Theta\|_2^2 \|\Sigma_1\|_{\mathrm{op}}\|\Sigma_1^{-1}\|_{\mathrm{op}} = \|\Sigma_Z^{1/2}\Omega\Sigma_Z^{-1/2}\|_2^2\kappa(\Sigma_1)\\
		&\le \delta^2 \kappa(\Sigma_Z)\kappa(\Sigma_1) \le 3\delta^2 \kappa^2(\Sigma_Z).
		\end{align*}
		Thus, using the same reasoning as above to bound the first term of $\tr(A_0)$, we get 
		$$
		\tr(A_0)\ge  -2\Bigl(1-\dfrac\tau T\Bigr)^2 \tr(\Omega^{\T}\tilde\Omega) -\frac{64}3\Bigl(1-\dfrac\tau T\Bigr)^2\delta^4 -6\Bigl(1-\frac{\tau}T\Bigr)^2\delta^2  \kappa^2(\Sigma_Z).
		$$
		If $\delta^2<\frac 12 \Bigl(4+3\kappa^2(\Sigma_Z)\Bigr)^{-1}$, then using~\eqref{eq:normA0} we obtain the bound
		$$
		\frac{\frac12\|A_0\|_2^2}{1-\|A_0\|_{\mathrm{op}}} \le \Bigl(1-\dfrac\tau T\Bigr)^2\delta^4  \Bigl(4+3\kappa^2(\Sigma_Z)\Bigr)^2.
		$$
		Taking into account that fact that $\frac14 \kappa^{-1}(\Sigma_Z)>\frac 12 \Bigl(4+3\kappa^2(\Sigma_Z)\Bigr)^{-1}$ and that $64/3+ \Bigl(4+3\kappa^2(\Sigma_Z)<24(1+\kappa^2(\Sigma_Z))^2$, we obtain that for $\delta^2<\frac 12 \Bigl(4+3\kappa^2(\Sigma_Z)\Bigr)^{-1}$,
		\begin{align*}
		\det(I+A_0)&\ge  \exp\Bigl\{\tr(A_0)-\frac{\frac12\|A_0\|_2^2}{1-\|A_0\|_{\mathrm{op}}}\Bigr\}\\
		&\ge   \exp\left\{-2\Bigl(1-\dfrac\tau T\Bigr)^2 \tr(\Omega^{\T}\tilde\Omega) -\frac{64}3\Bigl(1-\dfrac\tau T\Bigr)^2\delta^4 \right.\\
		&\quad\quad\qquad\left.-6\Bigl(1-\frac{\tau}T\Bigr)^2\delta^2  \kappa^2(\Sigma_Z)- \Bigl(1-\dfrac\tau T\Bigr)^2\delta^4  \Bigl(4+3\kappa^2(\Sigma_Z)\Bigr)^2\right\}\\
		&\ge   \exp\left\{-2\Bigl(1-\dfrac\tau T\Bigr)^2 \tr(\Omega^{\T}\tilde\Omega)  -6\Bigl(1-\frac{\tau}T\Bigr)^2\delta^2  \kappa^2(\Sigma_Z) - 24 \Bigl(1-\dfrac\tau T\Bigr)^2\delta^4 (1+\kappa^2(\Sigma_Z))^2\right\}.
		\end{align*}
		and the lemma follows. 
	\end{proof}

	\section{Results on the estimation of the transition matrix}
	
	The following result is a reformulation of Corollary 4 in \citep{negahban2011}.
	\begin{prop}\label{cor:bound_risk_estimation2}
		Let $(X_0,\dots,X_n)$ be  $n$ consecutive realizations of a $\mathrm{VAR}_p(\Theta,\Sigma_Z)$ stationary  process~\eqref{eq:var_def} with the transition matrix $\Theta$ of rank at most $R$ and satisfying  $\|\Theta\|_{\op}=\gamma<1$. Assume that $n\ge p$. 
		
		Let $\widehat \Theta$ be a solution of the SDP~\eqref{eq:Theta_hat} with the regularization parameter $\lambda_n$ given by~\eqref{eq:lambda}.  Then there exist a universal  constants $c^*,c_2,c_3$ such that
		\begin{equation}\label{eq:upper_bound_risk_fro}
		\| \widehat{\Theta}-\Theta\|_{2}^{2}\leq c^*\dfrac{\kappa^{2}(\Sigma)}{(1-\gamma)^{2}}\dfrac{Rp}{n}.
		\end{equation}
		with probability greater than $1-c_2\exp(-c_3p)$.
	\end{prop}
	
	The next proposition shows the consistency of the transition matrix estimator in the operator norm. 
	\begin{prop}[Operator norm consistency]\label{thm:operator_norm}
		Let $(X_0,\dots,X_n)$ be consecutive realizations of a $\mathrm{VAR}_p(\Theta,\Sigma_Z)$ stationary  process~\eqref{eq:var_def} with the transition matrix $\Theta$ of rank at most $R$ satisfying  $\|\Theta\|_{\op}=\gamma<1$. Assume that $n\ge p$.  Let $\widehat \Theta$ be a solution of  SDP~\eqref{eq:Theta_hat} with the regularization parameter defined in~\eqref{eq:lambda}.  Then with probability at least $1-2/p$
		$$
		\|\widehat \Theta-\Theta\|_{\mathrm{op}} \le 8 c_1 \frac{\kappa(\Sigma)}{1-\gamma} \sqrt{\frac pn},
		$$	
		where $c_1$ is an absolute constant. 
	\end{prop}
	
	\begin{proof}
		We have
		$$
		\frac 1n \sum_{i=0}^{n-1} \|X_{i+1}-MX_i\|_2^2=\frac1 n  \|\mathfrak{Y}_n-M \mX_n\|_2^2.
		$$
		Its the derivative with respect to $M$  is equal to $-\dfrac2n (\mathfrak{Y}- M\mX_n)\mX_n^{\T}$. On the other hand, for any subgradient matrix $W\in \partial \|\widehat \Theta\|_*$, we are guaranteed $\|W\|_{\mathrm{op}}\le 1$. By the Karush-Kuhn-Tucker condition, any solution $\widehat \Theta$ of~\eqref{eq:Theta_hat} must satisfy
		$$
		-\dfrac2n (\mathfrak{Y}_n- \widehat \Theta\mX_n)\mX_n^{\T} +\lambda_1 W=0.
		$$
		Note that $\mathfrak{Y}_n=\Theta \mX_n+\mZ_n$, then we can rewrite the KKT condition as 
		$$
		\dfrac2n ((\widehat \Theta -\Theta)\mX_n-\mZ_n)\mX_n^{\T} +\lambda_1 W=0
		$$
		It implies that
		$$
		\widehat \Theta -\Theta= \left(\dfrac2n \mZ_n\mX_n^{\T}-\lambda_1W\right) \left(\dfrac2n \mX_n\mX_n^{\T}\right)^{-1}.
		$$
		Thus, from Lemma~\ref{lem:upper_bound_sampling_operator} and Lemma~\ref{lem:bound_stoch_term2} it follows that 
		\begin{align*}
		\|\widehat \Theta -\Theta\|_{\mathrm{op}}&=\Bigl\| \left(\dfrac2n \mZ_n\mX_n^{\T}-\lambda_1W\right) \left(\dfrac2n \mX_n\mX_n^{\T}\right)^{-1}\Bigr\|_{\mathrm{op}}
		\le  \left(2 \Bigl\|\frac 1n \mZ_n\mX_n^{\T}\Bigr\|_{\mathrm{op}}+\lambda_1\right) \Bigl\|\left(\dfrac2n\mX_n\mX_n^{\T}\right)^{-1}\Bigr\|_{\mathrm{op}}\\
		&\le 2c_2^*\frac{\|\Sigma\|_{\mathrm{op}}}{1-\gamma}\sqrt{\frac pn}\sigma_{\min}^{-1}\Bigl(\dfrac1n\mX_n\mX_n^{\T}\Bigr)\le 8 c_2^* \frac{\kappa(\Sigma)}{1-\gamma} \sqrt{\frac pn}
		\end{align*}
		with probability at least $1-2/p$. 
	\end{proof}

\section{Auxiliary results}

Throughout this section $\mX_n=(X_1,\dots,X_n)$ denotes a $p\times n$ matrix of $n$ consecutive realizations $X_t$ of a $\mathrm{VAR}_p(\Theta,\Sigma_Z)$ process defined in~\eqref{eq:var_def}.  The corresponding noise matrix $\mathfrak{Z}_n=(Z_1,\dots,Z_n)$ is a $p\times n$ matrix  with i.i.d. Gaussian centered columns $Z_t$ with the covariance matrix $\Sigma_Z$.

\begin{lemma}\label{lem:triangle_inequality_trace_norm}
	Assume that we solve \eqref{eq:Theta_hat1} - \eqref{eq:TestStat} with  $\lambda^t_1\geq 3\| \mathfrak{Z}_{t}\mX_{t}^{\T}\|_{\mathrm{op}}$ and  $\lambda^t_2\geq 3\| \mathfrak{Z}_{\ge t}\mX_{\ge t}^{\T}\|_{\mathrm{op}}$. Then, for $i=1,2$,   
	we have 
	\begin{align}\label{inequality_penalty1}
	\| \mathbf {Pr}^{\bot}_{\Theta_{i}}\bigl[\Theta-\widehat{\Theta}_i(t)\bigr ]\|_{*}\leq 5\| \mathbf {Pr}_{\Theta_{i}}\bigl [\Theta-\widehat{\Theta}_i(t)\bigr]\|_{*}.
	\end{align}
\end{lemma}

\begin{proof}
	We will prove \eqref{inequality_penalty1} for $i=1$. The proof for  $i=2$ is completely analogous.
	Using the triangle inequality, for any two matrices $A$ and $B$ we easily get
	\begin{align}\label{triangle_inequality_trace_norm}
	\| A\|_{*}-\| B\|_{*}
	\leq 
	\| \mathbf {Pr}_{A}\left [ A-B\right ]\|_{*}-\| \mathbf {Pr}^{\bot}_{A}\left [A-B\right ]\|_{*}.
	\end{align}
	Indeed, 
	\begin{align*}
	\| B\|_{*}&=\| B+A-A\|_{*}	
	=\| A+\mathbf {Pr}_{A}\left [B-A\right ] +\mathbf {Pr}^{\bot}_{A}\left [ B-A\right ]\|_{*}
	\\&
	\geq\| A +\mathbf {Pr}^{\bot}_{A}\left [ B-A\right ]\|_{*}	-\| \mathbf {Pr}_{A}\left [B-A\right ]\|_{*}
	\\&
	=\| A \|_{*}+\|\mathbf {Pr}^{\bot}_{A}\left [ B-A\right ]\|_{*}	-\| \mathbf {Pr}_{A}\left [B-A\right ]\|_{*}
	\end{align*}
	which implies \eqref{triangle_inequality_trace_norm}.
	
	Now, by convexity of $\| \mathfrak{Y}_{t}-\Theta \mX_{t}\|^{2}_{2}$ we have
	\begin{align*}
	\| \mathfrak{Y}_{t}-\widehat{\Theta}_1^t \mX_{t}\|^{2}_{2}-	\| \mathfrak{Y}_{t}-{\Theta} \mX_{t}\|^{2}_{2}&\geq -2\left\langle \mathfrak{Y}_{t}-{\Theta} \mX_{t},\left (\widehat{\Theta}_1^t-\Theta\right )\mX_{t} \right \rangle
	\\&
	=-2\left\langle \mathfrak{Z}_{t}\mX_{t}^{\T},\widehat{\Theta}_1^t-\Theta \right \rangle 
	\\&
	\geq  -2\| \widehat{\Theta}_1^t-\Theta\|_{*}\| \mathfrak{Z}_{t}\mX_{t}^{\T}\|_{\mathrm{op}}\geq
	-\dfrac{2}{3}\lambda^{t}_1\| \widehat{\Theta}_1^t-\Theta\|_{*}.
	\end{align*}
	On the other hand, using the definition of $\widehat{\Theta}_1$ we compute
	\begin{align*}
	\lambda^{t}_1\| \widehat{\Theta}_{1}(t)\|_{*}-\lambda^{t}_1\| \Theta\|_{*}&\leq  \| \mathfrak{Y}_{t}-{\Theta} \mX_{t}\|^{2}_{2}-	\| \mathfrak{Y}_{t}-\widehat{\Theta}_1^t \mX_{t}\|^{2}_{2}	\leq \dfrac{2}{3}\lambda^{t}_1\| \widehat{\Theta}_1^t-\Theta\|_{*},
	\end{align*}
	which combined with \eqref{triangle_inequality_trace_norm} proves \eqref{inequality_penalty1}.
\end{proof}

The following result about the largest and smallest eigenvalue of empirical covariance matrix of the process~\eqref{eq:var} follows from Lemma~4 of \citep{negahban2011}.
\begin{lemma}\label{lem:upper_bound_sampling_operator}
	Let $\delta\in(0,1)$ and $c_0>0$ be a universal constant. The operator norm of the matrix  $\frac{1}n \mX_{n}\mX_{n}^{\T}$ is well controlled in terms of the covariance matrix $\Sigma$:
	\begin{equation}\label{eq:sigma_max}
	\P\left\{\sigma_{\max}\left (\frac1n \mX_{n}\mX_{n}^{\T}\right )
	\leq 	2\| \Sigma\|_{\mathrm{op}}\left [1+\dfrac{c_0}{1-\gamma}\left (\dfrac{p}{n}\vee 1-\dfrac{\log(\delta)}{n}\right )\right ]\right\} \ge 1-2\delta
	\end{equation}
	and, if $n>p$, for some constants $c_1,c_2>0$,
	\begin{equation}\label{eq:sigma_min}
	\P\left\{	\sigma_{\min}\left (\frac1n \mX_{n}\mX_{n}^{\T}\right )
	\geq \dfrac{\sigma_{\min}(\Sigma)}{4}	\right\}\ge 1-c_1 e^{-c_2 p}.
	\end{equation}
\end{lemma}
\begin{proof}
	We follow the proof of Lemma~4 from~\citep{negahban2011}. The only difference is that we allow the dimension $p$ to be greater than the number of observations, $p>n$. 
	
	Let $u \in \mathbb{R}^{p}$ with $\| u \|_2= 1$ and $Y \in \mathbb{R}^{n}$  be a random vector with elements $Y_{i}= \langle u,X_{i} \rangle$. By the covering argument from the proof of Lemma~4 from~\citep{negahban2011}, it can be shown that 
	$$
	\P\left\{\frac1n \sigma_{\max}(\mX_n\mX_n^T) >t \right\}\le 4^p \max_{u\in \mathcal A} \P \left\{\frac1 n \sum_{
		i=1}^n \langle u,X_{i} \rangle^2>t/2\right\},
	$$
	where $\mathcal A$ is a $1/2$-cover of  $S^{p-1}$.  Thus, we need to bound $\dfrac{1}{n} \| Y \|_{2}^{2}$. Note that $Y\sim\mathcal{N}_{n}(0,R)$ with the covariance matrix $R$ with the elements $R_{ij}=u^\T \Sigma (\Theta^\T)^{|i-j|}u$.  By the Hanson-Wright inequality, we have 
	$$
	\P \left(| \| Y \|_{2}^{2}-\mathrm{tr}(R) \mid  > x \right)\leq 2\exp\left[-c \min\left(\frac{x^{2}}{\| R \| _{2}^{2}},\frac{x}{\| R \| _{\mathrm{op}}}\right)\right].
	$$
	Let $x=\dfrac{2\| R \| _{\mathrm{op}}}{c}\Bigl((p\vee n)-\frac 12\log \delta\Bigr)$. Then 	
	$$
	\P \left(  \frac 1n \| Y \|_{2}^{2}  > \frac 1n\tr(R) + \dfrac{ 2\| R \|_{\mathrm{op}}}{cn}\left(p\vee n-\frac12\log \delta \right) \right)\leq 2 \delta e^{-2(p\vee n)}.
	$$
	Using the bounds   $\| R \|_{2}^{2} \leq n \| R \|_{\mathrm{op}}^{2}$,	$\| R \|_{\mathrm{op}} \leq \dfrac{2 \| \Sigma \|_{\mathrm{op}}}{1-\gamma}$, and $\tr(R)/n\le \|\Sigma\|_{\mathrm{op}}$, we obtain~\eqref{eq:sigma_max}.
\end{proof}

\begin{lemma}\label{lem:bound_stoch_term1}
	Let $\mathfrak{Z}_n$ be a  $p\times n$ matrix with $n$ i.i.d. Gaussian  columns with zero mean and covariance matrix $\Sigma_Z$.
	For any $\delta\in(0,1)$, there exists a universal constant $c_{*}$ independent of $p$, $n$ and $\Sigma_{Z}$ such  that
	$$
	\P\left\{  \| \mathfrak{Z}_{n}\|_{\op} \leq c_1^* \sqrt{\| \Sigma_{Z}\|_{\mathrm{op}} }\left ( \sqrt{n}+\sqrt{p}+\sqrt{\log(2/\delta)}\right )  \right\}\ge 1-\delta. 
	$$
	
	{\it Proof} follows from Theorem 4.4.5 in \cite{vershynin_2018}.
	
	%
	%
\end{lemma}

The following lemma reformulates  Lemma~5 in \citep{negahban2011}.
\begin{lemma}\label{lem:bound_stoch_term2}
	Let $\mX_n=(X_1,\dots,X_n)\in \bR^{p\times n}$ be a realization of $n$ observations of a $\mathrm{VAR}_p(\Theta,\Sigma)$ process~\eqref{eq:var_def} and $\mathfrak{Z}_n$ be the corresponding $p \times n$ matrix with i.i.d.  columns of Gaussian $\cN_p(0,\Sigma_Z)$ noise.
	There exist universal constants $c_2^*$, $k_1, k_2$ (independent of $p$, $n$, $\Sigma$ and $\Sigma_Z$) such that
	$$
	\P\left\{\frac{1}{n}\| \mathfrak{Z}_n \mX_n^{\T}\|_{\op}\geq \dfrac{c_2^*\|\Sigma\|_{\op}}{1-\gamma}\sqrt{\dfrac{p}{n}}
	\right \}\leq k_1\exp\left (-k_2p\right). 
	$$
\end{lemma}

The proof of the following lemma is based on the Hanson--Wright inequality.

\begin{lemma}\label{lem:low_bound_sampling_operator}
	Let $\delta\in(0,1)$ and $A \in \mathbb{R}^{p\times p}$. Let $X=(X_1,\dots,X_n)$ be a realization of a $\mathrm{VAR}_p(\Theta,\Sigma)$ process defined in~\eqref{eq:var_def}. Assume that for some universal constant $c>0$,
	\begin{equation}\label{cond:threshold_frobenius_norm}
	\frac{\|A\|_2}{\|A\|_{\mathrm{op}}} \geq 
	\dfrac{2(1+\gamma^{2})\sigma_{\max}(\Sigma)}{ \sigma_{\min}(\Sigma_{Z})\wedge \sigma_{\min}(\Sigma)  }\sqrt{\frac1c\dfrac{\log(2/\delta)}{n}}.
	\end{equation}
	Then for any $\delta>0$
	\begin{align*}
	\P\left [\frac{1}{n}\sum_{t=1}^{n}\| AX_t\| _{2}^{2}\geq \dfrac{ \sigma_{\min}(\Sigma_{Z})\wedge \sigma_{\min}(\Sigma) }{2(1+\gamma^{2})}\| A\|^{2}_{2}\right ]&\geq 1-\delta.
	\end{align*}	
\end{lemma}
\begin{proof}
	Recall that $X_{t+1}=\Theta X_t+Z_t$, where $Z_t$ are i.i.d. Gaussian vectors with the covariance matrix $\Sigma_Z$. We can write
	$$
	\sum_{t=1}^{n}\| AX_t\| _{2}^{2} =\|B \mathbb{X}\|_2^2,
	$$
	where $B$ is a $pn \times pn$ block-diagonal matrix with the identical diagonal blocks $A$ and $\mathbb{X}$ is the $pn$-dimensional Gaussian vector obtained by concatenation of the vectors $X_t$, $t=1,\dots,n$.
	Denote the covariance matrix of $\mathbb{X}$ by $\mathbb{\Sigma}$, thus 
	$\mathbb{X} \sim \mathcal{N}_{pn}(0, \mathbb{\Sigma})$. 
	
	It can be easily seen that  the covariance matrix of $\mathbb X$
	$$
	\mathbb{\Sigma} = \mathbb{E}(\mathbb{X}\mathbb{X}^{\T})=\Bigl(\mathbb{E}(X_{i} X_{j}^{\T})\Bigr)_{1 \leq i,j \leq n}
	$$
	is a symmetric block-Toeplitz  matrix with the blocks 
	$$
	\mathbb{E}(X_{i} X_{j}^{\T})= \Sigma\one_{\{i=j\}}+ \Theta^{i-j}\Sigma \one_{\{i> j\}} + 
	\Sigma\bigl( \Theta^{j-i}\bigr)^{\T}\one_{\{i<j \}}.
	$$
	Using the representation $\mathbb X=\mathbb \Sigma^{1/2} W$, where  $W \sim \mathcal{N}_{pn}(0, I_{pn})$ is a $pn$-dimensional Gaussian vectors with i.i.d. entries, we see that  $B\mathbb X=(B\sqrt{\mathbb{\Sigma}})W \sim \mathcal{N}_{pn}(0, B\mathbb{\Sigma}B^{\T})$. Note that $\mathbb{E}( \| B\mathbb X\|_2^2)= \mathbb E \| B\sqrt{\mathbb{\Sigma}}W \|_{2}^{2})=\tr( B\mathbb{\Sigma}B^{\T})$. Then, using the Hanson-Wright inequality we get
	\begin{equation}\label{HW1 inequality}
	\P \left( \Bigl | \| B \mathbb{X}_{T} \|_{2}^{2}-\tr (B\mathbb{\Sigma}B^{\T}) \Bigr |  > x \right)\leq 2\exp\left[-c \min\left(\frac{x^{2}}{\| B\mathbb{\Sigma}B^{\T} \| _{2}^{2}},\frac{x}{\| B\mathbb{\Sigma}B^{\T} \| _{\mathrm{op}}}\right)\right],
	\end{equation}
	where $c>0$ is a universal constant. 
	
	Note that  $\tr(\mathbb{\Sigma}B^{\T}B)\geq \sigma_{\min}(\mathbb{\Sigma})\tr(B^{\T}B)$, since $\tr((\mathbb{\Sigma}-\sigma_{\min}(\mathbb{\Sigma})I_{d})B^{\T}B) \geq 0$. 
	Thus, 
	$$
	\tr( B\mathbb{\Sigma}B^{\T} )\geq \sigma_{\min}(\mathbb{\Sigma}) n \| A \|_{2}^{2}.
	$$
	
	Set $x= \|\mathbb{\Sigma}\|_{\mathrm{op}} \|A\|_{\mathrm{op}} \|A\|_{2} \sqrt{\log(2/\delta)n/c}$. We have $\| B\mathbb{\Sigma}B^{\T}  \|_{\mathrm{op}} \leq \| \mathbb{\Sigma} \|_{\mathrm{op}} \| A \|_{\mathrm{op}}^{2}$.  Consequently, 
	the condition 
	\begin{equation}\label{ineq1:HW}
	\frac{\|A\|_2}{\|A\|_{\mathrm{op}}} \ge  \sqrt{\frac{\log( 2/\delta)}{nc}}
	\end{equation}
	implies $\dfrac{cx}{\| B\mathbb{\Sigma}B^{\T} \|_{\mathrm{op}}} > \log(2/\delta)$.
	On the other hand, we have 
	$$
	\| B\mathbb{\Sigma}B^{\T} \|_{2}^{2} \leq \| \mathbb{\Sigma}B^{\T}B \|_{2}^{2} \| \leq \| \mathbb{\Sigma}\|_{\mathrm{op}}^{2} \| B \|_{\mathrm{op}}^{2} \| B \|_{2}^{2} \leq    \| \mathbb{\Sigma}\|_{\mathrm{op}}^{2}\|A\|_{\mathrm{op}}^{2}n \|A \|_{2}^{2} 
	$$
	which implies $\dfrac{cx^{2}}{\| B\mathbb{\Sigma}B^{\T}  \| _{2}^{2}} \geq \log(2/\delta)$. 
	In addition, 
	\begin{equation}\label{ineq2:HW}
	\frac{\|A\|_2}{\|A\|_{\mathrm{op}}} \ge 2 \kappa(\mathbb{\Sigma})\sqrt{\dfrac{\log(2/\delta)}{nc}}
	\end{equation}
	immediately implies  that $\tr(B\mathbb{\Sigma}B^{\T})-x \geq \dfrac{\tr(B\mathbb{\Sigma}B^{\T})}{2}$.
	Thus, \eqref{HW1 inequality} implies 
	$$
	\P \left( \| B \mathbb{X}_{T} \|_{2}^{2}\le \frac { \sigma_{\min}(\mathbb \Sigma)}2 \|A\|_2^2\right)	\le \P \left( \| B \mathbb{X}_{T} \|_{2}^{2}\le \frac 12\tr (B\mathbb{\Sigma}B^{\T}) \right)\leq \delta
	$$
	We can bound the maximal eigenvalue of $\mathbb \Sigma$ using the decomposition $ \mathbb{\Sigma} = \Lambda+\Lambda^{\T}$  into the sum of a lower and upper triangular parts of  $\mathbb \Sigma$ where the matrices $\Lambda$ and $\Lambda^{\T}$ have the same diagonal blocks $\frac 12\Sigma$. Then  
	$
	\sigma_{\max}(\mathbb{\Sigma})=\sigma_{\max}(\Lambda+\Lambda^{\T})\leq 2 \sigma_{\max}(\Lambda)\leq \sigma_{\max}(\Sigma).
	$
	
	To bound $\sigma_{\min}(\mathbb\Sigma)$, note that $\sigma_{\min}(\mathbb{\Sigma})=(\sigma_{\max}(\mathbb{\Sigma}^{-1}))^{-1}$ where the inverse matrix $\Sigma$ has a simple block-tridiagonal with constant upper and lower diagonal blocks  $-\Theta^{\T} \Sigma_Z^{-1}$ and $-\Theta^{\T} \Sigma_Z^{-1}$, respectively. On the diagonal of this matrix we have the same first $T-1$ blocks $\Theta^{\T}\Sigma_{Z}^{-1}\Theta+\Sigma^{-1}$ and the last block is equal to $\Sigma_Z^{-1}$.
	Using the same reasoning as above, we see that 
	\begin{align*}
	\sigma_{\max}(\mathbb{\Sigma}^{-1})& \leq \max\Bigl( \|  \Theta^{\T}\Sigma_{Z}^{-1}\Theta+\Sigma^{-1}  \|_{\mathrm{op}}, \|\Theta^{\T}\Sigma_{Z}^{-1}\Theta+\Sigma_{Z}^{-1}\|_{\mathrm{op}}, \| \Sigma_{Z}^{-1}\|_{\mathrm{op}}\Bigr)\\
	&\le \gamma^2 \|\Sigma_Z^{-1}\|_{\mathrm{op}}+ \max(\|\Sigma_Z^{-1}\|_{\mathrm{op}},\|\Sigma^{-1}\|_{\mathrm{op}}).
	\end{align*}
	Therefore 
	$
	\sigma_{\min}( \mathbb{\Sigma}) \ge \dfrac{ \sigma_{\min}(\Sigma_{Z})\wedge \sigma_{\min}(\Sigma)  }{1+\gamma^{2}}
	$
	and 
	$$
	\kappa(\mathbb{\Sigma})\leq \dfrac{(1+\gamma^2)\sigma_{\max}(\Sigma)}{ \sigma_{\min}(\Sigma_{Z})\wedge \sigma_{\min}(\Sigma)}.
	$$
	Finally, these estimates together with condition \eqref{cond:threshold_frobenius_norm} imply \eqref{ineq1:HW} and \eqref{ineq2:HW}  and  the statement of the lemma follows.
\end{proof}

The following lemma is from~\citep{Rump2018}:
\begin{lemma}\label{lem:detbound}
	Let $A$ be real or complex $n\times n$ matrix. Then 
	\begin{enumerate}
		\item $|\det(I+A)|\le \exp\Bigl\{\mathrm{Re}(\tr(A))+\frac 12\|A\|_2^2 \Bigr\}$
		\item Suppose the eigenvalues $\lambda_k$ of $A$ satisfy $\mathrm{Re}(\lambda_k)>-1$, denote $\mu_k=\min(0,\mathrm{Re}(\lambda_k))$. Then 
		$$
		|\det(I+A)|\ge \exp\Bigl\{\mathrm{Re}(\tr(A))-\frac {\frac12\|A\|_2^2}{1+\min_k \mu_k} \Bigr\}
		$$
		\item Let $\rho(A)$ be spectral radius of $A$ and $\rho(A)<1$. Then 
		$$
		|\det(I+A)|\ge \exp\Bigl\{\mathrm{Re}(\tr(A))-\frac {\frac12\|A\|_2^2}{1-\rho(A)} \Bigr\}.
		$$
	\end{enumerate}
\end{lemma}

The following lemma is the result of Theorem~1.1 of~\cite{Tippett2000} reformulated in the notation of our paper. 
\begin{lemma}\label{lem:stab_lyapunov}
	Let $\Theta$ be a stable matrix, $\|\Theta\|_{\mathrm{op}}=\gamma<1$. Let $\Sigma=\Sigma^{\T}>0$ be the solution of the Lyapunov equation $\Sigma=\Theta\Sigma\Theta^{\T}+\Sigma_Z$, where $\Sigma_Z$ is symmetric positive definite.  Then 
	\begin{equation}\label{eq:cond_number}
	\kappa(\Sigma)\le \kappa(\Sigma_Z)\frac{1+\gamma^2}{1-\gamma^2}.
	\end{equation}
\end{lemma}

	\section{Additional simulation results}
	We consider the situation of uncorrelated Gaussian noise, $\Sigma_Z=\id_p$. We report the results of 100 simulations. 
	The transition matrices $\Theta_1$ and $\Theta_2$ are defined as follows:
	$$
	\Theta_1=UD_1U^{\T},\quad \Theta_2=U D_2 U^{\T},
	$$
	where $U$ is a unitary matrix, $D_1$ and $D_2$ are diagonal matrices with the largest absolute eigenvalues $\gamma_1$ and $\gamma_2$, respectively. Thus the Frobenius norm of the change is expressed in terms of the eigenvalues of $D_1$ and $D_2$: $\|\Delta\Theta\|_2=\|D_1-D_2\|_2$.
	
	
	\subsection{Quantile of the test statistic under the null}
	
	To simulate the quantile, we need to sample the test statistic $\mathcal G(t)$ under the null with the transition matrices estimated from the data. Since we do not know whether our data contains a change-point or not, we will sample the quantiles of $\mathcal G(t)$ based on the distribution $\mathrm{VAR}_p(\widehat\Theta_1,\Sigma_Z)$ and on  the distribution $\mathrm{VAR}_p(\widehat\Theta_2,\Sigma_Z)$, where $\widehat\Theta_1$ and $\widehat\Theta_2$ are estimations of $\Theta_1$ and $\Theta_2$ obtained from the first $\lfloor Th\rfloor $ and the last $\lfloor Th\rfloor$ observations of the process. To guarantee the stability of simulated processes, these two estimators are adjusted to have the operator norm to be equal to $\gamma_1$ and $\gamma_2$, respectively. Thus, we will have two quantiles of level $1-\alpha$, the first one is based on the hypothesis that the true parameter matrix is the one of the first $\lfloor Th\rfloor $ observations and the other one is based on the hypothesis that the true parameter matrix equals to the one of the last $\lfloor Th\rfloor $ observations. The quantile simulation from a  given data subsample $\breve{ \mathfrak X}$ is presented in Algorithm~\ref{algo:quantile}.
	
	\RestyleAlgo{ruled}
	\begin{algorithm}[htbp!]
		\caption{Simulation of $1-\alpha$ quantile of $\mathcal G(t)$}\label{algo:quantile}
		\KwData{$\breve{ \mathfrak X}$, $\Theta$, $\Sigma_Z$, $\alpha$, $t$}
		\KwResult{Quantile $q_{\alpha,t}$ of the distrbution of $\mathcal G(t)$}
		S=1500\; 
		s=1\;
		\While{$s<S$}
		{
			Sample $\tilde X_0$ from $\breve{\mathfrak X}$ with replacement\;
			Sample observations $\tilde {\mathfrak X}=(\tilde X_0,\tilde X_1,\dots,\tilde X_T)$ from  $\mathrm{VAR}_p(\Theta,\Sigma_Z)$\;
			Fing the SDP solutions $\widetilde \Theta_1^t=\arg\min\limits_{M\in\bR^{p\times p}} \varphi_1^t(\tilde {\mathfrak X},M)$
			and $\widetilde \Theta_2^t=\arg\min\limits_{M\in\bR^{p\times p}} \varphi_2^t(\tilde{\mathfrak X},M)$\;
			Calculate
			$$
			\mathcal G_s(t)=\frac{t}T \Bigl(\varphi_1^t(\tilde {\mathfrak X},\widetilde\Theta_2^t)-\varphi_1^t(\tilde {\mathfrak X},\widetilde\Theta_1^t)\Bigr) + \frac{T-t}T\Bigl(\varphi_2^t(\tilde {\mathfrak X},\widetilde\Theta_1^t)-\varphi_2^t(\tilde {\mathfrak X},\widetilde\Theta_2^t)\Bigr)
			$$
			s=s+1\;
		}
		Calculate the $1-\alpha$ level empirical quantile $q_{\alpha,t}$ from the sample $\mathcal G_1(t),\dots, \mathcal G_S(t)$
	\end{algorithm}

	\subsection{Algorithm of testing for a change-point}
	
	We propose the following practical testing procedure. For each value of a possible change-point $t\in \mathcal T$, we calculate the value of the test statistic $\mathcal G(t)$.  The details about estimation of the transition matrices using the SDP programs  are given in Section 3. We simulate the empirical quantiles $q_{\alpha,t}^{(1)}$ and $q_{\alpha,t}^{(2)}$ of the test statistic distribution under the null based on the first and last portions of observations. We detect a change if for some $t\in\mathcal T$ the test statistic $\mathcal G(t)$ is greater that $q_{\alpha,t}=\max(q_{\alpha,t}^{(1)},q_{\alpha,t}^{(2)})$. The details are presented in Algorithm~\ref{algo:chpdetect}.
	
	\begin{algorithm}[h!]
		\caption{VAR change-point detection}\label{algo:chpdetect}
		\KwData{$\mathfrak X=(X_1,\dots,X_T)$, $\gamma_1$, $\gamma_2$, $\mathcal T$, $h$}
		\KwResult{$\psi=1$ if there is a change-point, otherwise $\psi=0$}
		\For{$t\in \mathcal T$}
		{
			
			Estimate $\breve \Theta_1$ using $X_1,\dots,X_{\lfloor Th\rfloor}$\;
			Estimate $\breve \Theta_2$ using $X_{T-{\lfloor Th\rfloor}+1},\dots,X_T$\;
			Adjustement of estimators:
			$$
			\widehat \Theta_1= \gamma_1\frac{\breve \Theta_1}{\|\breve \Theta_1\|_{\mathrm{op}}},\quad \widehat \Theta_2= \gamma_2\frac{\breve \Theta_2}{\|\breve \Theta_2\|_{\mathrm{op}}}
			$$
			
			Generate quantile $q^{(1)}_{\alpha,t}$ using $\mathfrak X^{(1)}=(X_0,\dots,X_{\lfloor Th\rfloor})$ and $\widehat\Theta_1$\;
			Generate quantile $q^{(2)}_{\alpha,t}$ based on $\mathfrak X^{(2)}=(X_{T-{\lfloor Th\rfloor}},\dots,X_T)$ and $\widehat\Theta_2$\;
			Calculate $\widehat \Theta_1^t=\arg\min\limits_{M\in\bR^{p\times p}} \varphi_1^t(\mathfrak X,M)$ and $\widehat \Theta_2^t=\arg\min\limits_{M\in\bR^{p\times p}} \varphi_2^t(\mathfrak X,M)$\;
			Calculate the test statstic
			$$\mathcal G(t)=\frac{t}T \Bigl(\varphi_1^t(\mathfrak X,\widehat\Theta_2^t)-\varphi_1^t(\mathfrak X,\widehat\Theta_1^t)\Bigr) + \frac{T-t}T\Bigl(\varphi_2^t(\mathfrak X,\widehat\Theta_1^t)-\varphi_2^t(\mathfrak X,\widehat\Theta_2^t)\Bigr)$$

			$\psi_t=0$\;
			\If {$\mathcal G(t) >q^{(1)}_{\alpha,t}\vee  q^{(2)}_{\alpha,t}$}{$\psi_t=1$}
			t=t+1\;
		}
		$\psi=\max\limits_{t=1,\dots,T-1}\psi_t$
	\end{algorithm}
	
	\subsection{Simulation results}
	We  provide the power of the test obtained for different simulation scenarios with the quantiles simulated at the level $\alpha=0.05$. We consider the VAR processes of dimension $p=100$. 
	
	First, we consider the case of the change-point location in the middle, $\tau=T/2$ for the number of observations $T$ varying from  1500 to 5000  with $\gamma_1=\gamma_2=0.9$ and the ranks $R_1=R_2=4$. The results  in Fig.~\ref{fig:power-T} demonstrate that the detection becomes easier for a larger number of 
	observations $T$. The red curve corresponds to the case of $T=1500$, the black one stands for $T=6000$. 
	
	Next, we consider the case of $T=5000$ observations of the VAR process  with $\gamma_1=\gamma_2=0.9$ and the ranks $R_1=R_2=5$. The change-point location $\tau/T$ varies from 0.1 to 0.9. The results  in Fig.~\ref{fig:power-tau} show that the detection of the change-point is easier if it is located in the middle and harder when the change-point is close to the interval boundaries. We can also see that the power curves for $\tau/T=0.1$ and $\tau=0.9$  are almost identical and the same holds for $\tau/T=0.3$ and $\tau/T=0.7$. This simulation result is in line with the definition of the jump energy $q(\tau/T) \|\Delta\Theta\|_2$ that is symmetric with respect to $\tau=T/2$. 
	
	\begin{figure}[htbp!]
		\begin{minipage}{0.46\linewidth}
			\centering
			\includegraphics[width=\linewidth]{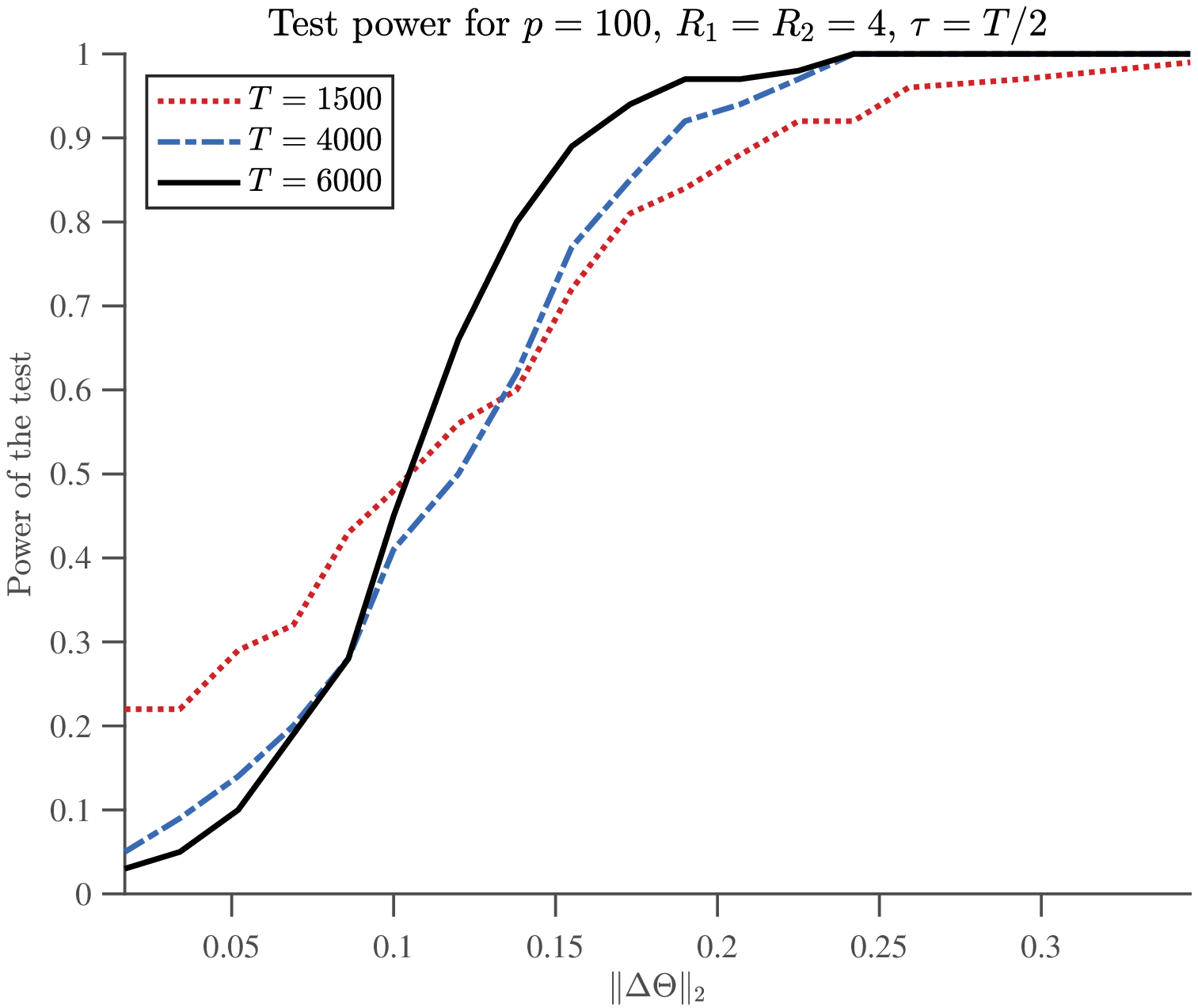}
			\caption{The test power for known location of the change-point and different number of observations $T$}
			\label{fig:power-T}
		\end{minipage}
		\hfill
		\begin{minipage}{0.46 \linewidth}
			\centering
			\includegraphics[width=\linewidth]{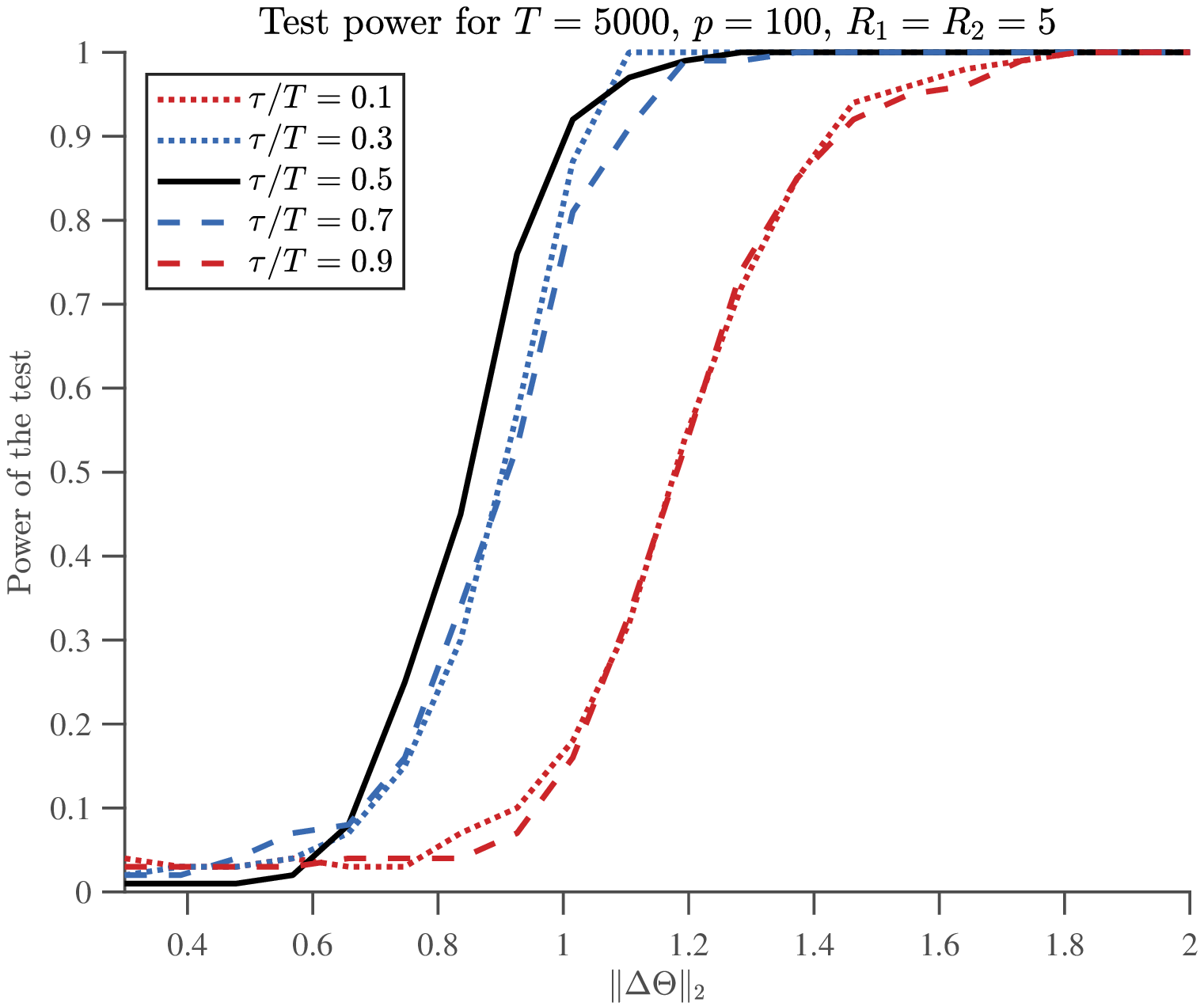}
			\caption{Dependence of the test power on different locations of the change-point}
			\label{fig:power-tau}
		\end{minipage}
	\end{figure}
	
	Finally, we consider $T=5000$ observations with fixed values $\gamma_1=\gamma_2=0.9$ and varying  matrix ranks $R_1=R_2=R\in\{5,9,17,25\}$. The results  in Fig.~\ref{fig:power-rank} demonstrate that the detection is easier for smaller rank of the matrices which is in line with the obtained detection rate of order $Rp/T$.

	\begin{figure}[htbp!]
		\centering
		\begin{minipage}{0.48\linewidth}	
			\includegraphics[width=\linewidth]{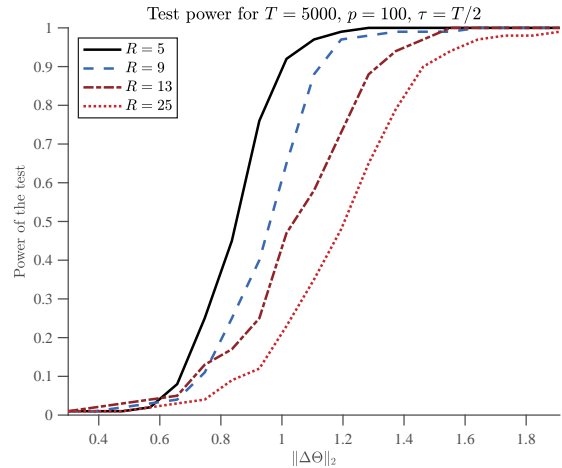}
			\caption{The test power depending on rank of the transition matrices}
			\label{fig:power-rank}
		\end{minipage}
	\end{figure}
	
\end{appendix}

\end{document}